\documentclass[11pt]{article}
\usepackage[margin=1in]{geometry}

\usepackage{rotating}
\usepackage{fancyvrb}
\usepackage{rotating}
\usepackage{fancyvrb}
\usepackage[%
backref=true,%
natbib=true,%
maxbibnames=99,%
doi=false,%
url=false,%
giveninits=true,%
style=alphabetic,%
uniquelist=false]{biblatex}
\DefineBibliographyStrings{english}{backrefpage={page}, backrefpages={pages}}
\usepackage{mymacros}

\usepackage{amsfonts,amsthm,amssymb,bigints,mathtools,empheq,braket,bm,bbm}
\usepackage{caption}
\usepackage{subcaption}
\usepackage[inline]{enumitem}
\usepackage[vlined,ruled,linesnumbered]{algorithm2e}
\usepackage{tabularx,booktabs}
\usepackage{csvsimple}
\usepackage{xcolor,tikz,pgfplots,pgf-pie}
\usepgfplotslibrary{fillbetween}
\usepackage{parskip}
\usepackage{hyperref}
\usepackage[title]{appendix}

\theoremstyle{definition} %
\newtheorem{theorem}{Theorem}
\newtheorem{proposition}{Proposition}
\newtheorem{lemma}{Lemma}
\newtheorem{corollary}{Corollary}

\theoremstyle{definition} %
\newtheorem{definition}{Definition}
\newtheorem{example}{Example}
\newtheorem{assumption}{Assumption}
\theoremstyle{remark} %
\newtheorem{remark}{Remark}

\newcommand{\PDRCCL}{\text{\textbf{P-CC}}}
\newcommand{\ODRCCL}{\text{\textbf{O-CC}}}

\newcommand{\SDC}{\text{\textbf{PP}}}

\newcommand{\WP}{\text{\textbf{HP}}}

\newcommand{\CVaR}{\text{CVaR}}
\newcommand{\VaR}{\text{VaR}}
\newcommand{\CCP}{\textbf{CC}}

\newcommand{\SOC}{\text{SOC}}

\newcommand{\Probt}{\mathbb{P}_{\text{true}}}
\newcommand{\Safe}{\mathcal{S}}
\newcommand{\Bad}{\mathcal{U}}
\newcommand{\Leb}{\mathbf{Leb}}
\newcommand{\Dist}[1]{\mathbf{d}\left( #1 \right)}

\newcommand{\disteq}{\stackrel{d}{=}}

\newcommand{\modify}[1]{{\color{black}#1}}
\newcommand{\redmodify}[1]{{\color{black}#1}}

\newenvironment{FirstUpdate}{\begingroup\color{black}}{\endgroup}
\newenvironment{SecondUpdate}{\begingroup\color{black}}{\endgroup}

\usepackage{comment}
\specialcomment{SecondUpdateRemoved}{\begingroup\color{teal}}{\endgroup}
\excludecomment{SecondUpdateRemoved}
\includecomment{SecondUpdateRemoved}

\newcommand*{\everymodeprime}{\ensuremath{^\prime}}



\title{\bf Convex Chance-Constrained Programs with Wasserstein Ambiguity}
\author{\normalsize {\bf Haoming Shen}\footnote{Department of Industrial Engineering, University of Arkansas, Fayetteville (haomings@uark.edu).}~~and {\bf Ruiwei Jiang}\footnote{Department of Industrial and Operations Engineering, University of Michigan, Ann Arbor (ruiwei@umich.edu).} \\
}
\date{ }

\begin{document}

\maketitle{}

\begin{abstract}
Chance constraints yield non-convex feasible regions in general. In particular,
when the uncertain parameters are modeled by a Wasserstein
ball,~\citep{xie-2019-distr-robus} and~\citep{chen-2018-data-driven} showed that
the distributionally robust (pessimistic) chance constraint admits a
mixed-integer conic representation. This paper identifies sufficient conditions
that lead to \emph{convex} feasible regions of chance constraints with
Wasserstein ambiguity. First, when uncertainty arises from the right-hand side
of a pessimistic \emph{joint} chance constraint, we show that the ensuing
feasible region is convex if the Wasserstein ball is centered around a
log-concave distribution (or, more generally, an \(\alpha\)-concave distribution
with \(\alpha \geq -1\)). In addition, we propose a block coordinate ascent
algorithm and prove its convergence to global optimum\redmodify{, as well as the rate of convergence}. {\color{black}Second, when
uncertainty arises from the left-hand side of a pessimistic \emph{two-sided}
chance constraint, we show the convexity if the Wasserstein ball is centered
around an elliptical and star-unimodal distribution. In addition, we propose a
family of second-order conic inner approximations, and we bound their
approximation error and prove their asymptotic exactness.} Furthermore, we
extend the convexity results to \emph{optimistic} chance constraints.%

\hfill\break%
\end{abstract}

\section{Introduction}\label{sec:intro}
Many optimization models include safety principles taking the form
\[
A(x) \; \xi \leq b(x),
\]
where \(x \in \reals^n\) represents decision variables, \(\xi \in \Xi \subseteq \reals^{q}\)
represents model parameters, and \(A(x) \in \reals^{m \times q}\) and
\(b(x) \in \reals^m\) are affine functions of \(x\). When $\xi$ is subject to
uncertainty and follows a probability distribution $\Probt$, a convenient way of
protecting these safety principles is to use chance constraint
\begin{align*}
  \Probt \Big[ A(x) \; \xi \leq b(x) \Big] \geq 1 - \epsilon, \tag{\CCP}
\end{align*}
where
\(1 - \epsilon \in (0, 1)\) represents a pre-specified risk threshold.
(\CCP) requires to satisfy all safety principles with high probability (i.e., \(1 - \epsilon\) is usually close to one, e.g., 0.95). (\CCP) was first studied in the \(1950\)s~\cite{charnes-1959-chanc-const-progr,
  charnes-1958-cost-horiz, miller-1965-chanc-const,
  prekopa-1970-on-proba} and finds a wide range of applications in, e.g.,
power system~\cite{wang-2011-chanc-const},
vehicle routing~\cite{stewart-1983-stoch-vehic-routin}, scheduling~\cite{deng-2016-decom-algor}, portfolio
management~\cite{li-1995-insur-inves}, and facility
location~\cite{miranda-2006-simul-inven}.
We mention two examples.
\begin{example}\label{exa:pp}
{\bf (Production Planning)} Suppose that we produce certain commodity at \(n\) facilities to serve \(m\) demand locations. If \(x_j\) denotes the production capacity of facility \(j\) and \(T_{ij}\) denotes the service coverage of facility \(j\) for location \(i\) (i.e., \(T_{ij} = 1\) if facility \(j\) can serve location \(i\) and \(T_{ij} = 0\) otherwise) for all \(i \in [m]\) and \(j \in [n]\), then chance constraint
\[
\Probt \big[ Tx \geq \xi \big] \geq 1 - \epsilon \tag{\SDC}
\]
assures that the production capacities are able to satisfy the demands \(\xi\) at all locations. Here, \(A(x)\) in (\CCP) equals the \(m \times m\) identify matrix and \(b(x)\) equals \(Tx\).
\end{example}

\begin{SecondUpdate}
\begin{example}\label{exa:water-plan}
\textbf{(Hydro Planning)} Over a discrete time horizon, a hydro power plant takes a random (precipitation) inflow \(\xi_i\) to its water reservoir and makes a plan to release \(x_i\) amount of water in each time unit \(i \in [t]\), in order to generate electricity and maintain the water inventory between a pre-specified lower bound \(\ell_{\text{low}}\) (dead storage) and an upper
bound \(\ell_{\text{high}}\) (flood reserve). If we denote by \(\ell_0\) the initial water inventory, then chance constraint
\begin{align*}
\Probt \left[
\ell_{\text{low}} \leq \ell_0 + \sum_{i=1}^t (\xi_i - x_i) \leq \ell_{\text{high}}
\right] \geq 1 - \epsilon \tag{\WP}
\end{align*}
assures that the reservoir maintains a safe inventory with high probability. Here, %
(\WP) admits the form of (\CCP) with \(m=2\).
\end{example}
\end{SecondUpdate}

\begin{SecondUpdateRemoved}
\end{SecondUpdateRemoved}

In~(\SDC), the random vector \(\xi\) is decoupled from the
decision variables \(x\) because, in this example, \(A(x)\) is independent of \(x\). For such chance constraints with \(A(x) \equiv A\), we follow the convention in the literature and refer to them as chance constraints with right-hand side (\RHS{}) uncertainty. {In contrast, \(\xi\) and \(x\) are coupled in~(\WP).} To distinguish chance constraints in this form from those with \RHS{} uncertainty, we call them chance constraints with left-hand side (\LHS{}) uncertainty. In addition, we say a chance constraint is \emph{individual} if \(m = 1\), \modify{\emph{two-sided} if \(m = 2\) and the two rows of \(A(x)\) are opposite (such as in~\redmodify{(\WP)})}, and \emph{joint} if \(m \geq 2\) (such as in~(\SDC)).

Although~(\CCP{}) provides an intuitive way to model uncertainty in safety principles, it produces a non-convex feasible region in general, giving rise to concerns of challenging computation. To this end, a stream of prior work proposed effective mixed-integer programming (MIP) approaches based on the notions of, e.g., sample average approximation~\cite{luedtke-2008-sampl-approx,luedtke-2008-integ-progr} and $p$-efficient points~\cite{prekopa-1990-dual-method,beraldi-2002-probab-set}, and derived valid inequalities to strengthen the ensuing MIP formulations (see, e.g.,~\cite{kuecuekyavuz-2012-mixin-sets,luedtke-2014-branc-and} and a recent survey~\cite{kuccukyavuz-2021-chance}). Another stream of prior work identified sufficient conditions for (\CCP) to produce a convex feasible region. For individual (\CCP),~\cite{panne-1963-minim-cost}
derived a second-order conic (\SOC{}) representation when \(\xi\) follows a Gaussian
distribution, and~\cite{lagoa-2001-applic}
and~\cite{calafiore-2006-distr-robus} further extended this result when \(\xi\) follows an elliptical log-concave distribution (see Definition~\ref{def:alpha-concave-dist}). \modify{Two-sided (\CCP) was first studied by~\cite{lubin-2015-two-sided}, who assumed a Gaussian \(\Probt{}\) and proved the convexity of the ensuing feasible region. Later,~\cite{fathabad-2021-tight-conic} generalized the study to a Gaussian mixture model.} For joint (\CCP) with \RHS{} uncertainty,~\cite{prekopa-2013-stoch} (see his Theorem~\(10.2\)) proved the convexity of the ensuing feasible region when \(\xi\) follows a log-concave distribution, examples of which include Gaussian, exponential, beta (if both shape parameters are at least 1), uniform on convex support, etc. Furthermore,~\cite{shapiro-2009-lectur} generalized this result to \(\alpha\)-concave distributions (see Definition~\ref{def:alpha-concave-dist}).

In most practical problems, the (true) distribution \(\Probt{}\) of the random parameters \(\xi\) is unknown or ambiguous to the modeler, who often replaces \(\Probt{}\) in (\CCP{}) with a crude estimate, denoted by \(\Prob{}\). Candidates of \(\Prob\) includes the empirical distribution based on past observations of \(\xi\) and Gaussian distribution, whose mean and covariance matrix can be estimated based on these past observations. Since \(\Prob\) may not perfectly model the uncertainty of \(\xi\), it is reasonable to take into account its neighborhood, or more formally, an ambiguity set \(\mathcal{P}\) around \(\Prob\). In this paper, we adopt a Wasserstein ambiguity set defined as
\begin{align*}
  \mathcal{P} := \Set{\QProb \in \mathcal{P}_0 \colon d_W(\QProb, \Prob) \leq \delta},%
\end{align*}
where \(\mathcal{P}_0\) is the set of all probability distributions, \(\delta > 0\) is a pre-specified radius of \(\mathcal{P}\), and \(d_W(\cdot, \cdot)\) represents the Wasserstein distance (see,
e.g.,~\cite{mohajerin-2018-data-driven,kuhn-2019-wasser-distr}). Specifically, the Wasserstein distance between two distributions \(\Prob_1\) and \(\Prob_2\) is defined through
\begin{align}
  d_W(\Prob{}_1, \Prob{}_2)
  := \inf_{\Prob_0 \sim (\Prob{}_1, \Prob{}_2)} \Expect{}_{\Prob_0} \Big[\norm{X_1 - X_2} \Big], \label{eq:intro-wass-def}
\end{align}
where \(X_1\), \(X_2\) are two random variables
following distributions \(\Prob{}_1\), \(\Prob{}_2\) respectively, \(\Prob_0\) is the coupling of \(\Prob_1\) and \(\Prob_2\), and
\(\norm{\cdot}\) represents a norm. \(d_W(\Prob{}_1, \Prob{}_2)\) can be interpreted as the minimum cost, with respect to \(\norm{\cdot}\), of transporting the probability masses of \(\Prob_1\) to recover \(\Prob_2\). Hence, the Wasserstein ambiguity set \(\mathcal{P}\) is a ball (in the space of probability distributions) centered around \(\Prob\), which for this reason is referred to as the reference distribution. Additionally, \(\mathcal{P}\) may include the true distribution \(\Probt\), i.e., \(\Probt \in \mathcal{P}\), when the radius \(\delta\) is large enough. As a result, the pessimistic counterpart
\begin{align}
\inf_{\QProb \in \mathcal{P}} \QProb \left[ A(x)\xi \leq b(x) \right] \geq 1 - \epsilon \tag{\PDRCCL{}}
\end{align}
implies (\CCP) because it requires that (\CCP) holds with respect to all distributions in \(\mathcal{P}\). In contrast, an optimistic modeler may be satisfied as long as there exists some distribution in \(\mathcal{P}\), with respect to which (\CCP) holds. This gives rise to the following optimistic counterpart of (\CCP):
\begin{align}
\sup_{\QProb \in \mathcal{P}} \QProb \left[ A(x)\xi \leq b(x) \right] \geq 1 - \epsilon. \tag{\ODRCCL{}}
\end{align}
\modify{(\ODRCCL) finds applications in portfolio management to quantify the profit opportunities in stock market~\cite{singh-2021-distr-robus}. In addition, when \(\Probt\) is ambiguous, it becomes impossible to solve an optimization model with (\CCP) directly. Nevertheless, replacing (\CCP) with (\ODRCCL) and (\PDRCCL), respectively, produces a confidence interval for the (unknown) optimal value. Besides, since (\ODRCCL) provides a relaxation of (\CCP), any valid inequality for (\ODRCCL) remains valid for computing (\CCP).}

In the existing literature, (\PDRCCL{}) is also known as distributionally robust chance constraint and, depending on the value of \(m\) and the ambiguity set \(\mathcal{P}\), the feasible region of (\PDRCCL{}) may be convex or non-convex. For individual (\PDRCCL{}) (i.e., \(m=1\)), convex representations have been derived when \(\mathcal{P}\) is Chebyshev, i.e., when \(\mathcal{P}\) consists of all distributions sharing the same mean and covariance matrix of \(\xi\). Specifically,~\cite{ghaoui-2003-worst-case,calafiore-2006-distr-robus} derived semidefinite and SOC representations of (\PDRCCL{}) with a Chebyshev \(\mathcal{P}\). With the same ambiguity set,~\cite{zymler-2011-distr-robus} showed that (\PDRCCL{}) is equivalent to its approximation based on conditional Value-at-Risk (\CVaR) even when the safety principle becomes nonlinear in \(\xi\). Additionally,~\cite{hanasusanto-2015-distr-robus} and~\cite{li-2019-ambig-risk} incorporated structural information (e.g., unimodality) into the Chebyshev \(\mathcal{P}\) and derived semidefinite and SOC representations of (\PDRCCL{}), respectively. For joint (\PDRCCL{}) (i.e., \(m \geq 2\)), however, convexity results become scarce.~\cite{hanasusanto-2017-ambig-joint} characterized \(\mathcal{P}\) by a conic support, the mean, and a positively homogeneous dispersion measure of \(\xi\), and showed that (\PDRCCL{}) with RHS uncertainty is conic representable. In addition, they showed that this result falls apart if one relaxes these conditions even in a mildest possible manner. More recently,~\cite{xie-2016-deter-refor} extended the convexity result when the safety principles depend on \(\xi\) nonlinearly and \(\mathcal{P}\) is characterized by a single moment constraint of \(\xi\). In this paper, we study (\PDRCCL{}) and (\ODRCCL{}) with \(\mathcal{P}\) being a Wasserstein ambiguity set.

To the best of our knowledge, the convexity results for either (\PDRCCL{}) or (\ODRCCL{}) with Wasserstein ambiguity do not exist in the existing literature to date. This is not surprising because~\cite{xie-2020-bicrit-approx} showed that it is strongly NP-hard to optimize over the feasible region of (\PDRCCL{}), if \(\mathcal{P}\) is centered around an empirical distribution of \(\xi\). In addition, for the same setting~\cite{xie-2019-distr-robus,chen-2018-data-driven,ji-2020-data-driven} derived mixed-integer conic representations for (\PDRCCL{}), implying a non-convex feasible region. This paper seeks to revise the choice of the reference distribution \(\Prob\), with regard to which (\PDRCCL{}) and (\ODRCCL{}) with Wasserstein ambiguity produce convex feasible regions. Our main results include:
\begin{enumerate}
\item For joint (\PDRCCL{}) with RHS uncertainty, we prove that the ensuing feasible region is convex if the reference distribution \(\Prob\) is log-concave. %
More generally, this result holds when \(\Prob\) is \(\alpha\)-concave with \(\alpha \geq -1\). Furthermore, we derive a block coordinate ascent algorithm for optimization models involving (\PDRCCL) and prove its convergence to global optimum.
\item \modify{For two-sided (\PDRCCL{}) with LHS uncertainty, we prove its convexity when the reference distribution \(\Prob\) is elliptical and star-unimodal. These conditions are tight in the sense that dropping either of them fails the convexity. Furthermore, we derive a family of second-order conic inner approximations for two-sided (\PDRCCL{}), bound their approximation error, and prove their asymptotic exactness.}
\item We extend the aforementioned convexity results for joint (\PDRCCL{}) with RHS uncertainty and \modify{two-sided} (\PDRCCL{}) with LHS uncertainty to their optimistic counterparts (\ODRCCL{}).
\end{enumerate}
In addition, we summarize the main convexity results in the following table.
\begingroup
\begin{table}[h]
    \centering
    \setlength{\tabcolsep}{18pt}
    \renewcommand{\arraystretch}{2.0}
    \begin{tabular}{ c | c c}
    \hline
     & (\PDRCCL) & (\ODRCCL) \\
    \hline
    LHS Uncertainty & Theorem~\ref{thm:ts-pcc-convexity} & Theorem~\ref{thm:ts-occ-convexity} \\
    RHS Uncertainty & Theorem~\ref{thm:drcc-rhs-cvx} & Theorem~\ref{thm:opt-rhs-cvx} \\
    \hline
    \end{tabular}
\end{table}
\endgroup

The remainder of this paper is organized as follows.
Section~\ref{sec:preliminary-results} reviews key definitions. %
Sections~\ref{sec:pessimistic-case} and~\ref{sec:algo} study convexity and solution approaches for (\PDRCCL{}), respectively. Section~\ref{sec:optimistic-drcc} extends the convexity results to (\ODRCCL{}).  Section~\ref{sec:exps} demonstrates (\PDRCCL) and (\ODRCCL) through two numerical experiments. Almost all proofs, except that for Theorem~\ref{thm:drcc-rhs-cvx}, are relegated to Appendix~\ref{apx-proofs}.

\noindent{\bf Notation}: We use \(\mathcal{X}^p\) and \(\mathcal{X}^o\) to
denote the feasible region of~(\PDRCCL{}) and~(\ODRCCL{}), respectively. We
denote the \(n\)-dimensional extended real system by \(\overline{\reals}^n\).
For a given decision \(x\), we denote by \(\Safe(x)\) the event
\(\set{\xi \colon A(x)\xi \leq b(x)}\) and by \(\Safe^c(x)\) its complement. For
\(a,b \in \reals\), \((a)^+ := \max \{a, 0\}\), \((a)^- := \min \{a, 0\}\),
\redmodify{\(a \wedge b := \min\{a, b\}\)}, and \(a \vee b := \max\{a, b\}\).
For a norm \(\norm{\cdot}\), \(\norm{\cdot}_*\) denotes its dual norm.
\(\norm{\cdot}_2\) represents the \(2\)-norm, i.e., for \(a \in \reals^n\),
\(\norm{a}_2 = \sqrt{\sum_{i=1}^n a_i^2}\).
\(I_n\)
denotes the \(n \times n\) identity matrix, \(\Leb{}(\cdot)\) denotes the
Lebesgue measure defined on the Borel \(\sigma\)-algebra of
\(\reals^{q}\), and the indicator \(\Ind{x \in \Omega}\) equals one if
\(x \in \Omega\) and zero if \(x \notin \Omega\). \modify{ For two random
variables \(X_1\) and \(X_2\), \(X_1 \disteq{} X_2\) means that \(X_1\) and
\(X_2\) are identically distributed.}

\section{Key Definitions and Examples}
\label{sec:preliminary-results}
We review definitions frequently used in subsequent discussions.

\begin{definition}%
\label{def:alpha-ccv}
A nonnegative function \(f\) defined on a convex subset of
\(\reals^n\) is said to be \(\alpha\)-concave with
\(\alpha \in \overline{\reals}\) if for all \(x, y \in \dom f\) and
\(\theta \in [0, 1]\)
\begin{align*}
f(\theta x + (1 - \theta) y ) \geq m_{\alpha}(f(x), f(y); \theta),
\end{align*}
where \(m_{\alpha} \colon \reals_+ \times \reals_+ \to \reals\) is
defined as
\begin{align*}
  m_{\alpha}(a, b; \theta) := 0 \quad \text{ if \(ab = 0\)},
\end{align*}
and if \(a > 0, b > 0, \theta \in [0, 1]\), then
\begin{align*}
  m_{\alpha}(a, b; \theta)
  :=
  \begin{cases}
  a^{\theta} b^{(1 - \theta)} & \text{ if \(\alpha = 0\),} \\
  \max\{a, b\} & \text{ if \(\alpha = +\infty\),} \\
  \min\{a, b\} & \text{ if \(\alpha = -\infty\),} \\
  (\theta a^{\alpha} + (1 - \theta) b^{\alpha})^{1 / \alpha} & \text{otherwise.}
  \end{cases}
\end{align*}
When \(\alpha = 0\) or \(\alpha = -\infty\), we say \(f\) is
log-concave or quasi-concave, respectively.
\end{definition}
The Minkowski sum of two Borel measurable subsets
\(A, B \subset \reals^n\) is Borel measurable. Let
\(\theta \in [0,1]\), then the convex combination of \(A, B\) is
defined through
\begin{align*}
  \theta A + (1 - \theta) B := \Set{\theta x + (1 - \theta) y \colon x \in A, y \in B}.
\end{align*}
\begin{definition} \label{def:alpha-concave-dist}
A probability measure \(\Prob\) defined on the Lebesgue subsets of a
convex subset \(\Omega \subseteq \reals^n\) is said to be
\(\alpha\)-concave if for any Borel measurable sets
\(A, B \subseteq \Omega\) and for all \(\theta \in [0,1]\),
\begin{align*}
  \Prob(\theta A + (1 - \theta) B) \geq m_{\alpha}(\Prob(A), \Prob(B); \theta).
\end{align*}
For a random variable \(\xi\) supported on \(\reals^n\), we
say it is \(\alpha\)-concave if the probability measure induced by \(\xi\) is \(\alpha\)-concave. In particular, \(\xi\) is log-concave if it induces a \(0\)-concave distribution.
\end{definition}
\begin{example}%
The PDF of an \(n\)-dimensional nondegenerate Gaussian is
\begin{align*}
  f(x) = \frac{1}{\sqrt{(2 \pi)^n \det(\Sigma)}} \exp \left[ - \frac{1}{2} \norm{\Sigma^{-1/2} (x - \mu)}_2 \right],
\end{align*}
where \(\mu\) and \(\Sigma\) represent its mean and
covariance, respectively. Since \(\ln f\) is concave, \(f\) is a
log-concave function and Gaussian random variables are
log-concave.
\end{example}

\begin{example}%
The PDF of a uniform distribution defined on a bounded convex subset
\(\Omega \subset \reals^n\) is
\begin{align*}
  f(x) = \frac{1}{\Leb{}(\Omega)} \Ind{x \in \Omega},
\end{align*}
where \(\Leb{}(\Omega)\) represents the volume of \(\Omega\). \(f\) is \(+\infty\)-concave on
\(\Omega\). Therefore, \(n\)-dimensional uniform distributions over a
bounded convex subset are \((1/n)\)-concave.
\end{example}

\begin{FirstUpdate}
\begin{definition}%
An \(n\)-dimensional random vector \(X\) is said to be elliptical and denoted by \(\mathcal{E}_n(\mu, \Sigma, \phi)\) if and only if there exist a
vector \(\mu \in \reals^n\), a positive semidefinite matrix
\(\Sigma \in \reals^{n \times n}\), and a function
\(\phi \colon \reals_+ \to \reals\) such that the characteristic function
\(t \mapsto \varphi_{X - \mu}(t)\) of \(X - \mu\) corresponds to
\(t \mapsto \phi(t^{\top}\Sigma t)\) for \(t \in \reals^n\).
\end{definition}
Examples of elliptical distributions include uniform distribution in a ball,
Gaussian, \(t\)-distribution, symmetric stable distribution, symmetric Laplace
distribution, logistic distribution, and Cauchy distribution.
\begin{definition}%
A set \(S \subseteq \reals^n\) is called
star-shaped if, for all \(\xi \in S\), the line segment connecting \(0\) and \(\xi\) is completely contained in \(S\). A distribution on \(\reals^n\) is called star-unimodal if it belongs to the closed convex hull of the set of uniform distributions on sets in \(\reals^n\) which are star-shaped.
\end{definition}
The above definitions of star-shapedness and star-unimodality assume that the mode is \(0\), which can be achieved without loss of generality by shifting a star-unimodal random variable by its mode. Intuitively, if a star-unimodal distribution admits a density function \(f_\xi\), then \(f_\xi(zd)\) is nonincreasing in \(z > 0\) for all \(d \in \reals^n\) and \(d \neq 0\). That is, the density function is nonincreasing along any ray emanating from the origin. Examples of star-unimodal distributions include uniform distribution in a ball, Gaussian, \(t\)-distribution, logistic distribution, and Cauchy distribution.

We review properties of \(\alpha\)-concave functions, \(\alpha\)-concave probability measures, as well as elliptical and star-unimodal distributions in Appendix~\ref{apx-preliminary}.
\end{FirstUpdate}

\section{Pessimistic Chance Constraint}
\label{sec:pessimistic-case}

We first review the definitions of value-at-risk (\VaR{}) and
\CVaR{}~\cite{rockafellar-1999-optim-condit}, as well as the~\CVaR{}
reformulation of \(\mathcal{X}^p\) derived by~\cite{xie-2019-distr-robus}. Then, we derive a new reformulation of \(\mathcal{X}^p\) for \(\alpha\)-concave reference distribution \(\Prob\). The new reformulation leads to convexity proofs for joint (\PDRCCL{}) with RHS uncertainty and  two-sided (\PDRCCL{}) with LHS uncertainty in Sections~\ref{sec:pess-joint-drcc} and~\ref{sec:pess-two-sided}, respectively.

\begin{definition}
Let \(X\) be a
random variable, inducing probability distribution
\(\Prob{}_{X}\). The \((1 - \epsilon)\)-\VaR{} of
\(X\) is defined through
\begin{gather*}
\VaR_{1 - \epsilon} (X) := \inf
\Set{x : \Prob{}_{X} \left[ X \leq x \right] \geq 1 - \epsilon},
\end{gather*}
and its \((1 - \epsilon)\)-\CVaR{} is defined through
\begin{align*}
\CVaR_{1 - \epsilon} (X) :=
  \min_{\gamma} \Set{\gamma + \frac{1}{\epsilon} \Expect{} \Big[ \posp{X - \gamma} \Big] }.
\end{align*}
\end{definition}

\begin{proposition}[Adapted from Theorem~\(1\) in~\cite{xie-2019-distr-robus}]\label{prop:pess-cvar-reform-1}
For \(\delta > 0\), it holds that
\begin{gather}
\mathcal{X}^p = \Set{ x \in \reals^n \colon \frac{\delta}{\epsilon} +
  \CVaR_{1 - \epsilon} \Big( - \Dist{{\zeta}, \Safe^c(x)} \Big)
  \leq 0 }. \label{eq:deter-reform-xie}
\end{gather}
Here, random variable \({\zeta}\) follows the reference distribution \(\Prob\) and
\(\Dist{\zeta, \Safe^c(x)}\) represents the distance from \(\zeta\) to
the ``unsafe'' set \(\Safe^c(x)\)~\cite{chen-2018-data-driven},
\begin{align*}
  \Dist{\zeta, \Safe^c(x)}
  & := \inf_{\xi \in \Xi} \Set{\norm{ \zeta - \xi } \colon A(x)\xi \nleq b(x)},
\end{align*}
and \(\Xi\) is the support of \(\xi\).
\end{proposition}
For all \(x \in \mathcal{X}^p\), it holds that
\begin{equation*}
a_i(x) = 0 \ \Rightarrow \ b_i(x) \geq 0 \quad \forall i \in [m],
\end{equation*}
where \(a_i(x)^{\top}\) represents row \(i\) of matrix \(A(x)\) and \(b_i(x)\) represents entry \(i\) of vector \(b(x)\), because otherwise \(\Prob[A(x)\zeta \leq b(x)] = 0\) and \(x \notin \mathcal{X}^p\). Assuming the above implication without loss of generality, we define function \(f: \reals^n \times \reals^m \rightarrow \reals\),
\begin{align*}
f(x, \zeta) := \min_{i \in [m] \setminus I(x)} \left\{ \frac{b_i(x) - a_i(x)^{\top} \zeta}{\|a_i(x)\|_*} \right\},
\end{align*}
where \(I(x):=\{i \in [m]: a_i(x) = 0\}\). Then, it follows from~\cite{xie-2019-distr-robus,chen-2018-data-driven} that
\begin{align*}
\Dist{\zeta, \Safe^c(x)} = \Big( f(x, \zeta) \Big)^+.
\end{align*}
In what follows, we derive new reformulations of \(\mathcal{X}^p\) based on \(f(x, \zeta)\). To this end, we need the following lemma to relate the the \CVaR{} of \(f(x, \zeta)\) to that of
\(-\Dist{\zeta, \Safe^c(x)}\) in~\eqref{eq:deter-reform-xie}.
\begin{lemma}\label{lem:reform-neg-cvar}
Let \(X\) be a random variable, then
\begin{align*}
  \CVaR_{1 - \epsilon} ( X^- )
  = \Ind{0 \geq \VaR_{1 - \epsilon} (X) } \cdot \left[
  \CVaR_{1 - \epsilon} (X)  - \frac{1}{\epsilon} \Expect [X^+]\right].
\end{align*}
\end{lemma}
Combining Proposition~\ref{prop:pess-cvar-reform-1} and
Lemma~\ref{lem:reform-neg-cvar} leads to the following reformulation of \(\mathcal{X}^p\).
\begin{corollary}\label{cor:pess-cvar-reform-2}
For \(\delta > 0\), it holds that
\begin{empheq}[left={\mathcal{X}^p = \empheqlbrace{} x \in \reals^n \colon}, right=\empheqrbrace]{gather}
\frac{\delta}{\epsilon} + \CVaR_{1 - \epsilon} \left( -f(x, \zeta) \right)
\leq \frac{1}{\epsilon} \Expect \left[ \left( -f(x, \zeta) \right)^+ \right] \label{eq:cor-pess-cvar-reform-1}\\
0 \geq \VaR_{1 - \epsilon} \left( -f(x, \zeta) \right) \label{eq:cor-pess-cvar-reform-2}
\end{empheq}
\end{corollary}
In this paper, we focus on cases in which \(\Prob\) is \(\alpha\)-concave. The next lemma shows that an \(\alpha\)-concave \(\Prob\) yields atomless \(\Dist{\zeta, \Safe^c(x)}\) and \(f(x, \zeta)\), which lead to a further reformulation of \(\mathcal{X}^p\).
\begin{lemma}\label{lem:cont-of-dist-to-set}
If the reference distribution \(\Prob\) is \(\alpha\)-concave, then for all \(x\), \(\Prob \big[ \Dist{\zeta, \Safe^c(x)} = y \big] = 0\) for all \(y > 0\) and \(\Prob \big[ f(x, \zeta) = y \big] = 0\) for all \(y \in \reals\).
\end{lemma}
We are now ready to present the new reformulations of \(\mathcal{X}^p\).
\begin{proposition}\label{prop:reform-w-cc}
Suppose that \(\Prob\) is \(\alpha\)-concave. Then, for \(\delta > 0\), it holds that
\begin{empheq}[left={\mathcal{X}^p = \empheqlbrace{} x \in \reals^n \colon}, right=\empheqrbrace]{gather}
\delta \leq \int_0^{\VaR_{\epsilon}\left( f(x, \zeta) \right)}
\Big(\Prob \left[  f(x, \zeta) \geq t \right] - ( 1 - \epsilon)\Big) \dLeb{t} \label{eq:deter-reform-int-1}\\
\Prob\left[A(x)\zeta \leq b(x)\right] \geq 1 - \epsilon
\label{eq:deter-reform-int-2}
\end{empheq}
\end{proposition}

\begin{remark}
We notice that constraint~\eqref{eq:deter-reform-int-2} is simply (\CCP) with respect to the reference distribution \(\Prob\) of the Wasserstein ball \(\mathcal{P}\). In addition, constraint~\eqref{eq:deter-reform-int-1} encodes a robust guarantee. Intuitively, the RHS of~\eqref{eq:deter-reform-int-1} evaluates the budget needed to shift the probability masses of \(\Prob\) so that the corresponding (\CCP) can be violated. Constraint~\eqref{eq:deter-reform-int-1} makes sure that this budget is beyond the radius of \(\mathcal{P}\), i.e., (\CCP) will not be violated as long as the shifted distribution lies within \(\mathcal{P}\).
\end{remark}

\subsection{Joint (\PDRCCL{}) with \RHS{} Uncertainty}
\label{sec:pess-joint-drcc}

For (\CCP) with~\RHS{} uncertainty, it is well
celebrated that the ensuing feasible region is
convex when \(\xi\) has an \(\alpha\)-concave distribution (particularly, \(\xi\) is log-concave when \(\alpha=0\))~\cite{prekopa-2013-stoch,shapiro-2009-lectur}.
\begin{proposition}[Theorem~\(4.39\) and Corollary~\(4.41\) in~\cite{shapiro-2009-lectur}]\label{prop:cc-rhs-constr-log-concave}
If \({\xi} \in \reals^m\) follows an \(\alpha\)-concave probability distribution, then \(H(x) := \Prob \left[ A\xi \leq b(x) \right]\) is \(\alpha\)-concave on the
set \(\mathcal{D} := \{x \in \reals^n \colon \exists \; \xi \text{ such}\) \(\text{that } A\xi \leq b(x)\}\) and the following set is convex and closed:
\begin{align*}
  \mathcal{X} := \Set{x \in \reals^n \colon \Prob \left[A\xi \leq b(x) \right] \geq 1 - \epsilon}.
\end{align*}
\end{proposition}

In this subsection, we seek to extend this result to (\PDRCCL{}).
\begin{theorem}\label{thm:drcc-rhs-cvx}
Suppose that the reference distribution \(\Prob\) of \(\mathcal{P}\) is
\(\alpha\)-concave with \(\alpha \geq -1\). Then, for
\(\delta > 0\) the set
\begin{align*}
  \mathcal{X}^p_{\text{R}} := \Set{x \in \reals^n \colon \inf_{\QProb \in \mathcal{P}}
  \QProb \left[ A\xi \leq b(x) \right]
  \geq 1 - \epsilon} %
\end{align*}
is convex and closed.
\end{theorem}
\modify{
Although Theorem~\ref{thm:drcc-rhs-cvx} pertains to (\PDRCCL) with linear inequalities, the convexity result extends to (\PDRCCL) with \emph{quasi-concave} inequalities. We present a detailed description and a proof for this generalization in Appendix~\ref{apx-thm:drcc-rhs-cvx}.} Before presenting a proof of Theorem~\ref{thm:drcc-rhs-cvx}, we present some useful lemmas. Without loss of generality, we assume that each row of matrix
\(A\), denoted by \(a_i^{\top}\) for all \(i \in [m]\), satisfies
\begin{enumerate}[label=(\roman*)]
\item \(a_i \neq 0\), because otherwise we can add a deterministic constraint \(b_i(x) \geq 0\) to \(\mathcal{X}^p_{\text{R}}\) and eliminate inequality \(i\) from (\PDRCCL{});
\item \(\norm{a_i}_* = 1\), because otherwise we can divide both sides of inequality \(i\) by \(\norm{a_i}_*\) and set \(a_i \leftarrow a_i/\norm{a_i}_*\), \(b_i(x) \leftarrow b_i(x)/\norm{a_i}_*\).
\end{enumerate}
Recall that for \(\zeta \in \reals^m\) the distance \(\Dist{\zeta, \Safe^c(x)}\) to the unsafe set satisfies \(\Dist{\zeta, \Safe^c(x)} = \big(f(x, \zeta)\big)^+\) with
\begin{align*}
f(x, \zeta) = \min_{i \in [m]} \Big\{b_i(x) - a_i^{\top}\zeta\Big\}
\end{align*}
and \(f(x, \zeta)\) is jointly concave in \((x, \zeta)\).
\begin{lemma}\label{lem:rhs-var-concave}
For all \(\epsilon \in (0, 1)\), if \({\zeta}\) has \modify{an} \(\alpha\)-concave distribution with \(\alpha \geq -1\), then
\(\VaR_{1 - \epsilon}\left(f(x, {\zeta})\right)\) is concave in \(x\) on \(\reals^n\).
\end{lemma}

\begin{lemma}\label{lem:cont-of-prob-fcns}
Suppose that \(f(\cdot, \cdot) \colon \reals^n \times \Xi \to \reals\) is a
continuous function, \({\zeta}\) follows an \(\alpha\)-concave distribution \(\Prob\), and
\(f(x, {\zeta})\) is atomless for any \(x \in \reals^n\). Then,
\begin{gather*}
  \psi(x, t) := \Prob \left[ f(x, {\zeta}) \geq t \right] - (1 - \epsilon)
  \qquad \text{ and } \qquad
  \phi(x, y) := \int_0^y \psi(x, t) \dLeb{t}
\end{gather*}
are both continuous on \(\reals^n \times \reals_+\).
\end{lemma}

Now we are ready to prove Theorem~\ref{thm:drcc-rhs-cvx}.

\begin{proof}[Proof of Theorem~\ref{thm:drcc-rhs-cvx}]
First, recall that by Proposition~\ref{prop:reform-w-cc} we recast \(\mathcal{X}^p_{\text{R}}\) as constraints~\eqref{eq:deter-reform-int-1}--\eqref{eq:deter-reform-int-2}. For ease of exposition, we denote by \(G(x)\) the RHS of~\eqref{eq:deter-reform-int-1}.

Second, to show that \(\mathcal{X}^p_{\text{R}}\) is closed, it suffices to prove the closedness of the feasible region of~\eqref{eq:deter-reform-int-1} because that of~\eqref{eq:deter-reform-int-2} follows from Proposition~\ref{prop:cc-rhs-constr-log-concave}. To this end, we notice that \(\VaR_{\epsilon}\big(f(x, \zeta)\big)\) is continuous in \(x\) due to its concavity. Then, by Lemma~\ref{lem:cont-of-prob-fcns} the mapping
\begin{align*}
  x \mapsto \int_0^{\VaR_{\epsilon}(f(x, {\zeta}))} \Prob \Big[ f(x, {\zeta}) \geq t \Big] \dLeb{t}
\end{align*}
is continuous. It follows that \(G(x)\) is continuous and the feasible region of~\eqref{eq:deter-reform-int-1} is closed.

Third, to show that \(\mathcal{X}^p_{\text{R}}\) is convex, it suffices to prove the convexity of the feasible region of~\eqref{eq:deter-reform-int-1} because that of~\eqref{eq:deter-reform-int-2} follows from Proposition~\ref{prop:cc-rhs-constr-log-concave}. To that end,
by Proposition~\ref{prop:cc-rhs-constr-log-concave} and
Lemma~\ref{lem:v-shift-of-alpha-cve}, \(\psi\) is \(\alpha\)-concave in \((x, t)\) on \(\dom \psi := \Set{(x, t) \colon \psi(x, t) \geq 0} = \Set{(x, t) \colon t \leq \VaR_{\epsilon}(f(x, \zeta))}\), which is convex by Lemma~\ref{lem:rhs-var-concave}. Then, for any \(x_0, x_1 \in \mathcal{X}^p_{\text{R}}\) and any \(t_0 \in S_0 := [0, \VaR_{\epsilon} ( f(x_0, {\zeta}))]\),
\(t_1 \in S_1 := [0, \VaR_{\epsilon} ( f(x_1, {\zeta}))]\) it holds that
\begin{align*}
  \psi(x_{1/2}, t_{1/2})
  \geq m_{\alpha} \left[ \psi(x_0, t_0), \psi(x_1, t_1); \frac{1}{2} \right],
\end{align*}
where \(x_{1/2} = (x_0 + x_1)/2\) and \(t_{1/2} = (t_0 + t_1)/2\). It follows that
\begin{align*}
m_{\alpha^{\ast}_1} \left[ \int_{S_0} \psi(x_0, t) \dLeb{t}, \int_{S_1} \psi(x_1, t) \dLeb{t}; \frac{1}{2}\right] \leq & \int_{\frac{1}{2}S_0 + \frac{1}{2}S_1} \psi(x_{1/2}, t) \dLeb{t} \\
\leq & \int_{S_{1/2}} \psi(x_{1/2}, t) \dLeb{t}
\end{align*}
where the first inequality is due to  Proposition~\ref{prop:brunn-minkow-ineq} and \(\alpha^{\ast}_1 \geq -\infty\) is a function of \(\alpha\) (see Proposition~\ref{prop:brunn-minkow-ineq}), and the second inequality is because \(\frac{1}{2}S_0 + \frac{1}{2}S_1 \subseteq S_{1/2} := [0, \VaR_{\epsilon}(f(x_{1/2}, \zeta))]\). In other words, we obtain that
\[
m_{\alpha^{\ast}_1}\left[G(x_0), G(x_1); \frac{1}{2}\right] \leq G\left(x_{1/2}\right)
\]
and \(G(x)\) is midpoint \(\alpha^{\ast}_1\)-concave, and particularly, midpoint quasi-concave. Then, its  continuity implies that \(G(x)\) is quasi-concave and constraint~\eqref{eq:deter-reform-int-1} yields a convex feasible region. This finishes the proof.
\end{proof}

\begin{FirstUpdate}
We close this section by commenting on the worst-case distribution with respect to (\PDRCCL). For expectation-oriented optimization, it has been observed
that if the Wasserstein ball is centered around a Gaussian reference distribution, then the worst-case probability distribution is also Gaussian~\cite{kuhn-2019-wasser-distr}. In contrast, the following example demonstrates that this is not the case for (\PDRCCL). We present a detailed proof for this example in Appendix~\ref{apx-exa:not-gaussian}.
\begin{example} \label{exa:not-gaussian}
Consider the (\PDRCCL)
\begin{align*}
\inf_{\Prob \in \mathcal{P}} \Prob \left[ \xi \leq x \right] \geq 1 - \epsilon,
\tag{Ex}
\end{align*}
where \(\epsilon \in (0, 1/2)\), and the Wasserstein ball \(\mathcal{P}\) is centered around the 1-dimensional standard Gaussian distribution and has a radius \(\delta > 0\). Then, there does not exist a Gaussian distribution $\nu \in \mathcal{P}$ such that $\nu\left[ \xi \leq x \right] = \inf_{\Prob \in \mathcal{P}} \Prob \left[ \xi \leq x \right]$.
\end{example}
\end{FirstUpdate}

\begin{FirstUpdate}
\subsection{Two-Sided~(\PDRCCL{}) with \LHS{} Uncertainty}%
\label{sec:pess-two-sided}
We move on to two-sided (\PDRCCL) with \LHS{} uncertainty, defined through
\begin{align*}
  \mathcal{X}^p_{\text{T}} := \Set{(x, \ell, u) \in \reals^{n+2} \colon
  \inf_{\QProb \in \mathcal{P}} \QProb \left[ \ell \leq x^{\top} \xi \leq u \right] \geq 1 - \epsilon
  }.
\end{align*}
To study the convexity of \(\mathcal{X}^p_{\text{T}}\), we make the following two assumptions about the Wasserstein ball.
\begin{assumption}\label{asm:ellip-ref-and-norm}
The Wasserstein ball \(\mathcal{P}\) is such that
\begin{enumerate*}[label=(\roman*)]
\item the reference distribution \(\Prob\) is elliptical, particularly \(\mathcal{E}_n(0, \Sigma, \phi)\)
with \(\Sigma \succ 0\). %
\item the norm \(\norm{\cdot}\) in \(d_W\) is an ellipsoidal norm with regard to
\(\Sigma^{1/2}\), \ie{}, \(\norm{\cdot} = \norm{\Sigma^{-1/2}(\cdot)}_2\) (or
equivalently, \(\norm{\cdot}_{\ast} = \norm{\Sigma^{1/2}(\cdot)}_2\)).
\end{enumerate*}
\end{assumption}
\begin{assumption}\label{asm:unimodal-ref-w-diff-density}
The reference distribution \(\Prob\) of \(\mathcal{P}\) is star-unimodal.
\end{assumption}
Examples of \(\Prob\) satisfying both Assumptions~\ref{asm:ellip-ref-and-norm}--\ref{asm:unimodal-ref-w-diff-density} include uniform distribution in a ball, Gaussian, \(t\)-distribution, logistic distribution, and Cauchy distribution. We remark that Assumption~\ref{asm:unimodal-ref-w-diff-density} implies that the mode of \(\Prob\) is the origin, which can be achieved without loss of generality by shifting \(\xi\) by its mode, if different from \(0\). In addition, since \(\Prob\) is elliptical, it is identically distributed as \(R\cdot \Sigma^{1/2}U_n\) for a nonnegative random variable \(R\) and an \(n\)-dimensional random vector \(U_n\) uniformly distributed on the unit sphere \(S^{n-1}\) of \(\reals^n\) (see Remark~\ref{rmk:normal-of-ellip-dist} in Appendix~\ref{apx-preliminary}). We denote by \(\Prob_0\) the probability measure induced by \(R\cdot e_1^{\top} U_n\) and establish the convexity of \(\mathcal{X}^p_{\text{T}}\) as follows.
\begin{theorem}\label{thm:ts-pcc-convexity}
Suppose that Assumptions~\ref{asm:ellip-ref-and-norm} and~\ref{asm:unimodal-ref-w-diff-density} hold, \(\epsilon \in (0, \frac{1}{2})\), and \(\delta > 0\). Define
\begin{gather*}
g_{\epsilon}(\ell, u) := \int_0^{+\infty} \Big[
    \Phi(u - t) - \Phi(\ell + t) - (1 - \epsilon) \Big]^+ \dLeb{t} \\
  \text{and} \qquad
\mathcal{C}_{\delta} := \Set{
  (\ell, u) \in \reals^2 \colon
  \delta \leq g_{\epsilon}(\ell, u)
  },
\end{gather*}
where \(\Phi \colon \reals \to [0, 1]\) denotes the cumulative distribution function of \(\Prob_0\). Then,
\begin{gather*}
  \mathcal{X}^p_{\text{T}} = \Set{(x, \ell, u) \in \reals^{n+2}\colon
  \exists \, s \geq 0 \text{ such that } \norm{x}_{\ast} \leq s, \ (\ell, u, s) \in \text{co}(\mathcal{C}_\delta)},
\end{gather*}
where \(\text{co}(\mathcal{C}_\delta) := \cl{} \left(\Set{(\ell, u, s) \in \reals^3 \colon s > 0,
(\ell/s, u/s) \in \mathcal{C}_\delta} \right)\) is the cone induced by \(\mathcal{C}_\delta\) and \(\cl{}(\cdot)\) denotes the closure operator. Furthermore, \(\mathcal{X}^p_{\text{T}}\) is convex and closed.
\end{theorem}
\begin{figure}[!htbp]
\centering
\includegraphics[width=0.4\linewidth]{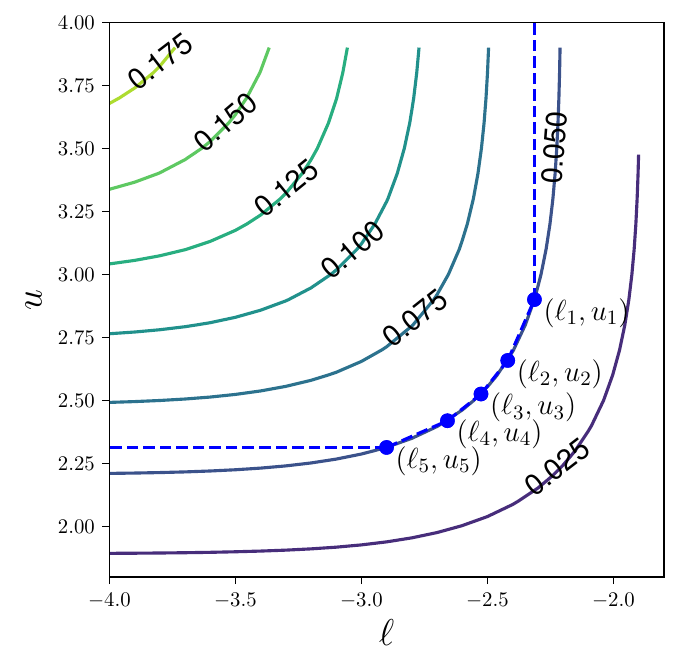}
\caption{Contours of \(g_{\epsilon}(\ell, u)\) with varying \(\delta\) and a polyhedral inner approximation \(\hat{\mathcal{C}}_N\) of \(\mathcal{C}_\delta\) with \(N = 5\) and \(\delta = 0.050\)}
\label{fig:Z0-approx}
\end{figure}
We visualize the function \(g_\epsilon(\ell, u)\) and set \(\mathcal{C}_\delta\). Figure~\ref{fig:Z0-approx} depicts the contour of \(g_\epsilon(\ell, u)\), which is symmetric with respect to the line \(\ell + u = 0\). This is because \(\Prob_0\) is elliptical and so
\(\Phi(u) - \Phi(\ell)\) is symmetric with respect to \(\ell + u = 0\) on \(\reals_- \times \reals_+\), \ie{},
\begin{align*}
  \Phi(u) - \Phi(\ell)
  = \Phi(-\ell) - \Phi(-u),
  \quad \forall (\ell, u) \in \reals_- \times \reals_+, %
\end{align*}
which implies that
\begin{align*}
g_{\epsilon}(\ell, u) = g_{\epsilon}(-u, -\ell), \quad \forall (\ell, u) \in \reals_- \times \reals_+.
\end{align*}
In addition, from the same figure, we observe that the superlevel set \(\mathcal{C}_\delta\) of \(g_{\epsilon}(\ell, u)\) takes a convex shape, which explains (intuitively) the convexity of \(\mathcal{X}^p_{\text{T}}\). %

Assumptions~\ref{asm:ellip-ref-and-norm}--\ref{asm:unimodal-ref-w-diff-density} are not only sufficient for the convexity of \(\mathcal{X}^p_{\text{T}}\), but also tight in the sense that Theorem~\ref{thm:ts-pcc-convexity} ceases to hold for the lack of either assumption. We demonstrate through the following two examples.
\begin{example}[Loss of Convexity Without Assumption~\ref{asm:unimodal-ref-w-diff-density}] \label{exa:counter-1}
Consider an example of \(\mathcal{X}^p_{\text{T}}\) with \(n=1\) and the reference distribution \(\Prob\) of the Wasserstein ball is identical to that of a random variable \(\zeta_1 := R_c \cdot U_1\), where \(R_c\) and
\(U_1\) are independent random variables following the arcsine distribution and the
uniform distribution on \(\set{-1, 1}\), respectively. Hence, \(\Prob\) satisfies Assumption~\ref{asm:ellip-ref-and-norm} but violates Assumption~\ref{asm:unimodal-ref-w-diff-density} (see Figure~\ref{fig:graph-of-zeta1-cdf}). By
Proposition~\ref{prop:repr-of-elli-dist}, we have
\begin{align*}
  \mathcal{X}^p_{\text{T}} \cap \set{(x, \ell, u) \in \reals^3 \colon x = 1}
  & = \Set{(1, \ell, u) \in \reals^3 \colon
    \inf_{\QProb \in \mathcal{P}} \QProb \left[ \ell \leq \xi \leq u \right] \geq (1 - \epsilon)} \\
  & = \Set{(1, \ell, u) \in \reals^3 \colon g_{\epsilon}(\ell, u) \geq \delta }.
\end{align*}
In addition, the cumulative distribution function of \(\zeta_1\) satisfies
\begin{align*}
  \Phi(t)
  & = \frac{1}{2} \Prob \left[ R_c \leq t \right] + \frac{1}{2} \Prob \left[ -R_c \leq t \right]
    = \frac{1}{2} (F_{R_c}(t) + 1) \cdot \Ind{t \geq 0} + \frac{1}{2} (1 - F_{R_c}(-t)) \cdot \Ind{t < 0},
\end{align*}
where \(F_{R_c}(\cdot)\) denotes the distribution function of \(R_c\). Next, we
show that, when restricted to the line segment
\(L_a := \set{(\ell, u) \in (-1, 0) \times (0, 1) \colon u - \ell = a}\) with
\(a \in (1, 2)\), \(\Phi(u) - \Phi(\ell)\) is a strictly convex function, which fails the midpoint concavity of the function \(g_{\epsilon}(\ell, u)\) and shows that \(\mathcal{X}^p_{\text{T}} \cap \set{(x, \ell, u) \in \reals^3 \colon x = 1}\) is non-convex.

To this end, we define a function \(\Psi_a: [a-1, 1] \mapsto \reals_+\) with \(\Psi_a(u) := \Phi(u) - \Phi(u-a)\).
A simple calculation shows that \(\Psi^{\prime\prime}_a(u) = \frac{1}{4\pi} \left(f_{\Psi^{\prime\prime}_a}(u) + f_{\Psi^{\prime\prime}_a}(a-u) \right)\), where \(f_{\Psi^{\prime\prime}_a}(u) := (2u-1)/(u(1-u))^{3/2}\).
Because \(f_{\Psi^{\prime\prime}_a}\) is strictly increasing on
\((0, 1)\) (see Figure~\ref{fig:graph-of-fPsi-hess}), we have
\begin{figure}[!htbp]
\centering
\begin{subfigure}{.48\textwidth}
\centering
\centerline{\includegraphics[width=0.8\linewidth]{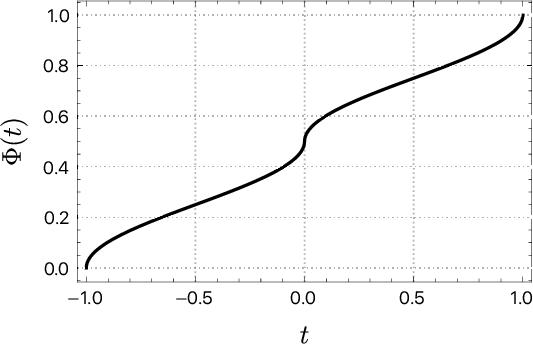}}
\caption{\label{fig:graph-of-zeta1-cdf} Distribution function of \(\zeta_1\)}
\end{subfigure}
\hfill
\begin{subfigure}{.48\textwidth}
\centering
\centerline{\includegraphics[width=0.8\linewidth]{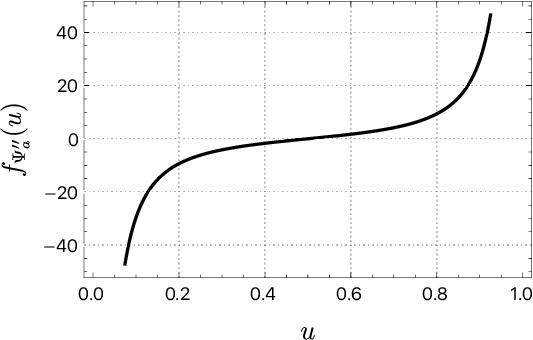}}
\caption{\label{fig:graph-of-fPsi-hess} Graph of \(f_{\Psi^{\prime\prime}_a}(u)\) on \((0,1)\)}
\end{subfigure}\\
\vspace{0.5em}
\begin{subfigure}{.48\textwidth}
\centering
\centerline{\includegraphics[width=0.7\linewidth]{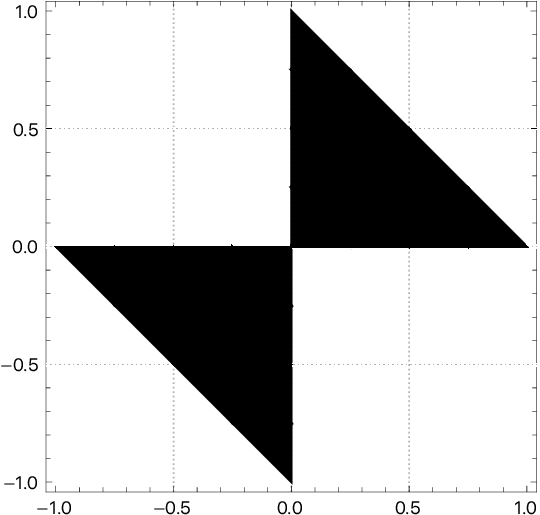}}
\caption{\label{fig:graph-of-star-shaped-unif} Density of \(\zeta_2\).}
\end{subfigure}
\begin{subfigure}{.48\textwidth}
\centering
\centerline{\includegraphics[width=0.8\linewidth]{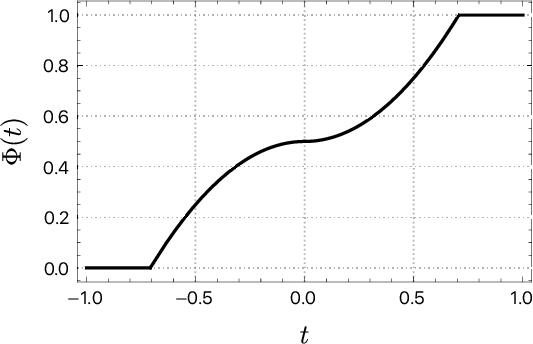}}
\caption{\label{fig:graph-of-zeta2-cdf} Density of \(\zeta^{\top}_2 [1; 1] / \sqrt{2}\).}
\end{subfigure}
\caption{Visualization of random variables \(\zeta_1\) and \(\zeta_2\) in Examples~\ref{exa:counter-1}--\ref{exa:counter-2}}
\end{figure}
\begin{align*}
  \Psi^{\prime\prime}_a(u)
  > \frac{1}{4\pi} \left(f_{\Psi^{\prime\prime}_a}(u) + f_{\Psi^{\prime\prime}_a}(1-u)\right) = 0, \quad \forall u \in (0, 1)
\end{align*}
implying that \(\Psi_a(u)\) is strictly convex. To see how this fails the midpoint concavity of \(g_{\epsilon}(\ell, u)\), we define \(u_1 := a - 1\) and \(u_2 := 1\). Then, for any \(a \in (1, 2)\), we have
\begin{align*}
  g_{\epsilon}\left(-\frac{a}{2}, \frac{a}{2}\right)
  & = \int\limits_0^{+\infty} \left[ \Phi\left(\frac{a}{2} - t\right) - \Phi\left(- \frac{a}{2} +t\right) - (1 - \epsilon) \right]^+ \dLeb{t} \\
  & = \int\limits_0^{+\infty} \left[ \Psi_{a - 2t} \left( \frac{a}{2} - t \right) - (1 - \epsilon) \right]^+ \dLeb{t}
    = \int\limits_0^{+\infty} \left[ \Psi_{a - 2t} \left( \frac{1}{2} (u_1 + u_2) - t \right) - (1 - \epsilon) \right]^+ \dLeb{t} \\
  & < \frac{1}{2} \int\limits_0^{+\infty} \left[ \Psi_{a - 2t} \left( u_1 - t \right) - (1 - \epsilon) \right]^+ \dLeb{t}
    + \frac{1}{2} \int\limits_0^{+\infty} \left[ \Psi_{a - 2t} \left( u_2 - t \right) - (1 - \epsilon) \right]^+ \dLeb{t} \\
  & = \frac{1}{2} \left[ g_{\epsilon}(u_1 - a, u_1) + g_{\epsilon}(u_2 - a, u_2) \right] = g_{\epsilon}(-1, a-1),
\end{align*}
where the inequality is due to the strict convexity of \(\Psi_a\),
and the last equality uses the symmetry of \(g_{\epsilon}\). Let
\(\delta_{\epsilon, a} := g_{\epsilon}(-1, a_1)\), then we see that
\((-1, a-1)\) and \((1-a, 1)\) are in the \(\delta_{\epsilon, a}\)-superlevel
set of \(g_{\epsilon}\), while their midpoint, \((-a/2, a/2)\), falls out of the
\(\delta_{\epsilon,a}\)-superlevel set of \(g_{\epsilon}\). This implies that \(\mathcal{X}^p_{\text{T}}\), and particularly the intersection \(\mathcal{X}^p_{\text{T}} \cap \set{(x, \ell, u) \in \reals^3 \colon x = 1}\), are non-convex.
\end{example}

\begin{example}[Loss of Convexity Without Assumption~\ref{asm:ellip-ref-and-norm}] \label{exa:counter-2}
Consider an example of \(\mathcal{X}^p_{\text{T}}\) with \(n=2\) and the reference distribution \(\Prob\) is a uniform distribution on the star-shaped set depicted in Figure~\ref{fig:graph-of-star-shaped-unif}, where we let \(\zeta_2 \in \reals^2\) denote the random vector following distribution \(\Prob\). Then, \(\Prob\) satisfies Assumption~\ref{asm:unimodal-ref-w-diff-density} but violates Assumption~\ref{asm:ellip-ref-and-norm}. Proposition~\ref{prop:repr-of-elli-dist} yields
\begin{align*}
  & \mathcal{X}^p_{\text{T}} \cap \set{(x, \ell, u) \in \reals^4 \colon x^{\top} = [1/\sqrt{2}, 1/\sqrt{2}]} \\
  = \
  & \Set{(1/\sqrt{2}, 1/\sqrt{2}, \ell, u) \in \reals^4 \colon
    \inf_{\QProb \in \mathcal{P}} \QProb \left[ \ell \leq \xi^{\top}[1; \ 1]/\sqrt{2} \leq u \right] \geq (1 - \epsilon)} \\
  = \
  & \Set{(1/\sqrt{2}, 1/\sqrt{2}, \ell, u) \in \reals^4 \colon g_{\epsilon}(\ell, u) \geq \delta }.
\end{align*}
In addition, the cumulative distribution function of \(\zeta_2^{\top}[1; \ 1]/\sqrt{2}\) is
\begin{align*}
  \Phi(t)
  & = \left( \frac{1}{2} - t^2 \right) \cdot \Ind{-\frac{1}{\sqrt{2}} \leq t \leq 0}
    + \left( \frac{1}{2} + t^2 \right) \cdot \Ind{0 \leq t \leq \frac{1}{\sqrt{2}} },
\end{align*}
whose graph is depicted in Figure~\ref{fig:graph-of-zeta2-cdf}. Since \(\Phi(t)\)
is strictly convex on \([0,1/\sqrt{2}]\), \(\Phi(u) - \Phi(\ell)\) is strictly
convex on the set \(\set{(\ell, u) \in [-1/\sqrt{2}, 0] \times [0, 1/\sqrt{2}]}\), and so is
its restriction to the line segment
\(\set{(\ell, u) \in (-1, 0) \times (0, 1) \colon u - \ell = a}\). Hence, following a similar argument as in  Example~\ref{exa:counter-1}, we can show that \(\mathcal{X}^p_{\text{T}}\), and particularly the intersection \(\mathcal{X}^p_{\text{T}} \cap \set{(x, \ell, u) \in \reals^4 \colon x^{\top} = [1/\sqrt{2}, 1/\sqrt{2}]}\), are non-convex.
\end{example}

We end this section by mentioning two special cases of the two-sided (\PDRCCL) with \LHS{} uncertainty that often arise in practice and admit second-order conic representations. The first case considers symmetric bounds, i.e., \(\ell = -u\) for \(u \geq 0\).
\begin{corollary}\label{cor:ts-pcc-symm-bds}
Suppose that Assumption~\ref{asm:ellip-ref-and-norm} holds,
\(\epsilon \in (0, \frac{1}{2})\), and \(\delta > 0\). Define
\begin{gather*}
  \mathcal{X}^p_{\text{TS}} := \Set{(x, u) \in \reals^n \times \reals_+ \colon
  \inf_{\QProb \in \mathcal{P}} \QProb \left[ -u \leq x^{\top} \xi \leq u \right] \geq 1 - \epsilon
} \\
\text{and} \qquad g^s_{\epsilon}(r) := \int\limits_0^{+\infty} \Big[ \Phi(r - t) - \Phi(-r + t) - (1 - \epsilon) \Big]^+ \dLeb{t}.
\end{gather*}
Then, it holds that
\begin{align*}
  \mathcal{X}^p_{\text{TS}} = \Set{
  (x, u) \in \reals^n \times \reals_+ \colon
  u \geq \norm{x}_{\ast} \cdot \inf_r \set{r \geq 0 \colon g^s_{\epsilon}(r) \geq \delta}
  }.
\end{align*}
\end{corollary}

The second case considers individual (\PDRCCL), which can be obtained by driving \(\ell\) to \(-\infty\) in the two-sided (\PDRCCL). The proof of this case is similar to that of Corollary~\ref{cor:ts-pcc-symm-bds} and so omitted.
\begin{corollary}\label{cor:from-ts-to-single}
Suppose that Assumption~\ref{asm:ellip-ref-and-norm} holds, \(\epsilon \in (0, \frac{1}{2})\), and \(\delta > 0\). Define
\begin{gather*}
  \mathcal{X}^p_{\text{I}} := \Set{(x, u) \in \reals^n \times \reals_+ \colon
  \inf_{\QProb \in \mathcal{P}} \QProb \left[ x^{\top} \xi \leq u \right] \geq 1 - \epsilon
}, \\
\overline{g}_{\epsilon}(r)
  := \int\limits_0^{+\infty} \Big[
  \Phi(r - t) - (1 - \epsilon) \Big]^+ \dLeb{t}, \quad \text{and} \quad (\overline{g})^{-1}_{\epsilon}(s) := \inf_{r \geq 0} \set{r \colon \overline{g}_{\epsilon}(r) \geq s}.
\end{gather*}
Then, it holds that
\begin{align*}
  \mathcal{X}^p_{\text{I}} = \Set{
  (x, u) \in \reals^n \times \reals_+ \colon
  u \geq \norm{x}_{\ast} \cdot (\overline{g})^{-1}_{\epsilon}(\delta)
  }.
\end{align*}
\end{corollary}
The same representation in Corollary~\ref{cor:from-ts-to-single} has been derived in~\cite{chen-2021-sharin-value}. In fact, using integration by part, one can verify that the coefficient \((\overline{g})^{-1}_{\epsilon}(\delta)\) in the above representation equals the \(\eta^{\ast}\) in Theorem~\(4.8\) of~\cite{chen-2021-sharin-value}.

\end{FirstUpdate}

\section{Solution Approaches for Pessimistic Chance Constraint}%
\label{sec:algo}
The convexity results in the previous section inspires us to study solution approaches for solving (\PDRCCL) based on convex/continuous optimization. We study a block coordinate ascent algorithm for (\PDRCCL) with \RHS{} uncertainty in Section~\ref{sec:coordinate-ascent} and a second-order conic inner approximation approach for two-sided (\PDRCCL) with \LHS{} uncertainty in Section~\ref{sec:soc-inner-approx}.

\subsection{Block Coordinate Ascent Algorithm} \label{sec:coordinate-ascent}
We focus on a model with joint (\PDRCCL) and RHS uncertainty: \(\min_{x \in X} \{c^{\top}x: (\text{\PDRCCL})\}\), where \(c \in \reals^n\) represents cost coefficients and \(X \subseteq \reals^n\) represents a set that is deterministic, compact, and convex. By Proposition~\ref{prop:reform-w-cc}, this model is equivalent to
\begin{subequations}
\begin{align}
\min_{x \in X} \ & \ c^{\top}x \label{eq:algo-ref-0} \\
\text{s.t.} \ & \ \delta \leq \int_0^{\VaR_{\epsilon}(f(x, \zeta))} \Big(\Prob\big[f(x, \zeta) \geq t\big] - (1 - \epsilon)\Big) \dLeb{t}, \label{eq:algo-ref-1} \\
& \ \VaR_{\epsilon}\big(f(x, \zeta)\big) \geq 0, \label{eq:algo-ref-2}
\end{align}
where \(f(x, \zeta) = \min_{i \in [m]}\{b_i(x) - a_i^{\top}\zeta\}\). Here, constraint~\eqref{eq:algo-ref-1} appears challenging because its RHS involves an integral with upper limit \(\VaR_{\epsilon}(f(x, \zeta))\). To make the model computable, we define a new variable \(y \geq 0\) to represent \(\VaR_{\epsilon}(f(x, \zeta))\).
\begin{proposition}\label{prop:algo-reform}
For \(y \geq 0\), define
\begin{align*}
  \phi(x, y) := \int^y_0 \left(\Prob \Big[ f(x, {\zeta}) \geq t \Big] - (1 - \epsilon) \right) \dLeb{t}.
\end{align*}
If \(\Prob{}\) is \(\alpha\)-concave with \(\alpha \geq -1\), then \(\phi(x, y)\) is
\(\alpha^{\ast}_1\)-concave on
\[\dom \phi := \Set{(x, y) \in X \times \reals_+ \colon \Prob
  \left[ f(x, {\zeta}) \geq y \right] \geq (1 -
  \epsilon)},\]
where \(\alpha^*_1\) is defined in Proposition~\ref{prop:brunn-minkow-ineq}. In addition, \(\dom \phi\) is closed and constraints~\eqref{eq:algo-ref-1}--\eqref{eq:algo-ref-2} is equivalent to
\begin{align}
\delta \leq \max_{y \geq 0} \phi(x, y). \label{eq:algo-ref-3}
\end{align}
\end{proposition}
\end{subequations}

By Proposition~\ref{prop:algo-reform}, formulation~\eqref{eq:algo-ref-0}--\eqref{eq:algo-ref-2} is equivalent to \(\min_{x \in X}\{c^{\top}x: \text{\eqref{eq:algo-ref-3}}\}\). To address the integral arising from the RHS of constraint~\eqref{eq:algo-ref-3}, we switch the objective function with the constraint to obtain
\begin{align}
  \rho(u) := \sup_{x \in X, y \geq 0} \Set{
  \phi(x, y) \colon c^{\top} x \leq u}, \label{eq:rho-u}
\end{align}
where \(u\) represents a budget limit on the (original) objective function.
\modify{We notice that \(\rho(u)\) is non-decreasing in \(u\), and hence \(u^* \in \mathbb{R}\) is the optimal value of~\eqref{eq:algo-ref-0}--\eqref{eq:algo-ref-2} if and only if \(u^*\) is the smallest number such that \(\rho(u^*)\) exceeds \(\delta\). It follows that we can solve~\eqref{eq:algo-ref-0}--\eqref{eq:algo-ref-2} by searching for the intersection of the function \(\rho(u)\) with the constant \(\delta\), which can be done by a bisection line search on \(u\) and iteratively solving~\eqref{eq:rho-u} to evaluate \(\rho(u)\).}
In addition, \(\rho(u)\) may be interesting in its own right because it represents the largest Wasserstein radius \(\delta\) that allows us to find a solution \(x\) that satisfies (\PDRCCL) and incurs a cost no more than \(u\). Hence, the graph of \(\rho(u)\) depicts a risk envelope that interprets the trade-off between the robustness and the cost effectiveness of (\PDRCCL). We demonstrate the risk envelope numerically in Section~\ref{sec:exps-pess-joint-drcc}.

Evaluating \(\rho(u)\) is equivalent to maximizing \(\phi(x, y)\) over the intersection of \(\Set{(x,y) \in X \times \reals_+: c^{\top}x \leq u}\) and \(\dom\phi\). Unfortunately, projecting onto \(\dom\phi\) may be inefficient since it is the feasible region of (\CCP). To avoid projection, we propose a block coordinate ascent algorithm (Algorithm~\ref{algo:bcm-for-rho-u}; see, e.g.,~\cite{auslender-1976-optim,luo-1993-error-bound, bertsekas-2016-nonli-progr,grippof-1999-global-conver,beck-2013-conver-block,beck-2015-conver-alter}). This algorithm iteratively maximizes over \(y\) with \(x\) fixed and then maximizes over \(x\) with \(y\) fixed. Here, for fixed \(x\) with \(\Prob[A\zeta \leq b(x)] \geq 1 - \epsilon\), i.e., when \(x\) satisfies (\CCP), the maximization over \(y\) admits a closed-form solution \(y =  \VaR_{\epsilon} \big( f(x, {\zeta}) \big)\), that is,
\begin{align*}
  \max_{y \geq 0} \phi(x, y) = \phi\left(x, \VaR_{\epsilon} \Big( f(x, {\zeta}) \Big) \right) %
\end{align*}
because \(\phi(x, y)\) is increasing in \(y\) on the interval \(\big[0, \VaR_{\epsilon} \big( f(x, {\zeta}) \big)\big]\) and it becomes decreasing in \(y\) when \(y > \VaR_{\epsilon} \big( f(x, {\zeta}) \big)\). On the other hand, for fixed \(y\), we seek to maximize \(\phi(x, y)\), which appears challenging as it is an integral. Fortunately, we can recast \(\phi(x, y)\) as
\begin{align}
  \phi(x, y)
  & = \int_0^y \Prob \left[f(x, {\zeta}) \geq t \right]  \dLeb{t} - y \cdot (1 - \epsilon) \nonumber{}\\
  & = y \int_0^1 \Prob \left[ f(x, {\zeta}) \geq sy  \right] \dLeb{s} - y \cdot (1 - \epsilon) \nonumber{}\\
  & = y \int\displaylimits_{\Xi} \int\displaylimits_{\reals} \Ind{
    (\zeta, s) \colon f(x, {\zeta}) \geq sy} \cdot \IndSet{[0, 1]}{s}
    \dLeb{s} \dProb{\zeta} - y \cdot (1 - \epsilon), \nonumber{} \\
  & = y \int\displaylimits_{\Xi \times \reals} \Ind{
    (\zeta, s) \colon f(x, {\zeta}) \geq sy} \dMeas{\hat{\mathbb{P}}}{\zeta, s} - y \cdot (1 - \epsilon), \nonumber{}\\
  & = y \cdot \hat{\Prob} \left[ f(x, {\zeta}) \geq sy \right] - y \cdot (1 - \epsilon), \label{eq:reform-phi}
\end{align}
where the third equality is by Tonelli's theorem and \(\hat{\Prob}\) represents the product measure of \(\Prob\) and the uniform distribution on \([0, 1]\). Since these two distributions are log-concave on \(\Xi\) and \([0, 1]\), respectively, \(\hat{\Prob}\) is log-concave on \(\Xi \times [0, 1]\). As a result, the problem simplifies to the P-model of (\CCP)
with respect to a log-concave distribution, which has been well studied in~\cite{prekopa-2013-stoch,norkin-1993-analy-optim}. As a result, Algorithm~\ref{algo:bcm-for-rho-u} uses the existing solution approach as a building block and assumes that there exists an oracle, denoted by \(\mathcal{O}_u(y, \varepsilon)\), which for given \(y\) and \(\varepsilon > 0\) returns an
\(\varepsilon\)-optimal solution \(\hat{x} \in  \Set{x \in X: c^{\top}x \leq u}\) such that
\begin{gather*}
\hat{\Prob} \left[ f(\hat{x}, {\zeta}) \geq sy \right] \geq \max_{x \in X: \ c^{\top}x \leq u} \Set{
\hat{\Prob} \left[  f(x, {\zeta}) \geq sy \right] } - \varepsilon.
\end{gather*}

We are now ready to present Algorithm~\ref{algo:bcm-for-rho-u}.
\begin{algorithm}[!htbp]
  \caption{Evaluation of \(\rho(u)\)}\label{algo:bcm-for-rho-u}
  \SetKwInOut{Inputs}{Inputs}
  \SetKw{Return}{return}
  \Inputs{budget \(u\), risk level \(\epsilon\), oracle \(\mathcal{O}_u\), a diminishing sequence \(\set{\varepsilon_k}_k\), and an \(x_1\) such that \(y_1 := \VaR_{\epsilon}\big(f(x_1, \zeta)\big) > 0\).}
  \For{\(k = 1, 2, \ldots\)}{
    \(x_{k+1} \gets \mathcal{O}_u(y_k, \varepsilon_k)\)\;
    \(y_{k+1} \gets \VaR_{\epsilon} \big( f(x_{k+1}, {\zeta}) \big)\)\;
    \If{stopping criterion is satisfied}{
      \Return{\(\phi(x_{k+1}, y_{k+1})\).}
    }
  }
\end{algorithm}

Algorithm~\ref{algo:bcm-for-rho-u} needs an starting point \((x_1, y_1)\) such that
\(\VaR_{\epsilon}(f(x_1, {\zeta})) > 0\). This can be obtained by solving a (\CCP) feasibility problem,
\begin{align}
  \min_{x \in X} \Set{0 \colon \Prob \big[ f(x, {\zeta}) \geq \varepsilon_0 \big] \geq 1 - \epsilon, \ c^{\top} x \leq u}, \label{eq:algo-find-init-sol}
\end{align}
where \(\varepsilon_0\) is a small positive constant. If formulation~\eqref{eq:algo-find-init-sol} is infeasible for all \(\varepsilon > 0\), then \(\rho(u) = 0\) because \(\VaR_{\epsilon} \big( f(x, {\zeta}) \big)\) always remains non-positive. Numerically, one can solve~\eqref{eq:algo-find-init-sol} for a sequence
of diminishing \(\varepsilon_0\)'s to find a valid starting point. We close this section by showing that Algorithm~\ref{algo:bcm-for-rho-u} achieves global optimum.
\begin{theorem}\label{thm:algo-cvx-converge}
Let \(\set{(x_k, y_k)}_k\) be an infinite sequence of iterates produced by Algorithm~\ref{algo:bcm-for-rho-u}. Suppose that \(\Prob\) is log-concave and, for all \(k \geq 2\), \(x_k\) and \(y_k\) are \(\varepsilon_k\)-optimal, i.e.,
\begin{gather*}
  \max_{x}\phi(x, y_{k-1}) - \varepsilon_k \leq \phi(x_k, y_{k - 1}) \leq \max_{x}\phi(x, y_{k-1})
  \quad \text{ and } \quad
  \abs[\Big]{y_k - \VaR_{\epsilon} \big( f(x_k, {\zeta}) \big)} \leq \varepsilon_k
\end{gather*}
with \(\lim_{k \to \infty}\varepsilon_k = 0\). Then, any limit point of \(\set{(x_k, y_k)}_k\) is a global optimal solution to~\eqref{eq:rho-u}.
\end{theorem}

\begin{remark} \label{rmk:convergence-cvx}
{\color{black} Block alternating minimization/maximization algorithms have been applied to improve conservative approximations of (distributionally robust) chance-constrained programs. For example,~\cite{zymler-2011-distr-robus} applied them to improve the CVaR approximation of a moment~(\PDRCCL{}),~\cite{chen-2010-from-cvar} applied them to improve an order statistics approximation of the same CVaR approximation, and~\cite{jiang2022also} applied them to improve a hinge-loss approximation of chance-constrained programs.} The convergence of block alternating minimization algorithms for convex programs~\cite{auslender-1976-optim,luo-1993-error-bound, bertsekas-2016-nonli-progr,
grippof-1999-global-conver,beck-2013-conver-block,
beck-2015-conver-alter} requires stronger
sense of convexity~\cite{auslender-1976-optim, luo-1993-error-bound}, {\color{black}continuous differentiability and}
a unique minimizer with respect to each
block~\cite{bertsekas-2016-nonli-progr}, or a Lipschitz
gradient~\cite{beck-2013-conver-block, beck-2015-conver-alter}. {\color{black}The convergence of Algorithm~\ref{algo:bcm-for-rho-u} does not follow from these existing results because, for fixed \(y\), \(\phi(x,y)\) may not even be differentiable. We make this concrete through the following Example~\ref{exa:counter-3}.}
\end{remark}

{\color{black}
\begin{example}[\(\phi(x,y)\) may not be differentiable in \(x\) for fixed \(y\)]\label{exa:counter-3}
Consider an example of \(\phi\) with \(\mathbb{P}\) being a uniform distribution on the interval \(\Xi := [-2, 2]\) and \(f(x, \zeta)\) is defined on \(\reals_+ \times \Xi\) as
\begin{align*}
f(x, \zeta) := \min \Set{\zeta + 2, \zeta + x, 2 - \zeta, x - \zeta}.
\end{align*}
Fix \(y = 1\), and we simplify \(\phi\) as follows:
\begin{align*}
\phi(x, 1) + (1 - \epsilon)
& = \int_0^1 \Prob \left[ f(x, \zeta) \geq t \right] \dLeb{t}
= \int_0^1 \Prob \Big[ t + (-x) \vee (-2) \leq \zeta \leq (x \wedge 2) - t \Big] \dLeb{t} \\
& = 2 \int_0^1 \Big[ (x - t) \wedge (2 - t) \Big]^+ \dLeb{t} \\
& = 2 \left( \Ind{0 \leq x \leq 2} \int_0^1 (x - t)^+ \dLeb{t} + \Ind{x > 2} \int_0^1 (2 - t) \dLeb{t} \right) \\
& = 2 \left( \Ind{0 \leq x \leq 1} \frac{1}{2} x^2 + \Ind{1 < x \leq 2} \left(x - \frac{1}{2}\right) + \Ind{x > 2} \left(2 - \frac{1}{2}\right) \right).
\end{align*}
Then, the left and right derivatives of \(\phi(x, 1) + (1 - \epsilon)\) at
\(x = 2\) are \(1/2\) and \(0\), respectively. Therefore, \(\phi(x, 1)\) is not differentiable.
\end{example}
Under additional (mild) assumptions, one can show that, for fixed \(x\), \(\phi(x, y)\) is continuously differentiable in \(y\) and the gradient is Lipschitz. Following this, we establish the linear convergence of Algorithm~\ref{algo:bcm-for-rho-u}. In other words, it takes \(O(1/\varepsilon)\) iterations for the algorithm to achieve an \(\varepsilon\)-optimal solution to (\PDRCCL). We refer the interested readers to Appendix~\ref{apx-sec:rate-of-converge}.
}

\begin{FirstUpdate}
\subsection{\SOC{} Inner Approximation}\label{sec:soc-inner-approx}
We focus on the two-sided (\PDRCCL) with \LHS{} uncertainty and its feasible region \(\mathcal{X}^p_{\text{T}}\). Although Theorem~\ref{thm:ts-pcc-convexity} provides a convex representation of \(\mathcal{X}^p_{\text{T}}\), it is not computable (say, in a commercial solver) because the function
\(g_{\epsilon}(\ell, u)\) is defined through an integral. This section derives \SOC{} inner approximations of \(\mathcal{X}^p_{\text{T}}\), which can be directly and efficiently computed by commercial solvers.

\subsubsection{Inner Approximations for \(\mathcal{C}_\delta\) and \(\mathcal{X}^p_\text{T}\)}
\label{sec:poly-approx-Z0}
To illustrate the basic idea, we plot the boundary of \(\mathcal{C}_\delta\), \ie{}, the contour of the function \(g_\epsilon(\ell, u)\), in Figure~\ref{fig:Z0-approx}. Since \(\mathcal{C}_\delta\) is convex, we can obtain a polyhedral inner approximation using two extreme rays of \(\mathcal{C}_\delta\) and a set of points on its boundary, denoted by \(\bd{}(\mathcal{C}_{\delta})\) (see the dotted line in Figure~\ref{fig:Z0-approx} for an illustration of this inner approximation). We now formalize this idea.
\begin{definition} \label{def:inner}
Given \(N\) points \( \set{(\ell_1, u_1), \ldots, (\ell_N, u_N)}\) on \(\bd{}(C_{\delta})\) with \(\ell_1 > \ell_2 > \cdots > \ell_N\), define a polyhedron
\begin{subequations}
\begin{empheq}[left={\hat{\mathcal{C}}_N = \empheqlbrace{} (\ell, u) \in \reals_- \times \reals_+ \colon}, right={\empheqrbrace{}}]{align}
  & \ell \leq \ell_1 \label{eq:approx-ZN0-p1}\\
  & (u - u_i) (\ell_i - \ell_{i + 1}) \geq (u_i - u_{i+1}) (\ell - \ell_i), \quad \forall i \in [N-1] \label{eq:approx-ZN0-pi} \\
  & u \geq u_N. \label{eq:approx-ZN0-pN}
\end{empheq}
\end{subequations}
\end{definition}
In this definition, inequality~\eqref{eq:approx-ZN0-p1}
(resp.~\eqref{eq:approx-ZN0-pN}) is the vertical (resp.~horizontal) ray emitting
from \((\ell_1, u_1)\) (resp.~\((\ell_N, u_N)\)) and inequalities~\eqref{eq:approx-ZN0-pi} are the line segments connecting two neighboring points \((\ell_i, u_i)\) and \((\ell_{i+1}, u_{i+1})\). Then, \(\hat{\mathcal{C}}_N\) constructs an inner approximation for \(\mathcal{X}^p_{\text{T}}\).
\begin{algorithm}[tbp]
  \caption{Construction of \(\hat{\mathcal{C}}_N\)}\label{algo:seeds-for-C-hat}
  \SetKwInOut{Inputs}{Inputs}
  \SetKw{Return}{return}
  \Inputs{\(\epsilon \in (0, \frac{1}{2}), \delta > 0\), %
    and a (small) error threshold \(\tau > 0\).}
  Initialize the set of points \(\texttt{PT} = \emptyset.\) \\
  Find a \(u \gets (\overline{g}_{\epsilon})^{-1}(\delta + \tau)\) and an \(\ell\) such that \(g_{\epsilon}(\ell, u) = \delta\). \\
  {\bf if} \(\ell + u > 0\) {\bf then} Replace \((\ell, u) \gets (-u, -\ell)\). \\
  \(\texttt{PT} \gets \texttt{PT} \cup \set{(\ell, u)}\). \\
  \While{\(\ell + u \leq 0\)}{%
    Find an \((\ell', u')\) on \(\bd{}(\mathcal{C}_{\delta})\) such that
    \begin{enumerate}
    \item \((\ell', u') \geq (\ell, u)\), and
    \item the line connecting \((\ell, u)\) and
    \((\ell', u')\) supports \(\mathcal{C}_{\delta + \tau}\).
    \end{enumerate}
    \(\texttt{PT} \gets \texttt{PT} \cup \set{(\ell', u')}.\) \\
    \((\ell, u) \gets (\ell', u')\).
  }
  \For{\((\ell, u)\) in \textup{\texttt{PT}}}{%
    \(\texttt{PT} \gets \texttt{PT} \cup \set{(-u, -\ell)}\).}
  Sort and label all points in \texttt{PT} from \(1\) to \(N\) such that \(\ell_1 > \ell_2 > \cdots > \ell_N\). \\
  \Return{\textup{\texttt{PT}}}
\end{algorithm}
\begin{proposition} \label{prop:soc-inner}
Given \(\hat{\mathcal{C}}_N\) in Definition~\ref{def:inner}, it holds that \(\hat{\mathcal{X}}^p_{\text{T}} \subseteq \mathcal{X}^p_{\text{T}}\), where
\begin{align*}
  \hat{\mathcal{X}}^p_{\text{T}} := \Set{ (x, \ell, u) \in \reals^{n+2} \colon
  \begin{aligned}
  & \exists s \in \reals: \norm{x}_{\ast} \leq s \\
  & \ell - x^{\top} \mu \leq \ell_1 s \\
  & u - x^{\top} \mu \geq u_N s \\
  & \left(\frac{\ell_i - \ell_{i+1}}{u_i - u_{i+1}}\right) \big(u - x^{\top} \mu - u_i s\big)
  \geq \ell - x^{\top} \mu - \ell_i s, \quad \forall i \in [N - 1]
  \end{aligned}
  }.
\end{align*}
\end{proposition}
We highlight that \(\hat{\mathcal{X}}^p_{\text{T}}\) is \SOC{} and so can be efficiently computed by commercial solvers. In view that \(\mathcal{C}_\delta\) and \(g_{\epsilon}(\ell, u)\) are symmetric with respect to \(\ell + u = 0\), we also construct \(\hat{\mathcal{C}}_N\) to be symmetric; that is, we pick the \(N\) points such that \(\ell_i + u_{N-i+1} = 0\) for all \(i \in [N]\). We propose Algorithm~\ref{algo:seeds-for-C-hat} to find the points and construct a symmetric \(\hat{\mathcal{C}}_N\). Specifically, Algorithm~\ref{algo:seeds-for-C-hat} receives a (small) error threshold \(\tau > 0\) and returns a set of \(N\) points on \(\bd{}(\mathcal{C}_{\delta})\) such
that the ensuing inner polyhedral approximation \(\hat{\mathcal{C}}_N\) satisfies \(g_{\epsilon}(\ell, u) \in [\delta, \delta + \tau]\) for all \((\ell, u) \in \bd{}(\hat{\mathcal{C}}_N)\). In other words,
\begin{align*}
\mathcal{C}_{\delta} \subseteq \hat{\mathcal{C}}_N \subseteq \mathcal{C}_{\delta + \tau}.
\end{align*}
In Step 2 of Algorithm~\ref{algo:seeds-for-C-hat}, we search for \(u = (\overline{g}_{\epsilon})^{-1}(\delta + \tau)\), which is defined in Corollary~\ref{cor:from-ts-to-single} and visualized in Figure~\ref{fig:g_inverse}. This can be done by running a root-finding algorithm on the function \(\overline{g}_{\epsilon}(\cdot)\), e.g., a bisection line search. The same approach can be applied in Step 2 to find an \(\ell\) such that \(g_{\epsilon}(\ell, u) = \delta\), and in Step 6 to find an \((\ell', u')\) such that the line connecting \((\ell, u)\) and \((\ell', u')\) supports \(\mathcal{C}_{\delta + \tau}\). Since the domains of functions \(\overline{g}_{\epsilon}(\cdot)\) and \(g_{\epsilon}(\cdot, \cdot)\) are \(1\)-dimensional and \(2\)-dimensional, respectively, running a bisection line search on them is efficient. As a result, the runtime of Algorithm~\ref{algo:seeds-for-C-hat} is usually negligible, even when we choose a small error threshold \(\tau\) (see Section~\ref{sec:exps-hydro-runtime} for a numerical demonstration).

\subsubsection{Approximation Error and Asymptotic Exactness of \
\(\hat{\mathcal{X}}^p_\text{T}\)} \label{sec:approx-bound}

We show that \(\hat{\mathcal{X}}^p_\text{T}\) is asymptotically exact, \ie{}, \(\hat{\mathcal{X}}^p_\text{T}\) asymptotically recovers \(\mathcal{X}^p_\text{T}\), as the error threshold \(\tau\) in Algorithm~\ref{algo:seeds-for-C-hat} decreases to zero and the points \(\{(\ell_i, u_i)\}_{i=1}^N\) become dense. More generally, we quantify the error of \(\hat{\mathcal{X}}^p_\text{T}\) in approximating \(\mathcal{X}^p_\text{T}\) with a positive \(\tau\).

To this end, we study the error of \(\hat{\mathcal{C}}_N\) in approximating \(\mathcal{C}_\delta\) and the same approximation guarantee carries over to \(\hat{\mathcal{X}}^p_\text{T}\) by construction. We shall show that \(\mathcal{C}_\delta\) becomes a subset of \(\hat{\mathcal{C}}_N\) if we slightly expand the latter. But since \(0 \notin \mathcal{C}_\delta\), we need to define such expansion with respect to a new origin within \(\mathcal{C}_\delta\), giving rise to the next definition.
\begin{definition}
Given a set \(C \subseteq \reals^2\), a point \((\ell_0, u_0) \in C\), and a positive scalar \(\gamma > 0\), define
\[
\gamma \cdot C := (\ell_0, u_0) + \Big\{\gamma(\ell - \ell_0, u - u_0): \ (\ell, u) \in C\Big\}.
\]
\end{definition}
The new origin \((\ell_0, u_0)\) partitions \(\reals^2\) into four (shifted) orthants \(\set{\mathcal{O}_i, i \in [4]}\), and hence \(\mathcal{C}_{\delta}\) and
\(\hat{\mathcal{C}}_N\) into four disjoint subsets:
\begin{align*}
  \mathcal{C}_{\delta}
  = \bigcup_{i=1}^4 \Big(\mathcal{C}_{\delta} \cap \mathcal{O}_i\Big), \quad
  \hat{\mathcal{C}}_N
  = \bigcup_{i=1}^4 \Big(\hat{\mathcal{C}}_N \cap \mathcal{O}_i\Big).
\end{align*}
Our approach is to find constants \(\gamma_i\) such that
\(\mathcal{C}_{\delta} \cap \mathcal{O}_i \subseteq \gamma_i \cdot (\hat{\mathcal{C}}_N \cap \mathcal{O}_i)\) for all \(i \in [4]\). Then, \(\max_{i \in [4]} \{\gamma_i\}\) gives the approximation error of \(\hat{\mathcal{C}}_N\). We present the main result of this section as follows.
\begin{theorem}\label{cor:approx-guarantee-for-hierarchy}
Under Assumptions~\ref{asm:ellip-ref-and-norm} and~\ref{asm:unimodal-ref-w-diff-density}, suppose that \(\epsilon \in (0, 1/2)\), \(\Prob_0\) has a CDF \(\Phi\) and a density function \(\Phi^{\prime}\), and
\(\hat{\mathcal{C}}_{N}\) is obtained from Algorithm~\ref{algo:seeds-for-C-hat} with an error threshold \(\tau > 0\). Then, for
\((\ell_0, u_0) := (\ell_N, -\ell_N)\) it holds that
\begin{gather*}
\hat{\mathcal{C}}_N \subseteq \mathcal{C}_{\delta}
\subseteq \gamma_{\tau} \cdot \hat{\mathcal{C}}_N, \\[1em]
\text{where} \qquad{}
\gamma_{\tau} :=
\max \Set{
  \frac{u_0 - (\overline{g}_{\epsilon})^{-1}(\delta)}{u_0 - (\overline{g}_{\epsilon})^{-1}(\delta + \tau)}, \quad
  1 + \frac{\sqrt{2}\ln{}\big((\delta + \tau) / \delta\big)}{\underline{D} \cdot \underline{\rho}}
  }.
\end{gather*}
Here, \(\underline{D} := \Phi^{\prime}(u_0)/\big(\Phi(u_0) - \Phi(\ell_0) - (1 - \epsilon)\big)\) and
\(\underline{\rho} := 1/\norm{(1, 0)}_{\mathcal{C}_{\delta+\tau}}\), where \(\norm{\cdot}_{\mathcal{C}_{\delta+\tau}}\) is defined through \(\norm{x}_{\mathcal{C}_{\delta+\tau}} := \inf \set{r > 0 \colon x \in r \cdot \mathcal{C}_{\delta+\tau}}\). Furthermore, \(\gamma_{\tau} \to 1\) as \(\tau \to 0\).
\end{theorem}
\begin{figure}[!htbp]
\centering
\begin{subfigure}[t]{0.4\textwidth}
\centering
\resizebox{1.0\linewidth}{!}{\begin{tikzpicture}

\begin{axis}[
tick align=outside,
tick pos=left,
xlabel={\(\displaystyle \tau / \delta\)},
xmin=2.75e-06, xmax=5.225e-05,
xtick style={color=black},
xtick={0,1e-05,2e-05,3e-05,4e-05,5e-05,6e-05},
xticklabels={0,1,2,3,4,5,6},
ylabel={\(\displaystyle \gamma_{\tau}\)},
ymin=1.0, ymax=1.4,
ytick style={color=black}
]
\addplot [semithick,mark=*]
table {%
5e-06 1.02685752404044
1e-05 1.05373799293458
1.5e-05 1.0806414400817
2e-05 1.10755077575162
2.5e-05 1.13449597967031
3e-05 1.16146425463388
3.5e-05 1.18845563428059
4e-05 1.21543579253735
4.5e-05 1.24246915266955
5e-05 1.26952571093449
};
\end{axis}

\end{tikzpicture}}
\caption{Approximation error bound \(\gamma_{\tau}\)}\label{fig:approx-bd-gamma}
\end{subfigure}
\hfill
\begin{subfigure}[b]{0.4\textwidth}
\centering
\resizebox{1.01\linewidth}{!}{\begin{tikzpicture}

\begin{axis}[
tick align=outside,
tick pos=left,
xlabel={\(\displaystyle s\)},
xmin=-0.005, xmax=0.10,
xtick style={color=black},
xticklabel style={
        /pgf/number format/fixed,
        /pgf/number format/precision=3
},
scaled x ticks=false,
ylabel={\(\displaystyle (\bar{g}_{\epsilon})^{-1}(s)\)},
ymin=0.80, ymax=4.2,
ytick style={color=black}
]
\addplot [semithick]
table {%
0 1.63636363636364
2.42438811644486e-05 1.66666666666667
0.000136114702768985 1.6969696969697
0.000334805446943915 1.72727272727273
0.000615928702857585 1.75757575757576
0.000975242647795754 1.78787878787879
0.00140866399123948 1.81818181818182
0.00191224825733521 1.84848484848485
0.00248220545756173 1.87878787878788
0.00311489778577876 1.90909090909091
0.00380681163493545 1.93939393939394
0.00455460441586295 1.96969696969697
0.00535506239815976 2
0.00620511058693757 2.03030303030303
0.00710180981718671 2.06060606060606
0.00804235971463597 2.09090909090909
0.00902408551379137 2.12121212121212
0.010044442878792 2.15151515151515
0.0111010095453974 2.18181818181818
0.0121914849281335 2.21212121212121
0.013313685706878 2.24242424242424
0.0144655416733553 2.27272727272727
0.0156450919511855 2.3030303030303
0.0168504819065824 2.33333333333333
0.0180799590268646 2.36363636363636
0.019331869137051 2.39393939393939
0.0206046512656373 2.42424242424242
0.0218968374402447 2.45454545454545
0.0232070444820913 2.48484848484848
0.0245339751041699 2.51515151515152
0.0258764046946386 2.54545454545455
0.0272331949172695 2.57575757575758
0.0286032677200124 2.60606060606061
0.0299818928071383 2.63636363636364
0.0313793308269921 2.66666666666667
0.032783503970362 2.6969696969697
0.0341973296238252 2.72727272727273
0.0356200463103369 2.75757575757576
0.0370509475989319 2.78787878787879
0.0384893666161458 2.81818181818182
0.0399347023637558 2.84848484848485
0.0413863723423519 2.87878787878788
0.0428438611375265 2.90909090909091
0.0443066762961187 2.93939393939394
0.0457743662104323 2.96969696969697
0.0472465139878842 3
0.0487227338456963 3.03030303030303
0.050202670458133 3.06060606060606
0.051685995916938 3.09090909090909
0.0531724089650234 3.12121212121212
0.0546616321514763 3.15151515151515
0.0561534105157532 3.18181818181818
0.05764751097763 3.21212121212121
0.0591437182445703 3.24242424242424
0.0606418369002541 3.27272727272727
0.0621416870550354 3.3030303030303
0.0636431041734369 3.33333333333333
0.065145938827221 3.36363636363636
0.0666500536625573 3.39393939393939
0.0681553247852691 3.42424242424242
0.0696616384469235 3.45454545454545
0.0711688913929499 3.48484848484848
0.0726769900314714 3.51515151515152
0.0741858490880628 3.54545454545455
0.0756953917477496 3.57575757575758
0.0772055478450338 3.60606060606061
0.0787162541078945 3.63636363636364
0.0802274538382717 3.66666666666667
0.0817390943863197 3.6969696969697
0.0832511304660664 3.72727272727273
0.0847635181268646 3.75757575757576
0.0862762221036531 3.78787878787879
0.0877892066367442 3.81818181818182
0.089302441601933 3.84848484848485
0.0908158994962612 3.87878787878788
0.0923295556957974 3.90909090909091
0.0938433882169554 3.93939393939394
0.0953573756479847 3.96969696969697
0.0968715048988364 4
};
\end{axis}

\end{tikzpicture}}
\caption{Plot of \((\bar{g}_{\epsilon})^{-1}(s)\)}\label{fig:g_inverse}
\end{subfigure}
\caption{Visualization of \(\gamma_{\tau}\) and \((\bar{g}_{\epsilon})^{-1}(s)\)}\label{fig:approx-visual}
\end{figure}
We depict the approximation error bound \(\gamma_\tau\) as a function of \(\tau\) in Figure~\ref{fig:approx-bd-gamma}, using a Gaussian \(\Prob_0\). From this figure, we observe that \(\gamma_\tau\) is close to \(1\) for small \(\tau\), suggesting that \(\hat{\mathcal{C}}_N\) is a tight inner approximation of \(\mathcal{C}_\delta\). Accordingly, the ensuing \(\hat{\mathcal{X}}^p_{\text{T}}\) is a tight inner approximation of \(\mathcal{X}^p_{\text{T}}\).

\end{FirstUpdate}

\section{Optimistic Chance Constraint}
\label{sec:optimistic-drcc}
This section extends the convexity results for (\PDRCCL) in Section~\ref{sec:pessimistic-case} to (\ODRCCL). We first present a \CVaR{} reformulation for \(\mathcal{X}^o\) by
adapting Theorem~\(1\) in~\cite{xie-2019-distr-robus}. Then, we study joint (\ODRCCL{}) with \RHS{} uncertainty and \modify{two-sided (\ODRCCL{}) with LHS uncertainty} in Sections~\ref{sec:opt-joint-drcc-rhs} and~\ref{sec:opti-two-sided}, respectively.

\begin{theorem}\label{thm:opt-det-reform}
For \(\delta > 0\) it holds that
\begin{align*}
  \mathcal{X}^o = \Set{
  x \in \reals^n \colon
  \CVaR_{\epsilon} \Big( - \Dist{{\zeta}, \Safe(x)} \Big)
  + \frac{\delta}{1 - \epsilon} \geq 0},
\end{align*}
where the \CVaR{} is with respect to the reference distribution \(\Prob\) and \(\Dist{{\zeta}, \Safe(x)}\) is the distance from \(\zeta \in \reals^m\) to the safe set \(\Safe(x)\),
\begin{align*}
\Dist{\zeta, \Safe(x)}
& := \inf_{\xi \in \Xi} \Set{\norm{ \zeta - \xi } \colon A(x)\xi \leq b(x)}.
\end{align*}
\end{theorem}

\subsection{Joint (\ODRCCL{}) with \RHS{} Uncertainty}
\label{sec:opt-joint-drcc-rhs}

When \(\xi\) arises from the RHS, we recall the formulation of (\ODRCCL{}):
\begin{align*}
  \mathcal{X}^o_{\text{R}} := \Set{ x \in \reals^n \colon
  \sup_{\QProb \in \mathcal{P}} \QProb \Big[ A\xi \leq b(x) \Big] \geq 1 - \epsilon}.
\end{align*}
As a preparation, we show that the distance \(\Dist{\zeta, \Safe(x)}\) from \(\zeta \in \reals^m\) to the safe set \(\Safe(x)\) is convex.
\begin{lemma} \label{lem:dist-convex}
\(\Dist{\zeta, \Safe(x)} \equiv \min_{\xi \in \Xi}\Big\{\norm{\xi - \zeta}: A\xi \leq b(x)\Big\}\) is jointly convex in \((\zeta, x)\) on
\(\Xi \times \reals^n\).
\end{lemma}

Now we are ready to present the main result of this subsection.
\begin{theorem}\label{thm:opt-rhs-cvx}
Suppose that the reference distribution \(\Prob\) of \(\mathcal{P}\) is \(\alpha\)-concave with
\(0 \leq \alpha \leq 1/m\). Then, \(\mathcal{X}^o_{\text{R}}\) is convex and
closed for \(\delta > 0\).
\end{theorem}

\modify{
Although Theorem~\ref{thm:opt-rhs-cvx} pertains to (\ODRCCL) with linear inequalities, the convexity result extends to (\ODRCCL) with \emph{quasi-concave} inequalities. We present this generalization in Appendix~\ref{thm:opt-rhs-cvx-general}.
}

\begin{FirstUpdate}
\subsection{Two-Sided~(\ODRCCL{}) with LHS Uncertainty}%
\label{sec:opti-two-sided}

We extend the convexity result for two-sided~(\PDRCCL) in Section~\ref{sec:pess-two-sided} to the optimistic setting. Specifically, define
\begin{align*}
  \mathcal{X}^o_{\text{T}} := \Set{(x, \ell, u) \in \reals^n \times \reals_- \times \reals_+ \colon
  \sup_{\QProb \in \mathcal{P}} \QProb \left[ \ell \leq x^{\top} \xi \leq u \right] \geq 1 - \epsilon
  }.
\end{align*}
When the Wasserstein ball \(\mathcal{P}\) has an elliptical reference distribution \(\Prob\), we can ``project'' \(\mathcal{P}\) onto a Wasserstein ball for \(1\)-dimensional distributions centered around \(\Prob_0\), which is induced by \(R\cdot e_1^{\top}U_n\) with \(\Prob \disteq R\cdot \Sigma^{1/2}U_n\) (see Remark~\ref{rmk:normal-of-ellip-dist} in Appendix~\ref{apx-preliminary}). This relates \(\mathcal{X}^o_{\text{T}}\) to a primitive set
\begin{align*}
  \mathcal{X}^o_{\text{T}_0} := \Set{(\ell, u) \in \reals_- \times \reals_+ \colon
  \sup_{\QProb \in \mathcal{P}_0} \QProb \left[ \ell \leq \zeta \leq u \right] \geq 1 - \epsilon
  },
\end{align*}
where \(\mathcal{P}_0\) is a Wasserstein ball centered around \(\Prob_0\) with the same radius \(\delta\) as in \(\mathcal{P}\).
\begin{lemma}\label{lem:optimistic-ts-rel-T-T0}
Suppose that Assumption~\ref{asm:ellip-ref-and-norm} holds and \(\epsilon \in (0, 1/2)\). Then, for any \(x \neq 0\), \((x, \ell, u) \in \mathcal{X}^o_\text{T}\) if and only if
\(\left( \frac{\ell}{\norm{x}_{\ast}}, \frac{u}{\norm{x}_{\ast}}\right) \in \mathcal{X}^o_{\text{T}_0}\).
\end{lemma}

We are now ready to present the convexity result for \(\mathcal{X}^o_{\text{T}}\).
\begin{theorem}\label{thm:ts-occ-convexity}
Suppose that Assumptions~\ref{asm:ellip-ref-and-norm} and~\ref{asm:unimodal-ref-w-diff-density} hold, \(\epsilon \in (0, \frac{1}{2})\), and \(\delta > 0\). Define
\begin{gather*}
  h_{\epsilon}(\ell, u) := \int\limits_0^{+\infty} \Big[
    (1 - \epsilon) - (\Phi(u + t) - \Phi(\ell - t)) \Big]^+ \dLeb{t}.
\end{gather*}
Then, it holds that \(\mathcal{X}^o_{\text{T}_0} = \Set{
  (\ell, u) \in \reals_- \times \reals_+ \colon h_{\epsilon}(\ell, u) \leq \delta}\) and
\begin{gather*}
  \mathcal{X}^o_{\text{T}} = \Set{(x, \ell, u) \in \reals^n \times \reals_- \times \reals_+ \colon
  \exists \, s \geq 0 \text{ such that } \norm{x}_{\ast} \leq s, \ (\ell, u, s) \in \text{co}\left(\mathcal{X}^o_{\text{T}_0}\right)},
\end{gather*}
where \(\text{co}\left(\mathcal{X}^o_{\text{T}_0}\right) := \cl{} \left(\Set{
  (\ell, u, s) \in \reals_- \times \reals_+ \times \reals \colon s > 0,
  (\ell/s, u/s) \in \mathcal{X}^o_{\text{T}_0}} \right)\) represents the cone generated by \(\mathcal{X}^o_{\text{T}_0}\). Furthermore, \(\mathcal{X}^o_{\text{T}}\) is convex and closed.
\end{theorem}

\end{FirstUpdate}

\section{Numerical Experiments}\label{sec:exps}

We demonstrate the theoretical results through two numerical experiments: a (\SDC) model using joint (\PDRCCL{}) in Section~\ref{sec:exps-pess-joint-drcc} and \redmodify{a (\WP) model using two-sided (\PDRCCL) in Section~\ref{sec:exps-hydro-power}}.

\subsection{Production Planning}
\label{sec:exps-pess-joint-drcc}

We consider a (\SDC) model that seeks to procure production capacity so that all
demands can be satisfied with high probability and a minimal procurement cost
(see Example~\ref{exa:pp}). Specifically, we consider the following formulation
with (\PDRCCL):
\begin{align*}
  \min~%
  & c^{\top} x, \\
  \st{} ~
  & \inf_{\QProb \in \mathcal{P}} \QProb \left[ T x \geq {\xi} \right] \geq 1 - \epsilon,\\
  & 0 \leq x_i \leq U, \forall i \in [n],
\end{align*}
where \(c\) represents the procurement costs, \(U\) represents a homogeneous upper bound of production capacity for all facilities, and the reference distribution \(\Prob\) of the Wasserstein ball \(\mathcal{P}\) is assumed to be pairwise independent and Gaussian. To apply Algorithm~\ref{algo:bcm-for-rho-u}, we switch the objective function with (\PDRCCL) to obtain
\begin{align*}
\rho(u) = \max_{x \in \reals_+^n, y \in
  \reals_+} ~ & \phi(x, y) \equiv \bigintsss_0^y \left( \Prob \left[
  \min_{t \in [m]} \left( T_i x - \zeta_i \right) \geq t  \right] - (1 - \epsilon) \right) \dLeb{t} \\
  \st{} ~
  & c^{\top} x \leq u, \\
  & 0 \leq x_i \leq U, \quad \forall i \in [n],
\end{align*}
where we adjust the procurement budget \(u\) and apply the algorithm with various \(u\) to obtain a risk envelope. In addition, when applying Algorithm~\ref{algo:bcm-for-rho-u}, we employ the stochastic approach described in~\cite{norkin-1993-analy-optim} to be the oracle \(\mathcal{O}_u(y_k, \varepsilon_k)\) in Step 5 and terminate the algorithm whenever the change in \(y_k\) becomes sufficiently small, specifically, when \(|y_k - y_{k+1}| \leq 10^{-6}\).

\begin{figure}[!tbp]
\centering
\begin{tikzpicture}
\begin{axis}[
  ymode=log,
  xlabel={Iterations},%
  ylabel={\(\phi(x_k, y_k) - {\phi}^{\ast}\)},%
]

\addplot[
line width=1pt,
error bars/.cd,
y dir = both,
y explicit,
error bar style={line width=1.0pt},
error mark options={line width=0.5pt,mark size=3pt,rotate=90},
] table[x=x, y=y, y error=ey]{
    x   y             ey
    1   2.1730226     0.002259489
    2   1.5890378     0.006690989
    3   1.059986      0.006570152
    4   0.617507      0.012340062
    5   0.367558      0.007582801
    6   0.1317346     0.010366181
    7   0.0000998     0.00011013
};
\end{axis}
\end{tikzpicture}
\caption{Convergence of Algorithm~\ref{algo:bcm-for-rho-u} on Production Planning Instances; solid line = average of the difference \(\phi(x_k, y_k) - {\phi}^{\ast}\) across five runs, error bar = standard deviation of the difference}
\label{fig:exps-conv-of-algo-1}
\end{figure}

We demonstrate the convergence of Algorithm~\ref{algo:bcm-for-rho-u} in Figure~\ref{fig:exps-conv-of-algo-1}, which is obtained by running the algorithm for five times on an instance with \(n=10\), \(m=5\), \(U = 200\), \(c\) randomly drawn from the set \(\{1, \ldots, 10\}\), and \(\Expect_{\Prob}[\zeta_i]\) randomly drawn from the interval \([10, 51]\). In this figure, the solid line represents the difference between each iterate \(\phi(x_k, y_k)\) and the final iterate \(\phi^*\), averaged across the five runs, and the error bar represents the standard deviation of the difference. From this figure, we observe that Algorithm~\ref{algo:bcm-for-rho-u} converges at a linear rate in only a few iterations.

\begin{figure}[!htbp]
\centering
\begin{subfigure}{0.48\textwidth}
\begin{tikzpicture}
\begin{axis}[%
  x label style={at={(axis description cs:0.5,0.00)},anchor=north},%
  y label style={at={(axis description cs:0.1,.5)},rotate=0,anchor=south},%
  xlabel={\(u\) budget},%
  ylabel={\(\delta\) radius},%
  legend pos=north west,
  legend cell align={left},%
]
\addplot [smooth, mark=*, color=blue, line width=1pt] table [x=budget, y=radius, col sep=comma] {./exps/SDC-N10T5-epsi0.1-seed7283.csv};
\addlegendentry{\(\epsilon = 0.1\)}
\addplot [smooth, mark=*, color=green, line width=1pt] table [x=budget, y=radius, col sep=comma] {./exps/SDC-N10T5-epsi0.15-seed7283.csv};
\addlegendentry{\(\epsilon = 0.15\)}
\addplot [smooth, mark=o, color=magenta, line width=1pt] table [x=budget, y=radius, col sep=comma] {./exps/SDC-N10T5-epsi0.2-seed7283.csv};
\addlegendentry{\(\epsilon = 0.2\)}
\end{axis}
\end{tikzpicture}
\caption{\(n = 10\) and \(m = 5\)}
\end{subfigure}
\hfill
\begin{subfigure}{0.48\textwidth}
\begin{tikzpicture}
\begin{axis}[%
  x label style={at={(axis description cs:0.5,0.00)},anchor=north},%
  y label style={at={(axis description cs:0.1,.5)},rotate=0,anchor=south},%
  xlabel={\(u\) budget},%
  ylabel={\(\delta\) radius},%
  legend pos=north west,
  legend cell align={left},%
]
\addplot [smooth, mark=*, color=blue, line width=1pt] table [x=budget, y=radius, col sep=comma] {./exps/SDC-N30T20-epsi0.1-seed1116.csv};
\addlegendentry{\(\epsilon = 0.1\)}
\addplot [smooth, mark=*, color=green, line width=1pt] table [x=budget, y=radius, col sep=comma] {./exps/SDC-N30T20-epsi0.15-seed1116.csv};
\addlegendentry{\(\epsilon = 0.15\)}
\addplot [smooth, mark=o, color=magenta, line width=1pt] table [x=budget, y=radius, col sep=comma] {./exps/SDC-N30T20-epsi0.2-seed1116.csv};
\addlegendentry{\(\epsilon = 0.2\)}
\end{axis}
\end{tikzpicture}
\caption{\(n = 30\) and \(m = 20\)}
\end{subfigure}
\caption{Risk envelopes under different risk thresholds}\label{fig:exps-stodmd-risk-envs}
\end{figure}
We demonstrate the trade-off between the robustness and the budget in Figure~\ref{fig:exps-stodmd-risk-envs}, which is obtained by solving instances with \(\epsilon \in \{0.1, 0.15, 0.2\}\), \(n \in \{10, 30\}\), and \(m \in \{5, 20\}\). The vertical axis of this figure represents \(\rho(u)\), i.e., the largest Wasserstein radius \(\delta\) that allows (\PDRCCL) to be satisfied. From this figure, we observe that, for fixed \(\epsilon\), the largest allowable \(\delta\) is an S-shaped function of the the budget \(u\). That is, \(\delta\) remains at zero for small \(u\), and then \(\delta\) increases with a diminishing momentum as \(u\) becomes larger. In addition, for fixed \(\delta\), it needs a larger budget \(u\) to keep (\PDRCCL) satisfied as \(\epsilon\) decreases.

\begin{FirstUpdate}
\subsubsection{Out-of-Sample Performance}\label{sec:exps-pp-out-of-sample}

We compare the out-of-sample performance of (\PDRCCL) and
(\CCP{}) on production planning problems with
\(\delta \in \set{0.005, 0.01, 0.02, 0.05}, N \in \set{5, 10, 15, 30, 50, 70}, \epsilon = 0.1, m = 5\), and \(n = 10\). Specifically, we assume that \(\Probt\) is a multivariate Gaussian distribution with known mean and covariance and draw a set of \(N\) training
data samples from \(\Probt\). Then, we center
\(\mathcal{P}\) around the Gaussian distribution with empirical mean and covariance matrix, which we estimate from the training data. After
obtaining the optimal solutions to (\PDRCCL{}) and (\CCP{}) models, we compare their
out-of-sample performance by the probability of fully satisfying the demands with respect to \(\Probt\) and report the results in Figure~\ref{fig:exps-pp-oos}. In this figure, the solid line and dots are the
average out-of-sample performance across three randomly generated instances and
the shaded region around them are the \(95\%\) confidence interval.
\begin{figure}[!htbp]
\centering
\begin{subfigure}[b]{0.49\textwidth}
\centering
\includegraphics[width=\textwidth]{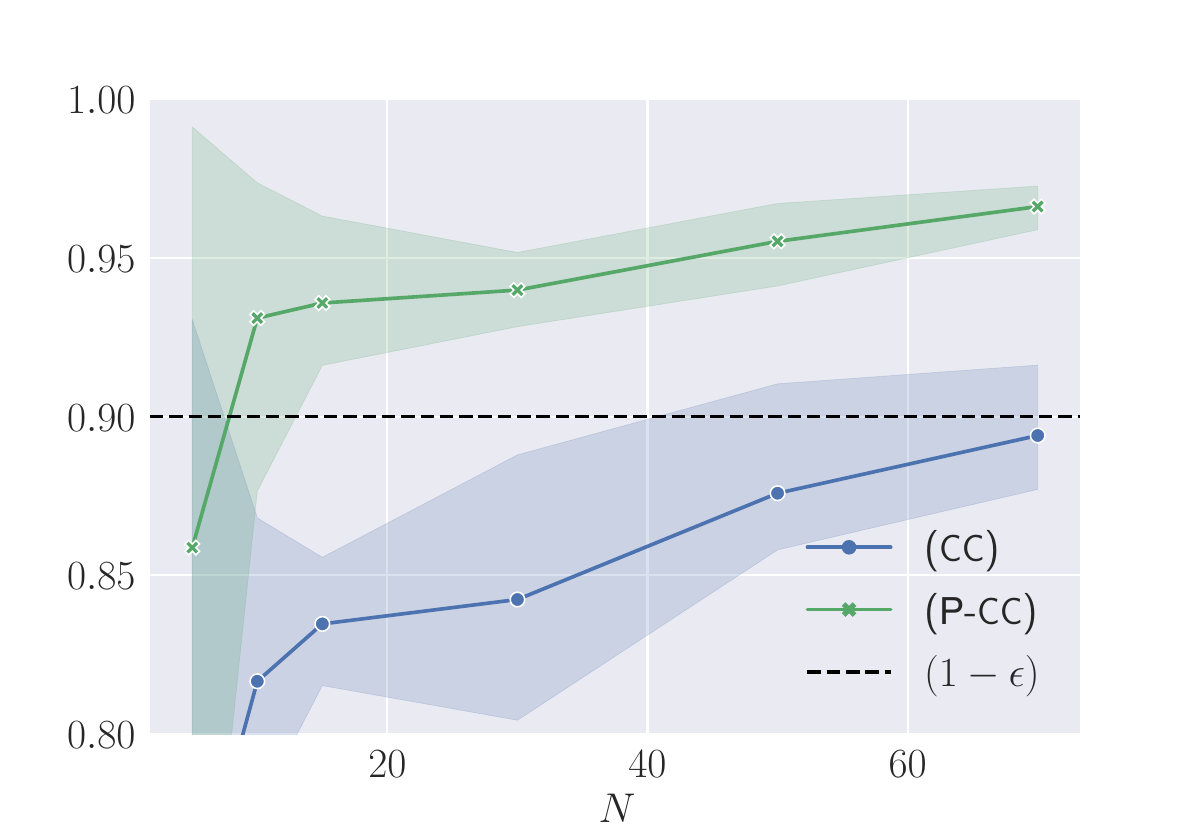}
\caption{Radius \(\delta = 0.02\)}\label{fig:exps-pp-oos-delta}
\end{subfigure}
\hfill
\begin{subfigure}[b]{0.49\textwidth}
\centering
\includegraphics[width=\textwidth]{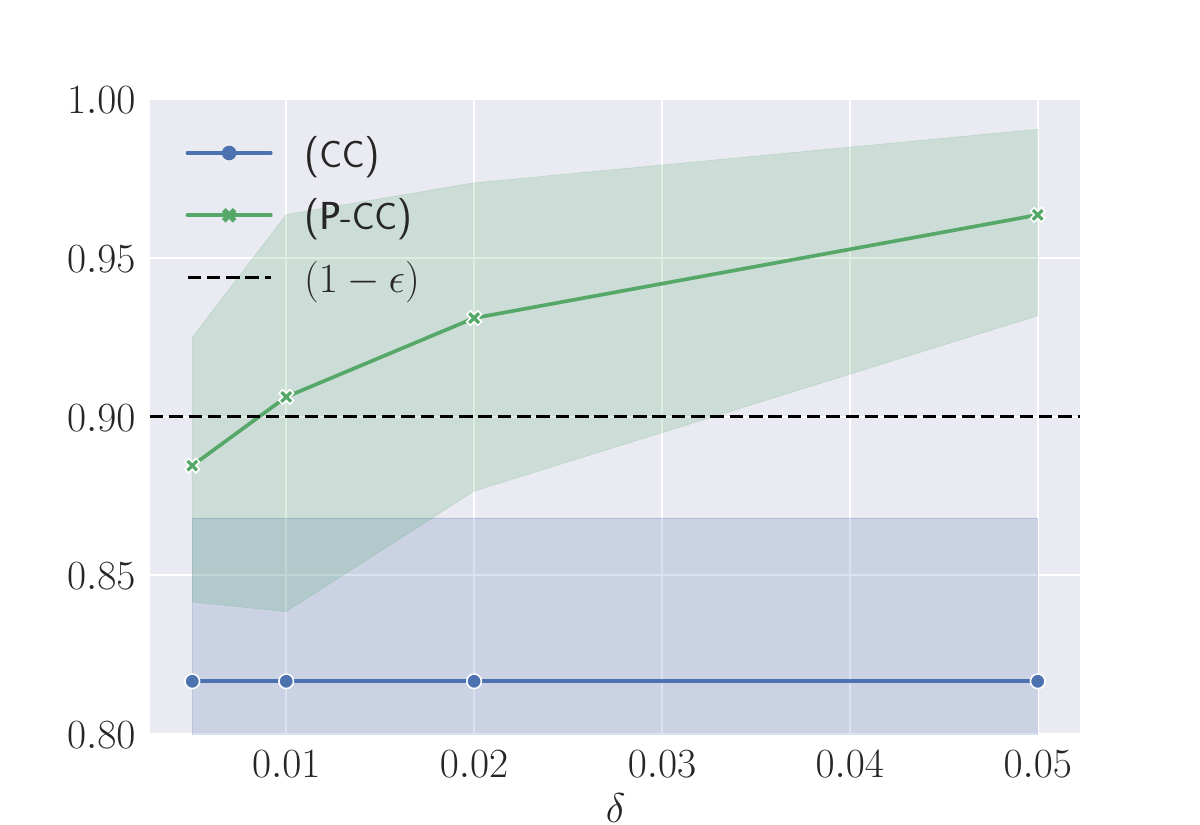}
\caption{Data size \(N = 10\)}\label{fig:exps-pp-oos-N}
\end{subfigure}
\caption{Out-of-sample performance of (\PDRCCL{}) and (\CCP{}) on the production planning problem}\label{fig:exps-pp-oos}
\end{figure}
From Figure~\ref{fig:exps-pp-oos-delta}, we
observe that the out-of-sample performance of both models improve as we obtain more training data. However, there is a significant difference in their sensitivity to having more data: (\PDRCCL{}) achieves the target reliability level of \(90\%\) with only \(N = 10\) data, whereas (\CCP{}) fails to achieve the same level even when \(70\) samples are provided. In addition, we observe from Figure~\ref{fig:exps-pp-oos-N} that, with as few as \(N=10\) data, the reliability of (\PDRCCL) quickly increases and achieves the target reliability level as soon as \(\delta\) exceeds \(0.01\). Intuitively, a small degree of pessimism suffices to improve reliability drastically.
\end{FirstUpdate}

\begin{SecondUpdateRemoved}

\end{SecondUpdateRemoved}

\begin{SecondUpdate}
\subsection{Hydro Planning}
\label{sec:exps-hydro-power}

We test our inner approximation approach for solving the two-sided~(\PDRCCL{}) using the hydro planning model in Example~\ref{exa:water-plan}, wherein we seek to maximize the revenue of electricity generation while maintaining the water inventory within the safety interval \([\ell_{\text{low}}, \ell_{\text{high}}]\):
\begin{align*}
\max \
& \ \sum_{t = 1}^T c_t x_t, \\
\st{} \
& \ \inf_{\Prob \in \mathcal{P}} \Prob \left[
\ell_{\text{low}} \leq \ell_0 + \sum_{i=1}^t \left( \xi_i - x_i \right) \leq \ell_{\text{high}}
\right] \geq (1 - \epsilon), \quad \forall t \in [T], \\
& \ x \in \reals^T_+,
\end{align*}
where \(c_t\) denotes the unit price of electricity in time unit \(t\). We follow~\cite{zymler-2011-distr-robus} by
setting \(T = 5, \ell_0 = 1, \ell_{\text{low}} = 1, \ell_{\text{high}} = 5\), and \(c_t = 10 + 5 \sin \left[ \frac{\pi (1 - t)}{3} \right]\) for all \(t \in [T]\). To calibrate the Wasserstein ball \(\mathcal{P}\), we follow~\cite{zymler-2011-distr-robus} to generate \(N\) historical data samples of \(\xi_i\) from a truncated Gaussian distribution with the support \(\Xi = [0, 2]\),
mean \(\mu = 1\), and standard deviation \(0.1\). In addition, we set
the correlation between \(\xi_i\) and \(\xi_{i+1}\) to be \(0.10\) for all
\(i \in [T]\) and the risk level \(\epsilon\) to be \(10\%\). We evaluate the
out-of-sample reliability of a given solution, that is, the (joint) probability of the water level stays within \([\ell_{\text{low}}, \ell_{\text{high}}]\) across all \(T\) time periods, using Monte Carlo sampling and the same approach for generating the calibration (training) data.

\begin{figure}[!htbp]
\centering
\begin{subfigure}{.49\textwidth}
\centering
\includegraphics[width=\textwidth]{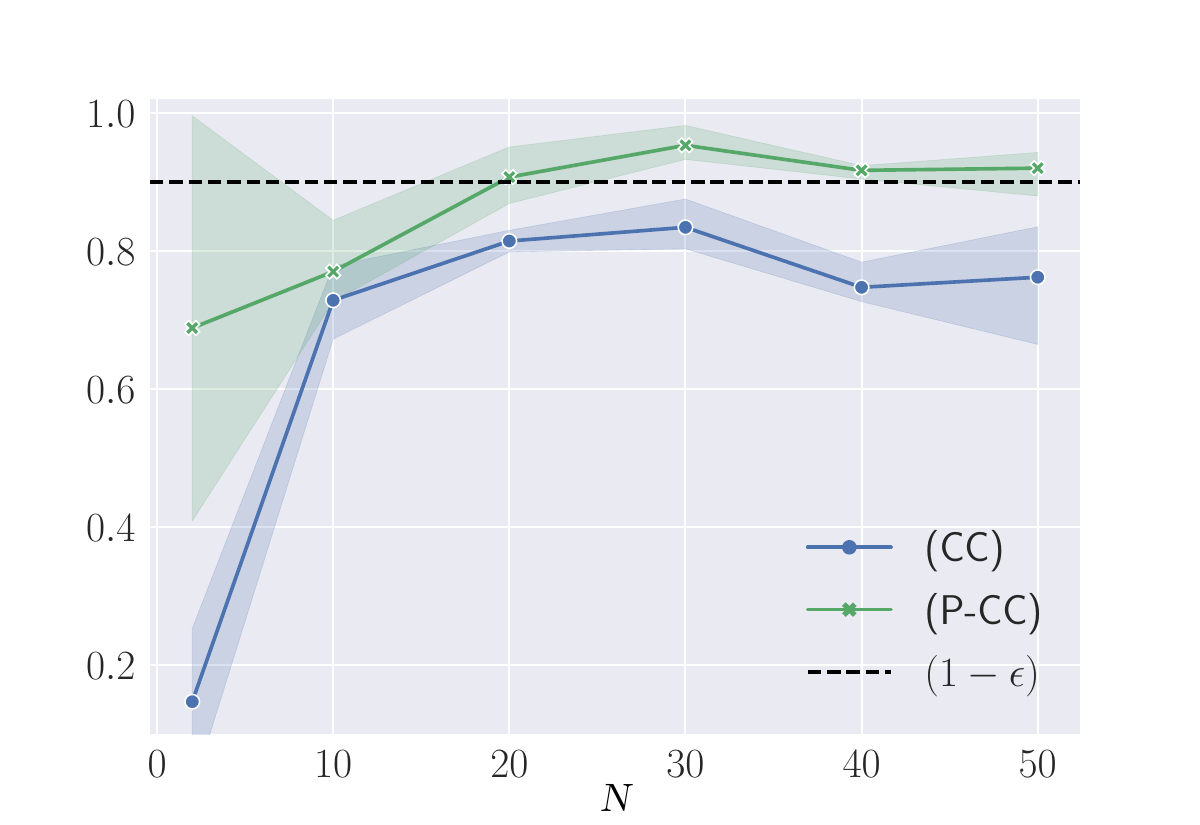}
\caption{\small Radius \(\delta = 0.019\)}\label{fig:exps-hydro-oos-1}
\end{subfigure}%
\hfill
\begin{subfigure}{.49\textwidth}
\centering
\includegraphics[width=\textwidth]{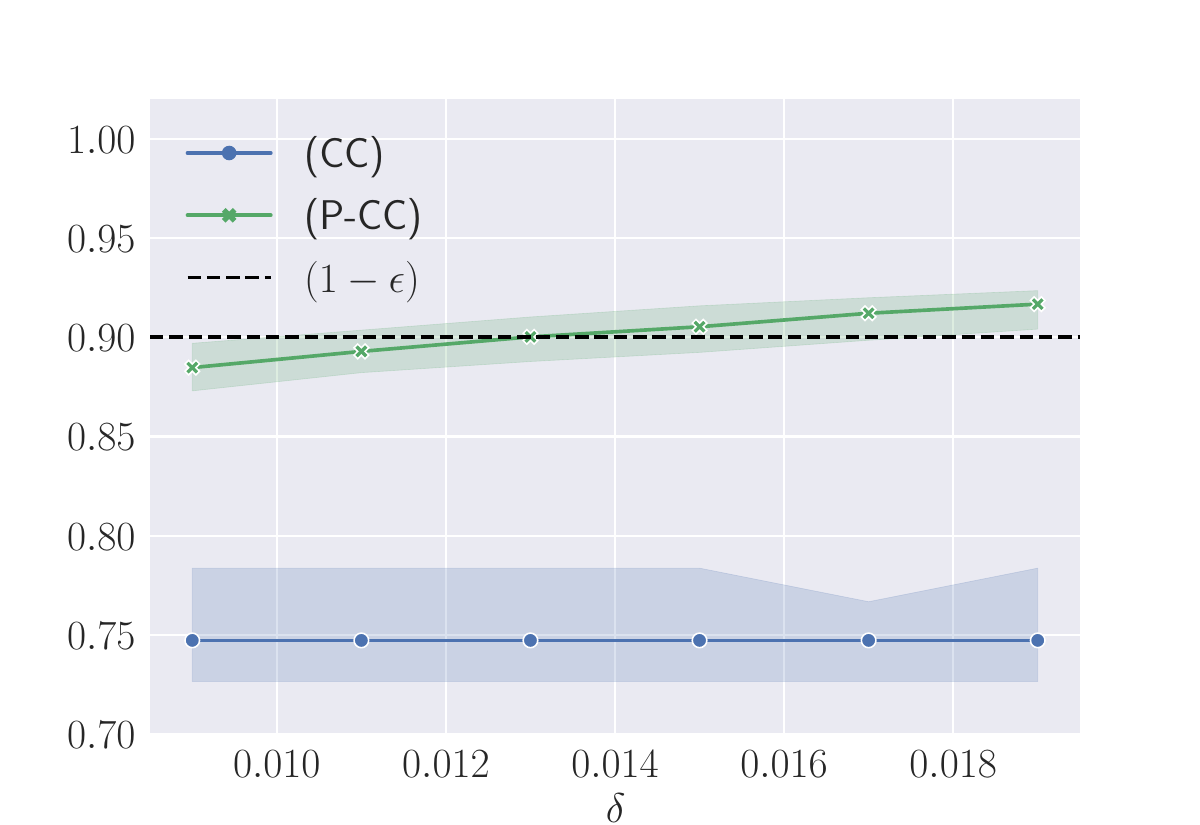}
\caption{\small Training data size \(N = 50\)}\label{fig:exps-hydro-oos-2}
\end{subfigure}
\caption{Out-of-sample reliability of (\textbf{P-CC}) and (\textbf{CC}) for hydro planning\label{fig:exps-hydro-oos}}
\end{figure}

\subsubsection{Out-of-Sample Reliability}%
\label{sec:exps-hydro-oos}

To compare~(\PDRCCL{}) and~(\CCP{}), we solve
randomly generated instances of the hydro planning problem with
\(\delta \in \set{0.009, 0.011, 0.013, 0.015, 0.017, 0.019}\), \(N \in \set{2, 10, 20, 30, 40, 50}\),
and report the results in Figure~\ref{fig:exps-hydro-oos}, where the solid lines
are the average value over five runs and the shaded regions represent the
associated \(95\%\) confidence intervals. In particular,
Figure~\ref{fig:exps-hydro-oos-1} depicts the out-of-sample reliability as the
training data size \(N\) increases. From this figure, we observe that the
out-of-sample reliability of both models improve as \(N\) increases,
but there is a significant difference in the effectiveness of using the data.
For example, (\PDRCCL) is able to achieve the target reliability of \(90\%\)
using only \(20\) data on hand, while (\CCP) fails to achieve the same
target even with \(50\) data samples. This demonstrates that the (\PDRCCL) model
can utilize the data more effectively in a data-driven context. In addition,
Figure~\ref{fig:exps-hydro-oos-2} depicts the improvement of the (\PDRCCL)
out-of-sample reliability as we increase the radius \(\delta\) of the
Wasserstein ball. From this figure, we observe that the reliability of (\PDRCCL)
achieves the target reliability of \(90\%\) as soon as \(\delta\) exceeds~0.012.

\subsubsection{Comparison with Moment Ambiguous Chance Constraints}%
\label{sec:exps-hydro-mom}

We compare the solution quality of our
Wasserstein~(\PDRCCL{}) with that of a \emph{moment} ambiguous~(\PDRCCL{})
studied in~\citet{xie-2017-distr-robus}, who adopted a moment ambiguity set using the first two moments of \(\xi\). To this end, we first generate random instances of the hydro planning problem with different (training) data sizes \(N \in \set{500, 700, 900, 1000}\). Next, we use the sample mean and covariance matrix to characterize the moment ambiguity set and use the \(k\)-fold
cross validation to select the smallest radius \(\delta > 0\) to achieve
the target reliability of \((1 - \epsilon) = 90\%\) for the Wasserstein ambiguity set. Then, we solve the problem instances with respect to the two ambiguity sets and compare the optimal revenues and the out-of-sample reliability of the respective
solutions.
\begin{table}[!htbp]
\centering
\caption{Comparisons between Wasserstein~(\PDRCCL{}) and moment~(\PDRCCL{})}\label{tab:exps-hydro-mom}
\begin{tabularx}{0.7\textwidth}{@{}l|XX|XX@{}}
\toprule
{}    & \multicolumn{2}{c}{\small Out-of-Sample Reliability (\%)} & \multicolumn{2}{c}{\small Optimal Revenue (\$)} \\
\(N\) &        {\small Wasserstein}                   &      {\small moment} &        {\small Wasserstein} &     {\small moment}\\
\midrule
500   &  94.6                                         &  99.6                &  68.7                     &  63.9 \\
700   &  93.4                                         &  99.6                &  69.0                     &  63.9 \\
900   &  94.0                                         &  99.7                &  68.8                     &  63.8 \\
1000  &  94.3                                         &  99.7                &  68.7                     &  63.8 \\
\midrule
Average   & 94.1                                         & 99.7                 & 68.8                       & 63.8 \\
\bottomrule
\end{tabularx}
\end{table}

Table~\ref{tab:exps-hydro-mom} reports their out-of-sample
reliability and optimal revenue across varying data sizes \(N\). We observe that the out-of-sample reliability of both models exceed the target reliability of~90\%, but that of the Wasserstein~(\PDRCCL{}) is significantly closer to the target level than that of the moment~(\PDRCCL{}). This implies
that the Wasserstein model is less conservative than the moment model. The comparison of optimal revenues confirms this implication because the Wasserstein model consistently outperforms the moment model.

\subsubsection{Comparison with a Mixed-Integer Reformulation of~(\PDRCCL{})}
\label{sec:exps-hydro-runtime}

To demonstrate the scalability of our convex (\PDRCCL{}) model, we report the runtime of our \SOC{} formulation proposed in Proposition~\ref{prop:soc-inner} and Algorithm~\ref{algo:seeds-for-C-hat} with the error threshold \(\tau = 10^{-5}\) under various parameter settings, and compare with the runtime of the mixed-integer second-order conic (MISOC)
reformulation proposed in~\cite[Theorem~2]{xie-2019-distr-robus}. Specifically,
we generate random hydro planning instances with (training) data size \(N \in \{500, 700, 900, 1000\}\), risk level
\(\epsilon \in \set{0.05, 0.10}\), and radius
\(\delta \in \set{0.01, 0.05, 0.09}\). The average runtime over three different
instances of our~\SOC{} formulation and that of~\cite{xie-2019-distr-robus}'s
MISOC reformulation are reported in
Table~\ref{tab:exps-hydro-mom-time}. We observe that our approach is faster and more scalable in \((N, \epsilon, \delta)\) than the mixed-integer reformulation.
\begin{table}[!htbp]
\centering\small
\caption{Comparison between \SOC{} and MISOC Reformulations}\label{tab:exps-hydro-mom-time}
\begin{tabularx}{0.7\textwidth}{@{}XXX|XX@{}}
  \toprule
        &              & {}         & \multicolumn{2}{c}{Runtime (sec)} \\
  \(N\) & \(\epsilon\) & \(\delta\) &  \SOC{}      &    MISOC               \\
  \midrule
  500   & 0.05         & 0.01       &      0.00050 &          2.21 \\
        &              & 0.05       &      0.00036 &          2.14 \\
        &              & 0.09       &      0.00028 &          2.23 \\
        & 0.10         & 0.01       &      0.00085 &          6.29 \\
        &              & 0.05       &      0.00066 &          6.23 \\
        &              & 0.09       &      0.00033 &          6.29 \\
  700   & 0.05         & 0.01       &      0.00074 &          4.73 \\
        &              & 0.05       &      0.00033 &          4.69 \\
        &              & 0.09       &      0.00026 &          4.73 \\
        & 0.10         & 0.01       &      0.00064 &         11.85 \\
        &              & 0.05       &      0.00081 &         12.06 \\
        &              & 0.09       &      0.00033 &         11.62 \\
  900   & 0.05         & 0.01       &      0.00048 &          7.49 \\
        &              & 0.05       &      0.00034 &          7.48 \\
        &              & 0.09       &      0.00043 &          7.39 \\
        & 0.10         & 0.01       &      0.00064 &         21.20 \\
        &              & 0.05       &      0.00042 &         21.07 \\
        &              & 0.09       &      0.00033 &         28.93 \\
  1000  & 0.05         & 0.01       &      0.00046 &         16.12 \\
        &              & 0.05       &      0.00034 &         16.11 \\
        &              & 0.09       &      0.00027 &         16.19 \\
        & 0.10         & 0.01       &      0.00177 &         17073.21 \\
        &              & 0.05       &      0.00068 &         25.13 \\
        &              & 0.09       &      0.00033 &         31.70 \\
  \midrule
 Average    &              &            &     0.00053  & 722.79 \\
  \bottomrule
\end{tabularx}
\end{table}

\end{SecondUpdate}

\clearpage{}

\printbibliography{}

\clearpage{}

\begin{appendix}

\section{Preliminary Results} \label{apx-preliminary}
We review properties of \(\alpha\)-concave functions, \(\alpha\)-concave probability measures, as well as elliptical and star-unimodal distributions.

\begin{lemma}[Lemma~\(4.8\) in~\cite{shapiro-2009-lectur}]\label{lem:monotone-m-alpha}
The mapping \(\alpha \mapsto m_{\alpha}(a, b; \theta)\) is
nondecreasing and continuous.
\end{lemma}
The monotonicity of \(m_{\alpha}\) implies that if \(f\) is
\(\alpha\)-concave, then it is \(\beta\)-concave for all
\(\beta \leq \alpha\). Under certain conditions, summation preserves \(\alpha\)-concavity.
\begin{proposition}[Theorem~\(4.19\) in~\cite{shapiro-2009-lectur}]\label{prop:sum-of-alpha-concave-funcs}
If the function \(f \colon \reals^n \to \reals_+\) is
\(\alpha\)-concave and the function \(g \colon \reals^n \to \reals_+\)
is \(\beta\)-concave, where \(\alpha, \beta \geq 1\), then
\(f(x) + g(x)\) is \(\min\{\alpha, \beta\}\)-concave.
\end{proposition}

The next two propositions review the relationship between \(\alpha\)-concave
probability measures and their densities.
\begin{proposition}[Theorem~\(4.15\) in~\cite{shapiro-2009-lectur}]\label{prop:alpha-prob-meas-and-pdf}
Let \(\Omega\) be a convex subset of \(\reals^n\) and \(s\) be the
dimension of the smallest affine subspace \(\mathcal{H}(\Omega)\)
containing \(\Omega\). The probability measure \(\Prob\) is
\(\alpha\)-concave with \(\alpha \leq 1 /s\) if and only if its
probability density function (PDF) with respect to the Lebesgue measure on
\(\mathcal{H}\) is \(\alpha^{\prime}\)-concave with
\begin{align*}
  \alpha^{\prime} :=
  \begin{cases}
  \alpha/ ( 1 - s \alpha) & \text{ if \(\alpha \in (-\infty, 1/s)\),} \\
  -1/s & \text{ if \(\alpha = -\infty\),} \\
  +\infty & \text{ if \(\alpha = 1/s\).}
  \end{cases}
\end{align*}
\end{proposition}

\begin{proposition}[Theorem~\(2\) in~\cite{gupta-1980-brunn-minkow}]\label{prop:brunn-minkow-ineq}
Let \(f_0, f_1\) be two non-negative Borel-measurable functions on \(\reals^n\)
with non-empty supports \(S_0\) and \(S_1\), respectively. Assume that \(f_0\)
and \(f_1\) are integrable with respect to the Lebesgue measure on \(\reals^n\).
Let \(\theta \in (0, 1)\) be a fixed number and \(f\) be a non-negative,
measurable function on \(\reals^n\) such that
\begin{align*}
f(x) \geq m_{\alpha} [f_0(x_0), f_1(x_1); \theta],
\end{align*}
whenever \(x = (1 - \theta)x_0 + \theta x_1\) with \(x_0 \in S_0, x_1 \in S_1\);
\(-1/n \leq \alpha \leq +\infty\). Then
\begin{align*}
  \int_{(1 - \theta)S_0 + \theta S_1} f(x) \dLeb{x}
  \geq m_{\alpha^{\ast}_n} \left[ \int_{S_0} f_0(x) \dLeb{x}, \int_{S_1} f_1(x) \dLeb{x}; \theta \right],
\end{align*}
where
\begin{align*}
  \alpha^{\ast}_n :=
  \begin{cases}
  \alpha / (1 + n \alpha) & \text{ if \(\alpha > -1/n\),} \\
  1/n & \text{ if \(\alpha = +\infty\),} \\
  -\infty & \text{ if \(\alpha = -1/n\).}
  \end{cases}
\end{align*}
\end{proposition}

\begin{FirstUpdate}
We move on to review preliminary results for elliptical distributions. Following the convention in the literature, we write
\(X \sim \mathcal{E}_n(\mu, \Sigma, \phi)\) if \(X\) is elliptically distributed
with parameters \(\mu, \Sigma, \phi\). Specifically, we say a random vector
\(Y\) is spherically distributed if \(Y \sim \mathcal{E}_n(0, I_n, \phi)\), which implies
\(\Lambda Y \sim \mathcal{E}_n(0, \Lambda \Lambda^{\top}, \phi)\)~(cf.~Proposition~\(1\) in~\cite{frahm-2004-gener-ellip}).

\begin{proposition}[Theorems~\(1\)--\(3\) in~\cite{cambanis-1981-theor-ellip}]\label{prop:repr-of-elli-dist}
\(X \sim \mathcal{E}_n(\mu, \Sigma, \phi)\) with
\(\Rank(\Sigma) = k\) if and only if
\begin{align*}
  X - \mu \disteq{} R \Lambda U_k,
\end{align*}
where \(\disteq{}\) represents being identical in distribution, \(U_k\) is a \(k\)-dimensional random vector uniformly distributed on the
sphere \(S^{k-1}\), \(R\) is a nonnegative random variable independent from
\(U_k\), \(\mu \in \reals^n\), and \(\Lambda \in \reals^{n \times k}\) with
\(\Rank(\Lambda) = k\) and \(\Sigma = \Lambda \Lambda^{\top}\). Furthermore, if
\(R\) has a density function \(f_R(\cdot)\), then the probability density function \(f_X(\cdot)\) of \(X\) can be written as
\begin{gather*}
  f_X(x) =
  C \cdot t^{-\frac{n-1}{2}} f_R\left(\sqrt{(x - \mu)^{\top} \Sigma^{-1} (x - \mu)}\right), \quad
  \forall x \in \reals^n \setminus \set{\mu},
\end{gather*}
where \(C\) is a nonnegative constant depending on \(\Lambda\) and \(k\) only.%
\end{proposition}
\begin{remark}\label{rmk:normal-of-ellip-dist}
Suppose that \(a \in \reals^n\) is a column vector and
\(X \sim \mathcal{E}_n(\mu, \Lambda \Lambda^{\top}, \phi)\), then by
Proposition~\ref{prop:repr-of-elli-dist}, we have
\begin{align*}
  a^{\top} (X - \mu)
  \disteq{} R (a^{\top} \Lambda) U_k
  = R \norm{\Lambda^{\top} a}_2 \cdot e^{\top}_a U_k
  \disteq{} R \norm{\Lambda^{\top} a}_2 \cdot e^{\top}_1 U_k,
\end{align*}
where
\(e_a := \left( \frac{\Lambda^{\top} a}{\norm{{\Lambda^{\top} a}}_2} \right)\)
is the normalized vector of \(\Lambda^{\top}a\), \(e_1\) denotes the first orthonormal basis of \(\reals^n\), and the last equality is because \(U_k\) is
invariant to orthogonal transformations. If we further assume that \(\Sigma\) is
positive definite (i.e., \(k=n\)), then
\(a^{\top}(X - \mu) / \norm{\Lambda^{\top} a}_2\) is identically distributed as
\(R \cdot e^{\top}_1 U_n\), which is a one-dimensional elliptical distribution
not dependent on \(a\).
\end{remark}
Next, we mention results for star-unimodal distributions.
\begin{proposition}[Theorem~\(2.1\) in~\cite{dharmadhikari-1988-unimod-convex}]\label{prop:charac-of-star-unimodal}
An \(n\)-dimensional random vector \(X\) is star-unimodal if and only if \(X\) is distributed as \(V^{1/n} Z\), where \(V\) and \(Z\) are independent and \(V\) is uniformly distributed on \((0,1)\).
\end{proposition}

\begin{proposition}\label{prop:proj-of-sphr-dist-is-unimodal}
Suppose that \(X \sim \mathcal{E}_n(0, I_n, \phi)\) is spherically distributed
and star-unimodal, then \(e^{\top}_1 X\) is unimodal.
\end{proposition}
We are not aware of this result in the literature, and so we provide a proof in the following.

\begin{proof}
By Propositions~\ref{prop:repr-of-elli-dist}
and~\ref{prop:charac-of-star-unimodal}, random vector \(X\) admits two representations:
\begin{align*}
R \cdot I_n \cdot U_n \disteq X \disteq V^{1/n} \cdot Z,
\end{align*}
where \(R\) is a nonnegative random variable independent from \(U_n\), which is
uniformly distributed on the sphere \(S^{n-1}\), \(V\) is a random variable
uniformly distributed on \((0,1)\) and is independent from \(Z\), an
\(n\)-dimensional random vector. Because \(X\) is spherically distributed, for
any two distinct unit vectors \(e_i, e_j \in S^{n-1}, e_i \neq e_j\), we have
\[ V^{1/n} e^{\top}_i Z \disteq R \cdot e^{\top}_i U_n \disteq R \cdot e^{\top}_j U_n \disteq V^{1/n} e^{\top}_j Z.\]
First, we show that \(Z\) is spherically distributed. Observe that
\begin{align*}
  \Prob \left[ e^{\top}_j V^{1/n} Z \leq t \right]
  & = \Expect_V \left[\Expect_{Z|V} \left[ \Ind{ V^{1/n} e^{\top}_j Z \leq t} | V\right]\right]
    = \int_0^1 \Prob \left[ v^{1/n} \cdot e^{\top}_j Z \leq t \right] \dLeb{v} \\
  & = \int_{+\infty}^t \Prob \left[ e^{\top}_j Z \leq v' \right] \dLeb{\left( \frac{t}{v'} \right)^n}
    = n t^n \int_t^{+\infty} \frac{1}{(v')^{n+1}} \Prob \left[ e^{\top}_j Z \leq v' \right] \dLeb{v'},
\end{align*}
where the first equality is by the smoothing property, the second equality is
because \(Z\) and \(V\) are independent, and the third equality is by the variable
substitution \(v' = t / v^{1/n}\). Because \(e^{\top}_j V^{1/n} Z\) is independent from \(e_j\), we can denote its distribution
function as \(F_Z(t)\) and have
\begin{gather*}
  (n t^n)^{-1} F_Z(t)
  = \int_t^{+\infty} \frac{1}{(v')^{n+1}} \Prob \left[ e^{\top}_j Z \leq v' \right] \\
  \implies
  \frac{\dif}{\dif{t}} \left[ (n t^n)^{-1} F_Z(t) \right]
  = \frac{\dif}{\dif{t}} \int_t^{+\infty} \frac{1}{(v')^{n+1}} \Prob \left[ e^{\top}_j Z \leq v' \right]
  = - \frac{1}{t^{n+1}} \Prob \left[ e^{\top}_j Z \leq t \right], \quad \forall t \neq 0, \\
  \implies
  \Prob \left[ e^{\top}_j Z \leq t \right] = - t^{n+1} \frac{\dif{}}{\dif{t}} \left[ \frac{1}{n t^n} F_Z(T) \right], \quad \forall t \neq 0,
\end{gather*}
from which we observe that the distribution function of \(e^{\top}_j Z\) does
not depend on \(e_j\). In other words,
\(e_j^{\top} Z \disteq e_i^{\top} Z\) for all \(e_i, e_j \in S^{n-1}, e_i \neq e_j\).
Let \(Z_0\) represent a random variable identically distributed as
\(e^{\top}_1 Z\) and define
\(\phi_0(t) := \Expect \left[ e^{\iu \sqrt{t} Z_0} \right]\). We notice
that the characteristic function of \(Z\) can be represented as
\begin{align*}
  \Expect \left[ e^{\iu t^{\top} Z} \right]
  = \Expect \left[ e^{\iu \norm{t}_2 e_t^{\top} Z} \right]
  = \Expect \left[ e^{\iu \norm{t}_2 e_1^{\top} Z} \right]
  = \Expect \left[ e^{\iu \norm{t}_2 Z_0} \right] = \phi_0(t^{\top} t),
\end{align*}
implying that \(Z\) is spherically distributed (clearly, \(\Expect [Z] = 0\)). By Proposition~\ref{prop:repr-of-elli-dist}, there exists a nonnegative random
variable \(R_Z\), independent from \(U_n\), such that \(Z \disteq R \cdot U_n\).

Second, we show that \(e^{\top}_1 X\) is unimodal. Let \(B_n\) be a random
variable uniformly distributed on the \(n\)-dimensional unit ball. Then,
\begin{align*}
  e^{\top}_1 X
  \disteq e^{\top}_1 \left( V^{1/n} \cdot R_Z \cdot U_n \right)
  \disteq e^{\top}_1 \left( R_Z \cdot (V^{1/n} U_n)\right)
  \disteq e^{\top}_1 \left( R_Z \cdot B_n \right)
  \disteq  R_Z \cdot e^{\top}_1 B_n,
\end{align*}
where the third equality is because \(B_n \disteq V^{1/n} U_n\) (see
Section~\(3.1.5\) of~\cite{fang-2018-symmet}). Furthermore, because the density
function of \(e^{\top}_1 B_n\) is monotone increasing on \((-\infty, 0)\) and
monotone decreasing on \((0, +\infty)\)~(see Section~\(3.1.5\)
of~\cite{fang-2018-symmet}), \(e^{\top}_1 B_n\) is unimodal. Therefore,
by Proposition~\ref{prop:charac-of-star-unimodal}, there exists a
random variable \(Z_B\), independent from \(V\), such that
\(e^{\top}_j B_n \disteq V Z_B\). It follows that
\[e^{\top}_1 X \disteq R_Z \cdot V Z_B \disteq V \cdot (R_Z Z_B),\] \ie{},
\(e^{\top}_1 X\) can be represented as the product between \(V\), a uniform distribution on \((0, 1)\), and \(R_Z Z_B\), which is independent of
\(V\). %
Therefore, \(e^{\top}_1 X\) is unimodal by Proposition~\ref{prop:charac-of-star-unimodal}.
\end{proof}
\end{FirstUpdate}

\begin{SecondUpdate}
\begin{proposition}[Steiner Formula; see (4.8) in Section~4.2 of~\cite{schneider2014convex}]\label{prop:steiner-formula}
Let \(\mathcal{K}\) be a convex body on \(\reals^n\) and \(\alpha > 0\), then
\begin{align*}
\Leb{} (\mathcal{K} + \alpha \mathcal{B}_1) = \sum_{j = 0}^n \alpha^j {n \choose j} V(\underbrace{\mathcal{K}, \cdots \mathcal{K}}_{n-j}, \underbrace{\mathcal{B}_1, \cdots \mathcal{B}_1}_j),
\end{align*}
where \(V(\cdot)\) is the mixed volume of \(n\) convex bodies. Furthermore, it
is monotone:
\begin{align*}
V(P_1, \cdots, P_n) \leq V(Q_1, \cdots, Q_n) \quad \text{ if convex bodies \(P_i \subseteq Q_i, \forall i \in [n]\).}
\end{align*}
\end{proposition}
\end{SecondUpdate}

\begin{SecondUpdate}
\begin{definition}[Section~2.1 in~\cite{beck-2015-conver-alter}]
For a given closed and proper convex function \(h\), the proximal operator is defined as
\begin{align*}
\prox_h(x) := \argmin_u \Set{ h(u) + \frac{1}{2} \norm{u - x}^2 }.
\end{align*}
\end{definition}

\begin{remark}\label{rmk:prox-opt-cond}
Let \(M > 0\) and \(h\) be a closed, proper and convex function. If
\(w = \prox_{\frac{1}{M}h}(x)\), then equivalently
\(0 \in \partial h(w) + M (w - x)\).
\end{remark}

\begin{corollary}\label{cor:approx-prox-opt-cond}
Let \(g\) and \(h\) be proper, closed, and convex functions. Let \(M > 0\) and
\(x\) be an approximate stationary point of \(g + h\), that is, there exists
\(\bm{e}\) with \(\norm{\bm{e}}_2\) small such that
\(0 \in \bm{e} + \partial g(x) + \partial h(x)\). Then, for some \(g^{\prime}\)
in \(\partial g(x)\), we have
\begin{align*}
x = \prox_{\frac{1}{M} h} \left( x - \frac{1}{M} \left(g^{\prime} + \bm{e} \right) \right).
\end{align*}
\end{corollary}

\begin{proof}
By assumption, there exists a \(g^{\prime} \in \partial g(x)\) such that
\begin{align*}
0 \in e + g^{\prime} + \partial h(x)
= \partial h(x) + M \left( \frac{1}{M} g^{\prime} + \frac{1}{M} \bm{e} \right)
= \partial h(x) + M \left( x - \left( x - \frac{1}{M} g^{\prime} - \frac{1}{M} \bm{e} \right) \right).
\end{align*}
Then, the conclusion follows by Remark~\ref{rmk:prox-opt-cond}.
\end{proof}

\begin{definition}[Section~2.2 in~\cite{beck-2015-conver-alter}]\label{def:prox-grad-mapping}
For a given continuously differentiable convex function \(f\), a closed and
proper convex function \(h\), and a positive constant \(M > 0\), the proximal
gradient mapping is defined as
\begin{align*}
T_M(x) := \prox_{\frac{1}{M}g} \left( x - \frac{1}{M} \nabla f(x) \right).
\end{align*}
The associated graident mapping is given by
\begin{align*}
G_M(x) = M(x - T_M(x)) = M \left( x - \prox_{\frac{1}{M}g} \left[ x - \frac{1}{M} \nabla f(x) \right] \right).
\end{align*}
\end{definition}

\begin{proposition}[Lemma~2.1 in~\cite{beck-2015-conver-alter}]\label{prop:prox-equiv-def}
Let \(h\) be a proper, closed, and convex function, and let \(M > 0\). Then
\begin{align*}
w = \prox_{\frac{1}{M}h}(x) \iff h(u) \geq h(w) + M (x - w)^{\top}(u - w), \quad \forall u \in \dom h.
\end{align*}
\end{proposition}

\begin{proposition}[Lemma~2.2 in~\cite{beck-2015-conver-alter}]\label{prop:local-quad-ubd}
Let \(f(x_1, x_2)\) be a closed and proper convex function. Also suppose
that the gradient of \(f\) is (uniformly) Lipschitz continuous with respect to
\(x_1\) with constant \(L_1\), that is, for any \(x_1, x_2, d\),
\begin{align*}
\norm{\nabla_{x_1} f(x_1 + d, x_2) - \nabla_{x_1} f(x_1, x_2)}_2
\leq L_1 \norm{d}_2.
\end{align*}
Then, it holds that
\begin{align*}
f(x_1 + d, x_2) \leq f(x_1, x_2) + \nabla_{x_1} f(x_1, x_2)^{\top} d + \frac{L_1}{2} \norm{d}_2^2.
\end{align*}
\end{proposition}

\begin{proposition}[Section~2.4 in~\cite{beck-2015-conver-alter}]\label{prop:prox-suff-decrease}
Suppose that \(f\) is continuously differentiable with Lipschitz gradient with
constant \(L > 0\), and that \(h\) is a closed, proper, and convex function.
Then, \(F(x) := f(x) + h(x)\) satisfies
\begin{align*}
F(x) - F \left( \prox_h \left( x - \frac{1}{L} \nabla f(x) \right) \right)
\geq \frac{1}{2L} \norm[\Big]{L \left(x - \prox_h \left( x - \frac{1}{L} \nabla f(x) \right)\right)}_2^2.
\end{align*}
\end{proposition}

\end{SecondUpdate}

Finally, we review the (reverse) Minkowski's inequality.
\begin{proposition}[Minkowski's Inequality; see Theorem~\(9\) in Chapter~\(3\) of~\cite{bullen-2013-handb}]
For \(p > 1\) and \(a_i, b_i \in \reals_+\) for all \(i \in [n]\), the following holds:
\begin{gather}
\left( \sum_{i = 1}^n (a_i + b_i)^p \right)^{1/p}
\leq \left( \sum_{i = 1}^n a^p_i \right)^{1/p} + \left( \sum_{i = 1}^n b^p_i \right)^{1/p}. \nonumber %
\end{gather}
If \(p < 1\) and \(p \neq 0\), then the inequality holds with the inequality sign reversed.
\end{proposition}

An implication of the Minkowski's inequality is as follows.
\begin{lemma}\label{lem:v-shift-of-alpha-cve}
If the function \(f \colon \reals^n \to \reals_+\) is an
\(\alpha\)-concave function with \(\alpha \in \overline{\reals}\)
and \(c \in \reals_+\) is a constant, then \(g(x) := f(x) - c\) is
\(\alpha\)-concave on \(D := \set{x \in \reals^n \colon f(x) > c}\).
\end{lemma}
\begin{proof}
When \(\alpha \geq 1\), the result follows from
Proposition~\ref{prop:sum-of-alpha-concave-funcs}. When \(\alpha = 0\),
the result was proved in~\cite{boyd-2004-convex} (see Exercise~\(3.48\)). When \(\alpha = -\infty\), shifting the function along the vertical direction does not affect the convexity of its super level sets. Hence, it suffices to prove the result when
\(\alpha < 1\) and \(\alpha \neq 0\).

We notice that \(D\) is convex
as it is the super-level set of the quasi-concave function \(f\). Now,
for any \(x_1, x_2 \in D\) and \(\theta \in (0, 1)\), the following
holds for \(x_{\theta} := \theta x_1 + (1 - \theta) x_2\):
\begin{align}
  f(x_{\theta}) \geq \Big( \theta \cdot (f(x_1))^{\alpha} + (1 - \theta) \cdot (f(x_2))^{\alpha} \Big)^{1 / \alpha}. \label{eq:minus-const-preserve-alpha-1}
\end{align}
By Minkowski's Inequality with \(p\) set to be \(\alpha\), we have
\begin{align*}
  \left( \left[ \theta^{1/\alpha} \cdot f(x_1) \right]^{\alpha} + \left[ (1 - \theta)^{1/\alpha} \cdot f(x_2) \right]^{\alpha} \right)^{1/\alpha}
  & \geq \left( \left[ \theta^{1/\alpha} \cdot (f(x_1) - c) \right]^{\alpha}
    + \left[ (1 - \theta)^{1/\alpha} \cdot (f(x_2) - c) \right]^{\alpha} \right)^{1/\alpha} \\
  & \phantom{\geq} + \left( \left[ \theta^{1/\alpha} \cdot c \right]^{\alpha}
    + \left[ (1 - \theta)^{1/\alpha} \cdot c \right]^{\alpha} \right)^{1/\alpha},
\end{align*}
from which we obtain
\begin{align}
  \Big( \theta \cdot (f(x_1) - c)^{\alpha} + (1 - \theta) \cdot (f(x_2) - c)^{\alpha} \Big)^{1/\alpha}
  + c \leq \Big( \theta \cdot (f(x_1))^{\alpha} + (1 - \theta) \cdot (f(x_2))^{\alpha} \Big)^{1/\alpha}. \label{eq:minus-const-preserve-alpha-2}
\end{align}
Combining~\eqref{eq:minus-const-preserve-alpha-1}
and~\eqref{eq:minus-const-preserve-alpha-2} concludes the proof:
\begin{align*}
  f(x_{\theta}) - c
  & \geq \Big( \theta \cdot (f(x_1))^{\alpha} + (1 - \theta) \cdot (f(x_2))^{\alpha} \Big)^{1 / \alpha} - c \\
  & \geq \Big( \theta \cdot (f(x_1) - c)^{\alpha} + (1 - \theta) \cdot (f(x_2) - c)^{\alpha} \Big)^{1/\alpha}.
\end{align*}
\end{proof}

\section{Proofs} \label{apx-proofs}

\subsection{Proof of Lemma~\ref{lem:reform-neg-cvar}}\label{apx-lem:reform-neg-cvar}
\begin{proof}
By definition of \CVaR{}, we have
\begin{align*}
  \CVaR_{1 - \epsilon} \left( X^- \right)
  = \VaR_{1 - \epsilon} \left( X^- \right)
  + \frac{1}{\epsilon} \cdot \Expect \left[ X^- - \VaR_{1 - \epsilon} (X^-) \right]^+.
\end{align*}
We discuss two cases:
\begin{enumerate}[label=(\roman*)]
\item If \(0 < \VaR_{1 - \epsilon}(X)\), then
      \(\VaR_{1 - \epsilon}(X^-) = 0\), from which
      \[\CVaR_{1 - \epsilon}(X^-) = 0 + \frac{1}{\epsilon} \cdot \Expect \left[ X^- - 0 \right]^+ = 0.\]
\item If \(0 \geq \VaR_{1 - \epsilon}(X)\), then
      \(\VaR_{1 - \epsilon}(X^-) = \VaR_{1 - \epsilon}(X)\). It follows that
      \begin{align*}
        \Expect \left[ X^- - \VaR_{1 - \epsilon}(X) \right]^+
        & = \Expect \left[ \left( X^- - \VaR_{1 - \epsilon}(X) \right)
          \cdot \Ind{X^- \geq \VaR_{1 - \epsilon}(X)} \right] \\
        & = \Expect \left[ \left( X - \VaR_{1 - \epsilon}(X) - X^+ \right)
          \cdot \Ind{X \geq \VaR_{1 - \epsilon}(X)} \right] \\
        & = \Expect \left[ \left( X - \VaR_{1 - \epsilon}(X) \right)
          \cdot \Ind{X \geq \VaR_{1 - \epsilon}(X)} \right] \\
        & \phantom{ = }
          - \Expect \left[ X^+ \cdot \Ind{X \geq \VaR_{1 - \epsilon}(X)} \right] \\
        & =  \Expect \left[ \left( X - \VaR_{1 - \epsilon}(X) \right)
          \cdot \Ind{X \geq \VaR_{1 - \epsilon}(X)} \right]
          - \Expect \left[ X^+ \right] \\
        & =  \Expect \left[ X - \VaR_{1 - \epsilon}(X) \right]^+
          - \Expect \left[ X^+ \right],
      \end{align*}
      where the first equality is by definitions of positive part \([\cdot]^+\) and \(\Ind{\cdot}\), the second is due to
      \(\VaR_{1 - \epsilon}(X) \leq 0\) and the definitions of positive and negative parts, and the fourth is because \(X < \VaR_{1 - \epsilon}(X)\) implies \(X^+ = 0\). We conclude the proof by noticing that
\begin{align*}
\VaR_{1 - \epsilon} \left( X^- \right)
  + \frac{1}{\epsilon} \cdot \Expect \left[ X^- - \VaR_{1 - \epsilon} (X^-) \right]^+ & = \VaR_{1 - \epsilon} \left( X \right)
  + \frac{1}{\epsilon} \cdot \Expect \left[ X - \VaR_{1 - \epsilon}(X) \right]^+
  - \frac{1}{\epsilon} \cdot \Expect \left[ X^+ \right] \\
  & = \CVaR_{1-\epsilon}(X) - \frac{1}{\epsilon} \cdot \Expect \left[ X^+ \right].
\end{align*}
\end{enumerate}
\end{proof}

\subsection{Proof of Lemma~\ref{lem:cont-of-dist-to-set}}\label{apx-lem:cont-of-dist-to-set}
\begin{proof}
We denote the set of points whose distance to \(\Safe^c(x)\) is exactly \(y\) by
\begin{align*}
  E := \Set{ \zeta \in \Xi \colon \Dist{\zeta, \Safe^c(x)} = y}.
\end{align*}
We notice that
\(\Dist{\zeta, \Safe^c(x)} = \Dist{\zeta, \cl \Safe^c(x)}\), where \(\cl \Safe^c(x)\) denotes the closure of \(\Safe^c(x)\). Then, by the item~(1)
of~\cite{erdoes-1945-some-remar}, we have \(\mathbf{Leb}(E) = 0\),
which further implies that \(\Prob(E) = 0\) because \(\Prob\) is
absolutely continuous with respect to \(\mathbf{Leb}(\cdot)\) (see
Theorem~\(2.2\) in~\cite{norkin-1993-analy-optim}).

In addition, the Lebesgue measure of the event \(\{\zeta \in \Xi: f(x, \zeta) = y\}\) equals zero because \(a_i(x) \neq 0\) for all \(i \in [m]\setminus I(x)\). It follows that \(f(x, \zeta)\) is atomless because \(\Prob\) is absolutely continuous with respect to \(\mathbf{Leb}(\cdot)\).
\end{proof}

\subsection{Proof of Proposition~\ref{prop:reform-w-cc}} \label{apx-prop:reform-w-cc}
\begin{proof}
First, moving the \CVaR{} term to the~\RHS{}
of~\eqref{eq:cor-pess-cvar-reform-1} yields
\begin{align}
  \delta
  & \leq \Expect \Big[ f(x, \zeta) \cdot \Ind{-f(x, \zeta) \geq \VaR_{1 - \epsilon} \left( -f(x, \zeta) \right)}\Big]
    - \Expect \Big[ f(x, \zeta) \cdot \Ind{-f(x, \zeta) \geq 0} \Big] \nonumber{} \\
  & = \Expect \Big[ f(x, \zeta) \cdot \Ind{\VaR_{1 - \epsilon}\left( -f(x, \zeta) \right) \leq - f(x, \zeta) \leq 0} \Big] \nonumber{}\\
  & = \Expect \Big[ f(x, \zeta) \cdot \Ind{0 \leq f(x, \zeta) \leq \VaR_{\epsilon}\left( f(x, \zeta) \right)} \Big], \label{eq:reform-deter-expect}
\end{align}
where the first equality is because \(f(x, \zeta)\) is atomless and the second equality is because
\(\VaR_{1 - \epsilon} (- X) = -\VaR_{\epsilon}(X)\).
Now, we use the \redmodify{layer cake} representation of nonnegative integrable
functions to further recast the \RHS{} of \eqref{eq:reform-deter-expect} as
\begin{align*}
  & \Expect \Big[ f(x, \zeta) \cdot \Ind{0 \leq f(x, \zeta) \leq \VaR_{\epsilon}\left( f(x, \zeta) \right)} \Big] \\
  = ~ & \int\displaylimits_{\Xi} f(x, \zeta) \cdot \Ind{0 \leq f(x, \zeta) \leq \VaR_{\epsilon}\left( f(x, \zeta) \right)} \dProb{\zeta} \\
  = ~ & \int\displaylimits_{\Xi} \int\displaylimits_{\reals_+}
        \Ind{ t \leq f(x, \zeta) \cdot \Ind{0 \leq f(x, \zeta) \leq \VaR_{\epsilon}\left( f(x, \zeta) \right)}} \dLeb{t} \, \dProb{\zeta} \\
  = ~ & \int\displaylimits_{\Xi} \int\displaylimits_{\reals_+}
        \Ind{ t \leq f(x, \zeta) \leq \VaR_{\epsilon}\left( f(x, \zeta) \right)} \dLeb{t} \, \dProb{\zeta} \\
  = ~ & \int\displaylimits_{\reals_+}
        \Prob \Big[  t \leq f(x, \zeta) \leq \VaR_{\epsilon}\left( f(x, \zeta) \right)\Big] \dLeb{t}, \tag{by the Tonelli's Theorem} \\
  = ~ & \int_0^{\VaR_{\epsilon}\left( f(x, \zeta) \right)}
        \Big(\Prob \left[ f(x, \zeta) \geq t \right] - ( 1 - \epsilon)\Big) \dLeb{t},
\end{align*}
where the first two equalities are by definitions of expectation and
\redmodify{layer cake} representation, respectively. We justify the third equality by arguing
that, for any \(x \in \mathcal{X}^p\) and \(\zeta \in \Xi\),
\begin{gather}
\Ind{ t \leq f(x, \zeta) \cdot \Ind{0 \leq f(x, \zeta) \leq \VaR_{\epsilon}\left( f(x, \zeta) \right)}}
= \Ind{ t \leq f(x, \zeta) \leq \VaR_{\epsilon}\left( f(x, \zeta) \right)} \label{eq:reform-mid-1}
\end{gather}
holds Lebesgue-almost everywhere for \(t \in \reals_+\). We discuss the following three cases:
\begin{enumerate}[label=(\roman*)]
\item If \(\zeta\) makes \(f(x, \zeta) < 0\), then the LHS
      of~\eqref{eq:reform-mid-1} simplifies to \(\Ind{t \leq 0}\), which coincides with the RHS.
\item If \(\zeta\) makes
      \(f(x, \zeta) \in [0, \VaR_{\epsilon}(f(x, \zeta))]\), then the LHS
      of~\eqref{eq:reform-mid-1} simplifies to \(\Ind{t \leq f(x, \zeta)}\), coinciding with the RHS.
\item If \(\zeta\) makes \(f(x, \zeta) > \VaR_{\epsilon}(f(x, \zeta))\), then
      the LHS and RHS of~\eqref{eq:reform-mid-1} simplify to \(\Ind{t \leq 0}\) and
      \(0\), respectively, which differ only at \(t = 0\) for \(t \in \reals_+\).
\end{enumerate}
The last equality is because
\begin{align*}
  \Prob \Big[  t \leq f(x, \zeta) \leq \VaR_{\epsilon}\left( f(x, \zeta) \right)\Big]
  = \Prob \Big[  t \leq f(x, \zeta) \Big] - \Prob \Big[ t \geq \VaR_{\epsilon}\left( f(x, \zeta) \right) \Big]
\end{align*}
when \(t \in [0, \VaR_{\epsilon} f(x, \zeta)]\). This recasts~\eqref{eq:cor-pess-cvar-reform-1} into~\eqref{eq:deter-reform-int-1}.

Second, constraint~\eqref{eq:cor-pess-cvar-reform-2} is equivalent to \(\Prob\big[ f(x, \zeta) \geq 0 \big] \geq 1 - \epsilon\) by definition of \VaR, which can be further recast as
\[
\Prob\Big[ a_i(x)^{\top}\zeta \leq b_i(x), \ \forall i \in [m]\setminus I(x) \Big] \geq 1 - \epsilon
\]
by definition of \(f(x, \zeta)\). For all \(x \in \mathcal{X}^p\) and \(i \in [m]\), we assume without loss of generality that \(b_i(x) \geq 0\) whenever \(a_i(x) = 0\) (because otherwise \(\Prob[A(x)\zeta \leq b(x)] = 0\)), and it holds that \(a_i(x)^{\top}\zeta \leq b_i(x)\) for all \(i \in I(x)\). It follows that~\eqref{eq:cor-pess-cvar-reform-2} is equivalent to~\eqref{eq:deter-reform-int-2}, which completes the proof.
\end{proof}

\subsection{A Generalized Theorem~\ref{thm:drcc-rhs-cvx} For Quasi-Concave Inequalities} \label{apx-thm:drcc-rhs-cvx}
\begin{FirstUpdate}
We generalize Theorem~\ref{thm:drcc-rhs-cvx} to quasi-concave inequalities as follows.
\begin{theorem} \label{thm:drcc-rhs-cvx-general}
Suppose that the reference distribution \(\Prob\) of \(\mathcal{P}\) is
\(\alpha\)-concave with \(\alpha \geq -1\) and
\(h \colon \reals^n \times \reals^m \to \reals\) is quasi-concave. Then, for
\(\delta > 0\) the set
\begin{align*}
  \mathcal{X}^p_{\text{R}} = \Set{x \in \reals^n \colon \inf_{\QProb \in \mathcal{P}}
  \QProb \left[ h(x, \xi) \geq 0 \right]
  \geq 1 - \epsilon} %
\end{align*}
is convex and closed.
\end{theorem}
\begin{proof}
Recall that for \(\zeta \in \reals^m\), the distance
\(\Dist{\zeta, \Safe^c(x)}\) to the unsafe set is
\begin{align*}
  \Dist{\zeta, \Safe^c(x)}
  & = \inf_{\xi \in \Xi} \Set{ \norm{\zeta - \xi} \colon (x, \xi) \not\in \mathcal{H}_{\geq 0}} \\
  & = \inf_{\xi \in \Xi} \Set{ \norm{\zeta - \xi} \colon (x, \xi) \not\in \cl{}\left(\mathcal{H}_{\geq 0}\right)},
\end{align*}
where \(\mathcal{H}_{\geq 0} := \set{(x, \xi) \colon h(x, \xi) \geq 0}\)
represents the superlevel set of \(h\) at level \(0\). To see he second equality, we first observe that \(\mathcal{H}_{\geq 0} \subseteq \cl{}\left( \mathcal{H}_{\geq 0} \right)\) and so \(\inf_{\xi \in \Xi} \Set{ \norm{\zeta - \xi} \colon (x, \xi) \not\in \mathcal{H}_{\geq 0}} \leq \inf_{\xi \in \Xi} \Set{ \norm{\zeta - \xi} \colon (x, \xi) \not\in \cl{}\left(\mathcal{H}_{\geq 0}\right)}\). Second, for any \(\varepsilon > 0\), there exists a \(\xi_0 \in \Xi\) such that
\((x, \xi_0) \not\in \mathcal{H}_{\geq 0}\) and
\begin{align}
  \inf_{\xi \in \Xi} \Set{ \norm{\zeta - \xi} \colon (x, \xi) \not\in \mathcal{H}_{\geq 0}} + 2 \varepsilon
  > \norm{\zeta - \xi_0} + \varepsilon. \label{eq:general-to-joint-1}
\end{align}
We discuss the following two cases.
\begin{enumerate}[label=(\roman*)]
\item If \((x, \xi_0)\) happens to fall out of
\(\cl{}\left( \mathcal{H}_{\geq 0} \right)\) as well, then
\begin{align}
  \norm{\zeta - \xi_0} + \varepsilon
  >
  \inf_{\xi \in \Xi} \Set{ \norm{\zeta - \xi} \colon (x, \xi) \not\in \cl{}\left(\mathcal{H}_{\geq 0}\right)}, \label{eq:general-to-joint-2}
\end{align}
which implies
\begin{align*}
  \inf_{\xi \in \Xi} \Set{ \norm{\zeta - \xi} \colon (x, \xi) \not\in \mathcal{H}_{\geq 0}} + 2 \varepsilon
  > \inf_{\xi \in \Xi} \Set{ \norm{\zeta - \xi} \colon (x, \xi) \not\in \cl{}\left(\mathcal{H}_{\geq 0}\right)}
\end{align*}
for all \(\varepsilon > 0\). It follows that \(\inf_{\xi \in \Xi} \Set{ \norm{\zeta - \xi} \colon (x, \xi) \not\in \mathcal{H}_{\geq 0}} \geq \inf_{\xi \in \Xi} \Set{ \norm{\zeta - \xi} \colon (x, \xi) \not\in \cl{}\left(\mathcal{H}_{\geq 0}\right)}\).
\item If \((x, \xi_0) \in \cl{}\left( \mathcal{H}_{\geq 0} \right)\),
then
\((x, \xi_0) \in \bd{} \left( \mathcal{H}_{\geq 0}\right)\). Hence, any open ball centered around \((x, \xi_0)\) has to intersect
with
\((\reals^n \times \Xi) \setminus \cl{}\left( \mathcal{H}_{\geq 0} \right)\).
As a result, there exists an
\((x, \xi_0^{\prime}) \not\in \cl{}\left( \mathcal{H}_{\geq 0} \right)\)
such that  \(\norm{\xi_0 - \xi_0^{\prime}} \leq \varepsilon / 2\) and
\begin{align}
  \norm{\zeta - \xi_0} + \varepsilon
  & = \norm{\zeta - \xi_0} + \varepsilon / 2 + \varepsilon / 2
  \geq \norm{\zeta - \xi_0} + \norm{\xi_0 - \xi_0'} + \varepsilon / 2
  \geq \norm{\zeta - \xi_0^{\prime}} + \varepsilon / 2 \nonumber{}\\
  & > \inf_{\xi \in \Xi} \Set{ \norm{\zeta - \xi} \colon (x, \xi) \not\in \cl{}\left(\mathcal{H}_{\geq 0}\right)}.
    \label{eq:general-to-joint-3}
\end{align}
Combining~\eqref{eq:general-to-joint-1} and~\eqref{eq:general-to-joint-3} yields
\begin{align*}
  \inf_{\xi \in \Xi} \Set{ \norm{\zeta - \xi} \colon (x, \xi) \not\in \mathcal{H}_{\geq 0}} + 2 \varepsilon
  > \inf_{\xi \in \Xi} \Set{ \norm{\zeta - \xi} \colon (x, \xi) \not\in \cl{}\left(\mathcal{H}_{\geq 0}\right)}
\end{align*}
for all \(\varepsilon > 0\). It follows that \(\inf_{\xi \in \Xi} \Set{ \norm{\zeta - \xi} \colon (x, \xi) \not\in \mathcal{H}_{\geq 0}} \geq \inf_{\xi \in \Xi} \Set{ \norm{\zeta - \xi} \colon (x, \xi) \not\in \cl{}\left(\mathcal{H}_{\geq 0}\right)}\).
\end{enumerate}
Because \(\cl{} \left( \mathcal{H}_{\geq 0} \right)\) is a closed convex set, it
can be represented as the intersection of hyperplanes:
\begin{align*}
\cl{} \left( \mathcal{H}_{\geq 0} \right) = \Set{(x, \xi) \in \reals^n \times \Xi \colon a_i^{\top} \xi \leq b_i(x), \quad \forall i \in \mathcal{I}},
\end{align*}
where \(\mathcal{I}\) is a (possibly infinite) index set, and for any \(i \in \mathcal{I}\), \(a_i \in \reals^m\) is independent of \(x\) while \(b_i(x)\) is an affine function of \(x\). In other words, we have recast \(\mathcal{X}^p_{\text{R}}\) as
\begin{align*}
\mathcal{X}^p_{\text{R}} = \Set{x \in \reals^n \colon \inf_{\QProb \in \mathcal{P}}
  \QProb \left[ a_i^{\top} \xi \leq b_i(x), \ \forall i \in \mathcal{I} \right]
  \geq 1 - \epsilon}, %
\end{align*}
which coincides with the chance constraint with linear inequalities in Theorem~\ref{thm:drcc-rhs-cvx}. Therefore, the conclusion follows from the proof of Theorem~\ref{thm:drcc-rhs-cvx}.
\end{proof}
\end{FirstUpdate}

\subsection{Proof of Lemma~\ref{lem:rhs-var-concave}} \label{apx-lem:rhs-var-concave}
\begin{proof}
We show that the hypograph of \(\VaR_{1 - \epsilon}\left(f(x, {\zeta})\right)\), i.e.,
\begin{align*}
  \mathcal{H} := \Set{(x, \theta) \colon \VaR_{1 - \epsilon} \left( f(x, {\zeta}) \right) \geq \theta}
\end{align*}
is convex. To this end, we note that
\begin{align*}
  \VaR_{1 - \epsilon} \left( f(x, {\zeta}) \right) \geq \theta
  & \iff \Prob \Set{f(x, {\zeta}) \leq \theta} \leq 1 - \epsilon
    \iff \Prob \Set{f(x, {\zeta}) - \theta \geq 0} \geq \epsilon
\end{align*}
where both equivalences are because \(f(x, {\zeta})\) is atomless. Since \(f(x, \zeta) - \theta\) is jointly concave in
\((x, \zeta, \theta)\) and \(\Prob\) is \(\alpha\)-concave,
\(\Prob \Set{f(x, {\zeta}) - \theta \geq 0}\) is
\(\alpha\)-concave in \((x, \theta)\) on the set
\begin{align*}
  \mathcal{H}' := \Set{(x, \theta) \colon \exists \; \zeta \text{ such that } f(x, \zeta) - \theta \geq 0}
\end{align*}
by Proposition~\ref{prop:cc-rhs-constr-log-concave}. Now, since \(\mathcal{H} \subseteq \mathcal{H}'\), \(\Prob \Set{f(x, {\zeta}) - \theta \geq 0}\) is also
\(\alpha\)-concave on \(\mathcal{H}\) and \(\mathcal{H}\) is convex because it is a super level set of \(\Prob \Set{f(x, {\zeta}) - \theta \geq 0}\).
\end{proof}

\subsection{Proof of Lemma~\ref{lem:cont-of-prob-fcns}} \label{apx-lem:cont-of-prob-fcns}
\begin{proof}
For any \((\hat{x}, \hat{t}) \in \reals^n \times \reals_+\), consider a sequence \(\set{(x_k, t_k)}_k\) that converges to \((\hat{x}, \hat{t})\) as \(k\) goes to infinity.
Then, for any \(\zeta \in \Xi\) such that \(f(\hat{x}, \zeta) - \hat{t} \neq 0\), we
have
\begin{align*}
  \lim_{k \to \infty} \Ind{f(x_k, \zeta) \geq t_k} = \Ind{ f(\hat{x}, \zeta) \geq \hat{t}}
\end{align*}
because the function \(f(x, \zeta) - t\) is continuous in \((x,t)\). Hence, as a function of
\(\zeta\), \(\Ind{f(x_k, \zeta) \geq t_k} \) converges pointwise to
\(\Ind{f(\hat{x}, \zeta) \geq \hat{t}} \) on the complement of
\[
\Bad{}_0 := \set{\zeta \in \Xi \colon f(\hat{x}, \zeta) = \hat{t}}.
\]
It follows that
\begin{align*}
  \lim_{k \to \infty} \psi(x_k, t_k) + (1 - \epsilon)
  & = \lim_{k \to \infty} \Prob \left[ f(x_k, {\zeta}) \geq t_k \right] \\
  & = \lim_{k \to \infty} \int\limits_{\Xi \setminus \Bad_0} \Ind{\zeta \colon f(x_k, \zeta) \geq t_k} \dProb{\zeta} \\
  & = \int\limits_{\Xi \setminus \Bad_0} \lim_{k \to \infty} \Ind{\zeta \colon f(x_k, \zeta) \geq t_k} \dProb{\zeta} \\
  & = \int \Ind{\zeta \colon f(\hat{x}, \zeta) \geq \hat{t}} \dProb{\zeta}
  = \psi(\hat{x}, \hat{t}) + (1 - \epsilon),
\end{align*}
where the second and fourth equality are because
\(\Leb{}(\Bad_0(x, t)) = 0\), and the third equality is by the dominated
convergence theorem. The continuity of \(\phi\) can be established in
a similar way: let \(\set{(x_k, y_k)}_k\) be a
sequence such that converges to \((\hat{x}, \hat{y})\). Then,
\begin{align*}
  \lim_{k \to \infty} \phi(x_k, y_k)
  & = \int_{\reals_+} \lim_{k \to \infty} \psi(x_k, t) \cdot \Ind{t \leq y_k} \dLeb{t}
    = \int_{\reals_+ \setminus \set{\hat{y}}} \psi(\hat{x}, t) \cdot \lim_{k \to \infty} \Ind{t \leq y_k} \dLeb{t} \\
  & = \int_{\reals_+ \setminus \set{\hat{y}}} \psi(\hat{x}, t) \cdot \Ind{t \leq \hat{y}} \dLeb{t} = \phi(\hat{x}, \hat{y}),
\end{align*}
where the first equality is by the dominated convergence theorem, and the
second equality is because \(\psi\) is continuous and
\(\Ind{t \leq y_k}\) has a limit as \(k \to \infty\) when \(t \neq \hat{y}\). This completes the proof.
\end{proof}

\subsection{Proof of Example~\ref{exa:not-gaussian}} \label{apx-exa:not-gaussian}

\begin{FirstUpdate}
\begin{proof}
Suppose that there exists a worst-case Gaussian distribution $\nu \sim \mathcal{N}(\mu, \sigma)$ for some $\mu$ and $\sigma$. Then, $\nu$ satisfies the following two conditions simultaneously:
\begin{enumerate}[label=(\Alph*)]
\item \(\nu \in \mathcal{P}\), that is,
\begin{equation}
d_W \left( \nu, \Prob_0 \right) = \int\limits_{\reals} \abs{F_{\Prob_0}(y) - F_{\nu}(y)} \dLeb{y} \leq \delta,  \label{ctrd-1}
\end{equation}
where \(\mathbb{P}_0\) denotes the 1-dimensional standard Gaussian distribution, and \(F_{\Prob_0}\) and \(F_{\nu}\) represent the cumulative
distribution functions of \(\mathbb{P}_0\) and \(\nu\),
respectively. The above expression of \(d_W \left( \nu, \Prob_0 \right)\) is because both \(\mathbb{P}_0\) and \(\nu\) are 1-dimensional distributions.
\item \(\nu\) attains the probability bound, i.e., \(\nu\left[ \xi \leq x \right] = \inf_{\Prob \in \mathcal{P}} \Prob \left[ \xi \leq x \right]\). As a result, on the one hand, \(\nu\left[ \xi \leq x \right] \geq 1 - \epsilon\) implies \(x \geq \VaR_{1 - \epsilon}(Y_\nu) = \VaR_{1-\epsilon}(\sigma Y + \mu) = \sigma \VaR_{1 - \epsilon}(Y) + \mu\), where \(Y\) and \(Y_{\nu}\) represent random variables following distributions \(\mathbb{P}_0\) and \(\nu\), respectively. On the other hand, by Corollary~\ref{cor:from-ts-to-single} (Ex) is equivalent to \(x \geq c_p := (\overline{g})^{-1}_{\epsilon}(\delta)\). It follows that
\begin{equation}
c_p = \sigma \VaR_{1 - \epsilon}(Y) + \mu. \label{ctrd-2}
\end{equation}
\end{enumerate}
In what follows, we show that if \(\nu\) satisfies condition (B), then it necessarily violates condition (A), establishing the claim. To this end, we simplify \(d_W(\nu, \Prob_0)\):
\begin{align*}
  d_W(\nu, \Prob_0)
  & = \int\limits_{\reals} \abs{F_{\Prob_0}(y) - F_{\nu}(y)} \dLeb{y} \\
  & = \int\limits_0^1 \Big|F_{\Prob_0}(\VaR_t(Y_{\nu})) - t\Big| \dLeb{(\VaR_t(Y_{\nu}))}
    = \int\limits_0^1 \abs[\bigg]{\int\limits_{-\infty}^{\sigma \VaR_t(Y) + \mu} \dif\Prob_0
    - \int\limits_{-\infty}^{\VaR_t(Y)} \dif\Prob_0} \dLeb{(\VaR_t(Y_{\nu}))} \\
  & = \int\limits_0^1 \abs[\bigg]{ \int\limits_{\VaR_t(Y)}^{\sigma \VaR_t(Y) + \mu} \dif{} \Prob_0} \dLeb{(\VaR_t(Y_{\nu}))}
    = \sigma \int\limits_{\reals} \abs[\bigg]{\int\limits_y^{\sigma y + \mu} \dif{}\Prob_0} \dLeb{y} \\
  & = \sigma \int\limits_{\reals}
    \int\limits_{\reals} \Ind{y \wedge (\sigma y + \mu) \leq t \leq y \vee (\sigma y + \mu)} \dif\Prob_0(t) \dLeb{y} \\
  & = \sigma \int\limits_{\reals}
    \int\limits_{\reals} \Ind{\sigma y \leq \sigma t \text{ or } (\sigma y + \mu) \leq t} \cdot
    \Ind{ \sigma t \leq \sigma y \text{ or } t \leq (\sigma y + \mu) } \dif\Prob_0(t) \dLeb{y} \\
  & = \int\limits_{\reals} \int\limits_{\reals}
    \Ind{ \sigma y \leq \sigma t \vee (t - \mu) } \cdot \Ind{ \sigma t \wedge (t - \mu) \leq \sigma y } \dLeb{(\sigma y)} \dif\Prob_0(t)
  = \int\limits_{\reals} \abs{\mu - (1 - \sigma) t} \dif\Prob_0(t),
\end{align*}
where the second and the fifth equalities are due to the change of variable \(y = \VaR_t(Y_{\nu})\). Now, for any \(t \in (\VaR_{1 - \epsilon}(Y), c_p)\), we have
\begin{gather*}
\mu - (1 - \sigma) t
= (c_p - t) + \sigma (t - \VaR_{(1 - \epsilon)}(Y)) > c_p - t > 0,
\end{gather*}
where the equality uses condition (B), particularly equality~\eqref{ctrd-2}. Then,
\begin{align*}
  d_W(\nu, \Prob_0)
  & = \int\limits_{\reals} \abs{\mu - (1 - \sigma) t} \dif\Prob_0(t)
  \geq \int\limits_{\VaR_{1 - \epsilon}(Y)}^{c_p} (\mu - (1 - \sigma) t) \dif\Prob_0(t)
    > \int\limits_{\VaR_{1 - \epsilon}(Y)}^{c_p} (c_p - t) \dif\Prob_0(t) = \delta,
\end{align*}
where the last equality uses the definition of $c_p$.
It follows that \(\nu\) necessarily violates inequality~\eqref{ctrd-1} and so condition (A).
\end{proof}
\end{FirstUpdate}

\subsection{Proof of Theorem~\ref{thm:ts-pcc-convexity}} \label{apx-thm:ts-pcc-convexity}
\begin{proof}
First, we present a technical fact that connects \(\mathcal{X}^p_{\text{T}}\) with \(\mathcal{C}_\delta\) and provide a proof in Appendix~\ref{apx-fact}.

\paragraph{Fact.} For any \(x \neq 0\), \((x, \ell, u) \in \mathcal{X}^p_\text{T}\) if and only if
\(\left( \frac{\ell}{\norm{x}_{\ast}}, \frac{u}{\norm{x}_{\ast}}\right) \in \mathcal{C}_\delta\).\\

Second, we show that \(\mathcal{C}_\delta\) is convex. Since
\(\Prob_0 \disteq R \cdot e^{\top}_1 U_n\) is
unimodal, its distribution function \(\Phi\) is convex on \((-\infty, 0)\) and concave on
\((0, +\infty)\). %
In addition, \(\mathcal{C}_\delta \subseteq \reals_- \times \reals_+\) because \(\epsilon < \frac{1}{2}\). Then, \(\Phi(u - t) - \Phi(\ell - t)\) is jointly concave in \((u, \ell, t)\), implying that \(g_{\epsilon}(\ell, u)\) is log-concave. It follows that \(\mathcal{C}_\delta\) is convex.

Therefore, to prove that \(\mathcal{X}^p_{\text{T}}\) is convex, it remains to
show that \((x, \ell, u) \in \mathcal{X}^p_{\text{T}}\) if and only if there
exists an \(s \geq \norm{x}_{\ast}\) such that
\((\ell, u, s) \in \text{co}(\mathcal{C}_\delta)\). To this end, we discuss
the following two cases:
\begin{enumerate}
\item Suppose that \(x = 0\). For any \((0, \ell, u) \in \mathcal{X}^p_{\text{T}}\), we have
      \(\ell \leq 0 \leq u\) because otherwise
      \(\Prob[\ell \leq 0 \leq u] < 1/2 < 1 - \epsilon\), violating the
      assumption that \((0, \ell, u) \in \mathcal{X}^p_{\text{T}}\). Then, \(s := 1/n\) for a
      sufficiently large integer \(n\) ensures that
      \((\ell/s, u/s) \in \mathcal{C}_\delta\) and so
      \((\ell, u, s) \in \text{co}(\mathcal{C}_\delta)\). On the contrary, for any
      \((0, \ell, u) \in \reals^{n+2}\) such that there exists an \(s \geq 0\)
      with \((\ell, u, s) \in \text{co}(\mathcal{C}_\delta)\), by definition of
      \(\text{co}(\mathcal{C}_\delta)\) there exists a sequence
      \(\{(\ell_n, u_n, s_n)\}_{n=1}^{\infty}\) converging to \((\ell, u, s)\)
      such that \(s_n > 0\) and
      \(g_{\epsilon}(\ell_n/s_n, u_n/s_n) \geq \delta\) for all \(n\). Then,
      \(\ell_n < 0\) and \(u_n > 0\) for all \(n\) because otherwise
      \(g_{\epsilon}(\ell_n/s_n, u_n/s_n) = 0 < \delta\). Driving \(n\) to
      infinity yields that \(\ell \leq 0\) and \(u \geq 0\). Hence,
      \((0, \ell, u) \in \mathcal{X}^p_{\text{T}}\).
\item Suppose that \(x \neq 0\). Pick any \((x, \ell, u) \in \mathcal{X}^p_{\text{T}}\), then the above fact implies that
      \(\left(\frac{\ell}{\|x\|_*}, \frac{u}{\|x\|_*} \right) \in \mathcal{C}_\delta\).
      Hence, \(s := \|x\|_* > 0\) ensures that
      \((\ell, u, s) \in \text{co}(\mathcal{C}_\delta)\). On the
      contrary, pick any \((x, \ell, u) \in \reals^{n+2}\) such that
      \(x \neq 0\) and there exists an \(s \geq \|x\|_* > 0\) with
      \((\ell, u, s) \in \text{co}(\mathcal{C}_\delta)\). By
      definition of \(\text{co}(\mathcal{C}_\delta)\), there exists a sequence
      \(\{(\ell_n, u_n, s_n)\}_{n=1}^{\infty}\) converging to
      \((\ell, u, s)\) such that \(s_n > 0\) and
      \(g_{\epsilon}(\ell_n/s_n, u_n/s_n) \geq \delta\) for all \(n\). Then,
\begin{align*}
  g_{\epsilon}\left(\frac{\ell }{\|x\|_*}, \frac{u }{\|x\|_*}\right) \geq & \  g_{\epsilon}\left(\frac{\ell }{s}, \frac{u}{s}\right)
  = \ \lim_{n \rightarrow \infty} g_{\epsilon}\left(\frac{\ell_n}{s_n}, \frac{u_n}{s_n}\right) \geq \delta,
\end{align*}
      where the first inequality is because the function
      \(g_{\epsilon}(\ell, u)\) is nonincreasing in \(\ell\) and nondecreasing
      in \(u\), and the equality is due to the dominated convergence theorem
      (or equivalently, the continuity of \(g_{\epsilon}\)). It follows that
      \(\left(\frac{\ell}{\|x\|_*}, \frac{u}{\|x\|_*} \right) \in \mathcal{C}_\delta\)
      and so \((x, \ell, u) \in \mathcal{X}^p_{\text{T}}\) by the above fact. This
      completes the proof.
\end{enumerate}
\end{proof}

\subsection{Proof of A Fact Connecting \(\mathcal{X}^p_\text{T}\) and \(\mathcal{C}_\delta\)} \label{apx-fact}
\begin{FirstUpdate}
\paragraph{Fact.} For any \(x \neq 0\), \((x, \ell, u) \in \mathcal{X}^p_\text{T}\) if and only if
\(\left( \frac{\ell}{\norm{x}_{\ast}}, \frac{u}{\norm{x}_{\ast}}\right) \in \mathcal{C}_\delta\).

\begin{proof}
We define a set
\begin{align*}
  \mathcal{X}^p_{\text{T}_0} := \Set{(\ell, u) \in \reals^2 \colon
  \inf_{\QProb \in \mathcal{P}_0} \QProb \left[ \ell \leq \xi \leq u \right] \geq 1 - \epsilon
  },
\end{align*}
where \(\mathcal{P}_0\) is centered around \(\Prob_0 \disteq R\cdot e_1^{\top} U_n\) and has radius \(\delta\). By Proposition~\ref{prop:reform-w-cc}, \(\mathcal{X}^p_{\text{T}_0}\) can be
recast as
\begin{gather}
\Prob_0 \left[ \ell \leq Y \leq u \right] \geq 1 - \epsilon, \label{eq:ts0-deter-reform-int-1}\\
\int\limits_0^{\VaR_{\epsilon}(f_0(\ell,u,Y))} \left(
\Prob_0 \left[ f_0(\ell,u,Y) \geq t \right] - (1 - \epsilon) \right) \dLeb{t} \geq \delta, \label{eq:ts0-deter-reform-int-2}
\end{gather}
where \(Y\) has distribution \(\Prob_0\), and
\(f_0(\ell, u, Y) := \min \set{Y - \ell, u - Y}\). We simplify constraint~\eqref{eq:ts0-deter-reform-int-2}:
\begin{align*}
  \int\limits_0^{\VaR_{\epsilon}(f_0(\ell,u,Y))} \left(
  \Prob_0 \left[ f_0(\ell,u,Y) \geq t \right] - (1 - \epsilon) \right) \dLeb{t}
  & = \int\limits_0^{+\infty} \left[
    \Prob_0 \left[ f_0(\ell,u,Y) \geq t \right] - (1 - \epsilon) \right]^+ \dLeb{t} \\
  & = \int\limits_0^{+\infty} \left[
    \Phi(u - t) + \Phi(\ell + t) - (1 - \epsilon) \right]^+ \dLeb{t},
\end{align*}
where the first equality is because the integrand is decreasing in \(t\) and the second equality is by definition of \(\Phi\). Because
constraint~\eqref{eq:ts0-deter-reform-int-2} implies that there exists a
\(t \geq 0\) such that \(\Phi(u - t)- \Phi(\ell + t) > (1 - \epsilon)\), or
equivalently,
\(\Prob_0 \left[ \ell + t \leq Y \leq u - t \right] \geq 1 - \epsilon\), we
conclude that
\(\mathcal{X}^p_{\text{T}_0} = \set{(\ell, u) \in \reals^2 \colon \delta \leq g_{\epsilon}(\ell, u)} = \mathcal{C}_\delta\). Hence, it remains to show that, for any \(x \neq 0\), \((x, \ell, u) \in \mathcal{X}^p_\text{T}\) if and only if
\(\left( \frac{\ell}{\norm{x}_{\ast}}, \frac{u}{\norm{x}_{\ast}}\right) \in \mathcal{X}^p_{\text{T}_0}\).

To this end, by Proposition~\ref{prop:reform-w-cc}, \((x, \ell, u) \in \mathcal{X}^p_{\text{T}}\) if and only if it satisfies
\begin{gather}
\Prob \left[ \ell \leq x^{\top} \zeta \leq u \right]
\geq 1 - \epsilon, \label{eq:ts-deter-reform-int-1}\\
\int\limits_0^{\VaR_{\epsilon}(f(\ell,u,\zeta))} \left(
\Prob \left[ f(\ell,u, x^{\top} \zeta)
\geq t \right] - (1 - \epsilon) \right) \dLeb{t} \geq \delta, \label{eq:ts-deter-reform-int-2}
\end{gather}
where \(f(\ell, u, x^{\top} \zeta)\) represents the minimum distance to the
unsafe set and can be rewritten using \(f_0\) and \(Y\):
\begin{align}
f(\ell, u, \zeta)
& := \frac{\min \set{x^{\top} \zeta - \ell, u - x^{\top}\zeta}}{\norm{x}_{\ast}}
= \min \Set{
  \frac{x^{\top}\zeta}{\norm{x}_{\ast}} - \frac{\ell}{\norm{x}_{\ast}},
  \frac{u}{\norm{x}_{\ast}} - \frac{x^{\top}\zeta}{\norm{x}_{\ast}}} \nonumber{}\\
& = f_0 \left( \frac{\ell}{\norm{x}_{\ast}}, \frac{u}{\norm{x}_{\ast}}, \frac{x^{\top} \zeta}{\norm{x}_{\ast}} \right)
  \disteq{} f_0 \left( \frac{\ell}{\norm{x}_{\ast}}, \frac{u}{\norm{x}_{\ast}}, Y \right). \label{eq:ts-rel-two-dist-fcns}
\end{align}
Likewise, we have
\begin{align}
  \Prob \left[ \ell \leq x^{\top} \zeta \leq u \right]
  & = \Prob \left[
    \frac{\ell}{\norm{x}_{\ast}}
    \leq \frac{x^{\top} \zeta}{\norm{x}_{\ast}}
    \leq  \frac{u}{\norm{x}_{\ast}} \right]
    = \Prob_0 \left[
    \frac{\ell}{\norm{x}_{\ast}}
    \leq Y
    \leq \frac{u}{\norm{x}_{\ast}} \right] \geq 1 - \epsilon. \label{eq:ts-rel-two-cc}
\end{align}
Now, take \((x, \ell, u) \in \mathcal{X}^p_{\text{T}}\) with \(x \neq 0\), then
by definition it
satisfies~\eqref{eq:ts-deter-reform-int-1},~\eqref{eq:ts-deter-reform-int-2},
and together with equations~\eqref{eq:ts-rel-two-dist-fcns}
and~\eqref{eq:ts-rel-two-cc} we have that
\((\frac{\ell}{\norm{x}_{\ast}}, \frac{u}{\norm{x}_{\ast}}) \in \mathcal{X}^p_{\text{T}_0}\).
Similarly, if \((x, \ell, u)\) satisfies
\((\frac{\ell}{\norm{x}_{\ast}}, \frac{u}{\norm{x}_{\ast}}) \in \mathcal{X}^p_{\text{T}_0}\),
then~\eqref{eq:ts-rel-two-dist-fcns},~\eqref{eq:ts-rel-two-cc},~\eqref{eq:ts0-deter-reform-int-1},
and~\eqref{eq:ts0-deter-reform-int-2} imply that
\((x, \ell, u) \in \mathcal{X}^p_{\text{T}}\).
\end{proof}
\end{FirstUpdate}

\subsection{Proof of Corollary~\ref{cor:ts-pcc-symm-bds}} \label{apx-cor:ts-pcc-symm-bds}
\begin{FirstUpdate}
\begin{proof}
By Proposition~\ref{prop:reform-w-cc}, \((x, u) \in \mathcal{X}^p_{\text{TS}}\)
if and only if
\begin{gather}
\Prob \left[ \abs[\bigg]{\frac{x^{\top} \zeta}{\norm{x}_{\ast}}} \leq \frac{u}{\norm{x}_{\ast}} \right]
\geq 1 - \epsilon, \label{eq:tss-deter-reform-int-1}\\
\int\limits_0^{+\infty} \left(
\Prob \left[ \abs[\bigg]{\frac{x^{\top} \zeta}{\norm{x}_{\ast}}} \leq \frac{u}{\norm{x}_{\ast}} - t \right] - (1 - \epsilon) \right)^+ \dLeb{t} \geq \delta. \label{eq:tss-deter-reform-int-2}
\end{gather}
Observe that if \((x, u)\) satisfies~\eqref{eq:tss-deter-reform-int-2}, then there exists a \(t > 0\) such that
\begin{align*}
  \Prob \left[ \abs[\bigg]{\frac{x^{\top} \zeta}{\norm{x}_{\ast}}} \leq \frac{u}{\norm{x}_{\ast}} \right]
  \geq \Prob \left[ \abs[\bigg]{\frac{x^{\top} \zeta}{\norm{x}_{\ast}}}
  \leq \frac{u}{\norm{x}_{\ast}} - t \right]
  \geq 1 - \epsilon.
\end{align*}
Therefore,~\eqref{eq:tss-deter-reform-int-1} is redundant
and we recast~\eqref{eq:tss-deter-reform-int-2} as
\begin{align*}
  \mathcal{X}^p_{\text{TS}}
  & = \Set{
  (x, u) \in \reals^n \times \reals_+ \colon
  g^s_{\epsilon}\left(\frac{u}{\norm{x}_{\ast}}\right) \geq \delta
  } \\
  & = \Set{
  (x, u) \in \reals^n \times \reals_+ \colon
  u \geq \norm{x}_{\ast} \cdot \inf_r \set{r \geq 0 \colon g^s_{\epsilon}(r) \geq \delta}
  },
\end{align*}
where the second equality is because \(g^s_{\epsilon}(\cdot)\) is increasing.
\end{proof}
\end{FirstUpdate}

\subsection{Proof of Proposition~\ref{prop:algo-reform}} \label{apx-prop:algo-reform}
\begin{proof}
We first show the \(\alpha^*_1\)-concavity of \(\phi(x, y)\) using a similar argument as in the proof of Theorem~\ref{thm:drcc-rhs-cvx}. Recall that \(\psi(x, t) = \Prob[f(x, \zeta) \geq t] - (1- \epsilon)\) and \(\phi(x, y) = \int_0^y\psi(x, t)\dLeb{t}\). Pick any \((x_0, y_0), (x_1, y_1) \in \dom\phi\), then their midpoint \((x_{1/2}, y_{1/2}) := \frac{1}{2}(x_0, y_0) + \frac{1}{2}(x_1, y_1)\) lies in \(\dom\phi\) because \(\dom\phi\) is convex by Lemma~\ref{lem:rhs-var-concave}. Define \(S_i = [0, y_i]\) and pick any \(t_i \in S_i\) for \(i = 0,1\). Since \(\psi(x, t)\) is \(\alpha\)-concave by Lemma~\ref{lem:v-shift-of-alpha-cve}, it holds that
\begin{align*}
  \psi(x_{1/2}, t_{1/2})
  \geq m_{\alpha} \left[ \psi(x_0, t_0), \psi(x_0, t_0); \frac{1}{2} \right].
\end{align*}
It follows from Proposition~\ref{prop:brunn-minkow-ineq} that
\begin{align*}
  \int_{\frac{1}{2}S_0 + \frac{1}{2} S_1} \psi(x_{1/2}, t) \dLeb{t}
  \geq m_{\alpha^{\ast}_1} \left[
  \int_{S_0} \psi(x_0, t) \dLeb{t}, \int_{S_1} \psi(x_1, t) \dLeb{t}; \frac{1}{2}
  \right],
\end{align*}
or equivalently, \(\phi(x_{1/2}, y_{1/2}) \geq m_{\alpha^*_1}[\phi(x_0,y_0),\phi(x_1,y_1); 1/2]\). This shows the midpoint \(\alpha^*_1\)-concavity of \(\phi(x,y)\), which together with its continuity (see Lemma~\ref{lem:cont-of-prob-fcns}) shows the \(\alpha^*_1\)-concavity.

Second, the closedness of \(\dom\phi\) follows from the continuity of \(\psi\) by Lemma~\ref{lem:cont-of-prob-fcns}.

Third, we show that constraints~\eqref{eq:algo-ref-1}--\eqref{eq:algo-ref-2} are equivalent to~\eqref{eq:algo-ref-3}. To this end, we pick any \(x\) that satisfies~\eqref{eq:algo-ref-1}--\eqref{eq:algo-ref-2}. Then, by letting \(y := \VaR_{\epsilon}\big(f(x, \zeta)\big) \geq 0\), we obtain \(\delta \leq \phi(x, y)\), which implies constraint~\eqref{eq:algo-ref-3}. On the contrary, pick any \(x\) that satisfies~\eqref{eq:algo-ref-3}. Then, by definition there exists a \(y \geq 0\) such that \(\delta \leq \phi(x, y)\). Since \(\delta > 0\) and \(\phi(x, y) = \int_0^y \big(\Prob\big[f(x, \zeta) \geq t\big] - (1 - \epsilon)\big) \dLeb{t}\), there exists a \(t \in [0, y]\) such that \(\Prob\big[f(x, \zeta) \geq t\big] \geq (1 - \epsilon)\), which implies that \(\Prob\big[f(x, \zeta) \geq 0\big] \geq (1 - \epsilon)\), i.e., constraint~\eqref{eq:algo-ref-2}. Finally, we notice that \(\phi(x, y) \leq \phi\big(x, \VaR_{\epsilon}\big(f(x, \zeta)\big)\big)\) and hence \(\delta \leq \phi\big(x, \VaR_{\epsilon}\big(f(x, \zeta)\big)\big)\), i.e., constraint~\eqref{eq:algo-ref-1}. This completes the proof.
\end{proof}

\subsection{Proof of Theorem~\ref{thm:algo-cvx-converge}} \label{apx-thm:algo-cvx-converge}
\begin{proof}
The proof relies on preparatory Lemmas~\ref{lem:feasible},~\ref{lem:first-order-opt}, and~\ref{lem:cont-dir-derivative-of-cvx}, whose proofs are provided in Appendix~\ref{apx-preparatory-lemmas}.

First, we define set \(S := \dom\phi \cap \{(x, y) \in X \times \reals_+: c^{\top}x \leq u\}\). Then, by compactness of \(X\) and closedness of \(\dom\phi\) (see Proposition~\ref{prop:algo-reform}), \(S\) is compact. Since all iterates \((x_k, y_k)\) lives in \(S\) (see Lemma~\ref{lem:feasible}), \(\{(x_k, y_k)\}_k\) has a limit point \((x^*, y^*) \in S\).

Second, we show that \((x^*, y^*)\) is a first-order local optimal solution to~\eqref{eq:rho-u}, which implies its global optimality due to the log-concavity of \(\phi(x, y)\). To this end, let \(\Delta := (d_x, d_y)\) be an arbitrary tangent direction of \(S\) at \((x^{\ast}, y^{\ast})\). Then, by definition there exists a sequence \(\{(x_{\ell}, y_{\ell})\}_{\ell}\) in \(S\) converging to
\((x^{\ast}, y^{\ast})\) and \(t_{\ell} \searrow 0\) such that
\[
\Delta = \lim_{\ell \to \infty} \frac{(x_{\ell}, y_{\ell}) - (x^{\ast}, y^{\ast})}{t_{\ell}}.
\]
Then, we examine the directional derivative of \(\phi(x, y)\) along direction \(\Delta\) to obtain
\begin{align*}
  \phi^{\prime}(x^{\ast}, y^{\ast}; \Delta)
  & = \phi^{\prime}\left(x^{\ast}, y^{\ast};
    \lim_{\ell \to \infty} \frac{1}{t_{\ell}} \Big[ (x_{\ell}, y_{\ell}) - (x^{\ast}, y^{\ast}) \Big] \right) \\
  & = \lim_{\ell \to \infty} \phi^{\prime} \left(x^{\ast}, y^{\ast}; \frac{1}{t_{\ell}} \Big[ (x_{\ell}, y_{\ell}) - (x^{\ast}, y^{\ast}) \Big]\right) \\
  & = \lim_{\ell \to \infty} \frac{1}{t_{\ell}} \phi^{\prime} \Big(x^{\ast}, y^{\ast}; (x_{\ell}, y_{\ell}) - (x^{\ast}, y^{\ast}) \Big)
    \leq 0,
\end{align*}
where the second and third equalities follow from the continuity and positive homogeneity of \(\phi^{\prime}(x^{\ast}, y^{\ast}; \Delta)\) in \(\Delta\), respectively (see Lemma~\ref{lem:cont-dir-derivative-of-cvx}), and the inequality follows from Lemma~\ref{lem:first-order-opt} because \((x^{\ast}, y^{\ast}) + (x_{\ell}, y_{\ell}) - (x^{\ast}, y^{\ast}) = (x_{\ell}, y_{\ell}) \in S\). This completes the proof.
\end{proof}

\subsection{Proofs of Preparatory  Lemmas~\ref{lem:feasible},~\ref{lem:first-order-opt}, and~\ref{lem:cont-dir-derivative-of-cvx}} \label{apx-preparatory-lemmas}
\begin{lemma} \label{lem:feasible}
Let \(\set{(x_k, y_k)}_k\) represent a sequence of iterates produced by Algorithm~\ref{algo:bcm-for-rho-u}. Then, all  iterates are feasible, i.e., \((x_k, y_k) \in S\) for all \(k\). In addition, it holds that
\begin{align}
  \lim_{k \to \infty} \phi(x_k, y_k) = \lim_{k \to \infty} \phi(x_{k+1}, y_k). \nonumber %
\end{align}
\end{lemma}

\begin{proof}
First, recall that \(S \equiv \dom\phi \cap \Set{(x, y) \in X \times \reals_+: c^{\top}x \leq u}\) is compact. Since \(\phi(x, y)\) is continuous by Lemma~\ref{lem:cont-of-prob-fcns}, it is bounded on \(S\). In addition, we notice that by construction the
\(\phi\)-values of the iterates produced by Algorithm~\ref{algo:bcm-for-rho-u} are non-decreasing,
\ie{},
\begin{align}
  0 < \phi(x_1, y_1) \leq \phi(x_2, y_1) \leq \phi(x_2, y_2) \leq \cdots \leq \phi(x_k, y_k) \leq \phi(x_{k+1}, y_k) \leq \cdots \label{eq:algo-monotone-seq}
\end{align}
Hence, this non-decreasing, bounded sequence converges to a finite value. It follows that the two subsequences
\(\set{\phi(x_k, y_k)}_k\) and \(\set{\phi(x_{k+1}, y_k)}_k\) converge to the same limit.

Second, we recall that \((x_1, y_1) \in S\) by construction. For all \(k \geq 2\), \(\phi(x_{k+1}, y_k) > 0\) by~\eqref{eq:algo-monotone-seq}, which implies that there exists a \(t \in [0, y_k]\) such that \(\Prob\big[f(x_{k+1}, \zeta) \geq t\big] > 1 - \epsilon\). Then, \(\Prob\big[f(x_{k+1}, \zeta) \geq 0\big] > 1 - \epsilon\), or equivalently, \(\VaR_{\epsilon} \big(f(x_{k+1}, {\zeta}\big)) > 0\).
It follows that \(y_{k+1} \equiv \VaR_{\epsilon} \big(f(x_{k+1}, {\zeta})\big) \geq 0\) and so \((x_{k+1}, y_{k+1}) \in S\). This completes the proof.
\end{proof}

\begin{lemma} \label{lem:first-order-opt}
Let \((x^{\ast}, y^{\ast})\) represent a limit point of the sequence \(\set{(x_k, y_k)}_k\). Then, it holds that
\begin{align*}
  \phi(x^{\ast} + d_x, y^{\ast}) \leq \phi(x^{\ast}, y^{\ast})
  \quad \text{and} \quad
  \phi(x^{\ast}, y^{\ast} + d_y) \leq \phi(x^{\ast}, y^{\ast})
\end{align*}
for all \(d_x \in \reals^n, d_y \in \reals\) such that
\((x^{\ast} + d_x, y^{\ast}) \in S\) and
\((x^{\ast}, y^{\ast} + d_y) \in S\). In addition, if \((x^{\ast} + d_x, y^{\ast} + d_y) \in S\), then the directional derivative of \(\phi(x, y)\) along \((d_x, d_y)\) satisfies
\begin{align*}
  \phi^{\prime}(x^{\ast}, y^{\ast}; (d_x, d_y)) := \lim_{s \to 0^+} \frac{1}{s} \Big[
  \phi(x^{\ast} + s d_x, y^{\ast} + s d_y) - \phi(x^{\ast}, y^{\ast}) \Big] \leq 0.
\end{align*}
\end{lemma}

\begin{proof}
We split the proof into three parts: the perturbation along \((0, d_y)\), the perturbation along \((d_x, 0)\), and the directional derivative \(\phi^{\prime}(x^{\ast}, y^{\ast}; (d_x, d_y))\). For notation brevity, we assume, by passing to a subsequence if needed, that \(\{(x_k, y_k)\}_k\) converges to \((x^*, y^*)\).

\noindent \textbf{(Perturbation along \((0, d_y)\))} By definition of \((x^*, y^*)\), it holds that
\begin{align*}
  \abs[\Big]{y^{\ast} - \VaR_{\epsilon} \big(f(x^{\ast}, {\zeta})\big)}
  & = \abs[\Big]{ \lim_{k \to \infty} y_{k} - \VaR_{\epsilon}\big(f(\lim_{k \to \infty} x_{k}, {\zeta})\big)} \\
  & = \abs[\Big]{ \lim_{k \to \infty} \left( y_{k} - \VaR_{\epsilon}\big(f(x_{k}, {\zeta})\big) \right)} \\
  & = \lim_{k \to \infty} \abs{\varepsilon_{k}} = 0,
\end{align*}
where the second and third equalities are due to the continuity of \(\VaR_{\epsilon} (f(x, {\zeta}))\) (see Lemma~\ref{lem:rhs-var-concave}) and
\(\abs{\cdot}\), respectively. Therefore,
\(\phi(x^{\ast}, y^{\ast} + d_y) \leq \phi(x^{\ast}, y^{\ast})\)
because \(y^{\ast} = \VaR_{\epsilon} (f(x^{\ast}, {\zeta}))\) is a maximizer of \(\phi(x^*, y)\) for fixed \(x^*\).

\noindent \textbf{(Perturbation along \((d_x, 0)\))} First, suppose that \((x^*+d_x, y^*)\) lies in the interior of \(S\), denoted by \(\text{int}(S)\). Then, since \(\{(x_k, y_k)\}_k\) converges to \((x^*, y^*)\), there exist neighborhoods \(N \subseteq S\) and \(N^d \subseteq S\) of \((x^*, y^*)\) and \((x^*+d_x, y^*)\), respectively, such that \((x_k, y_k) \in N\) and \((x_k + d_x, y_k) \in N^d\) for sufficiently large \(k\). Then, by construction it holds that
\[
\phi(x_k+d_x, y_k) \leq \max_x \phi(x, y_k) \leq \phi(x_{k+1}, y_k) + \varepsilon_k.
\]
Driving \(k\) to infinity yields
\[
\phi(x^*+d_x, y^*) \leq \phi(x^*, y^*)
\]
by continuity of \(\phi\) and Lemma~\ref{lem:feasible}.

Second, suppose that \((x^*+d_x, y^*)\) lies on the boundary of \(S\). Then, for all positive integers \(M\), \((x^*+(1 - 1/M)d_x, y^*) \in \text{int}(S)\) by convexity of \(S\). It follows that \(\phi(x^*+(1 - 1/M)d_x, y^*) \leq \phi(x^*, y^*)\). Driving \(M\) to infinity yields \(\phi(x^*+d_x, y^*) \leq \phi(x^*, y^*)\) by continuity of \(\phi\).

\noindent \textbf{(Directional derivative)} Since \(\phi(x, y)\) is log-concave and
\(\phi(x^{\ast}, y^{\ast}) > 0\), \(\phi\) is
directionally differentiable at \((x^{\ast}, y^{\ast})\) by Lemma~\(2.4\) in~\cite{norkin-1993-analy-optim}. Hence, \(\phi^{\prime}(x^{\ast}, y^{\ast}; (d_x, d_y))\) is well-defined. To compute \(\phi'(x^{\ast}, y^{\ast}; (d_x, d_y))\), we define \(\varphi(x, t) := \Prob \left[ f(x, {\zeta})\geq t \right]\) and recast the finite difference
\begin{align}
  & \phi(x^{\ast} + s d_x, y^{\ast} + s d_y) - \phi(x^{\ast}, y^{\ast}) \nonumber{} \\
  = \;
  & \phi(x^{\ast} + s d_x, y^{\ast} + s d_y) - \phi(x^{\ast} + s d_x, y^{\ast}) + \phi(x^{\ast} + s d_x, y^{\ast}) - \phi(x^{\ast}, y^{\ast}) \nonumber{}\\
  = \;
  & \int^{y^{\ast} + s d_y}_{y^{\ast}} \Big(
    \varphi(x^* + sd_x, t) - (1 - \epsilon) \Big) \dLeb{t}
    + \Big(\phi(x^{\ast} + s d_x, y^{\ast}) - \phi(x^{\ast}, y^{\ast})\Big). \label{eq:algo-direc-diff-1}
\end{align}
For the second term in~\eqref{eq:algo-direc-diff-1}, we have
\[
\lim_{s \to 0^+} \frac{1}{s} \Big[ \phi(x^{\ast} + s d_x, y^{\ast}) - \phi(x^{\ast}, y^{\ast}) \Big] = \phi'(x^*, y^*; (d_x, 0)) \leq 0
\]
because \(\phi(x^*+sd_x, y^*) \leq \phi(x^*, y^*)\) for all sufficiently small \(s > 0\). In what follows, we address the first term in~\eqref{eq:algo-direc-diff-1}. To that end, we notice that \(\varphi(x, t)\) is
log-concave on
\begin{align*}
  \dom \varphi := \Set{ (x, t) \in \reals \times \reals_+ \colon
  \exists \ \zeta \text{ such that } f(x, \zeta) - t \geq 0},
\end{align*}
and \((x^{\ast}, y^{\ast}) \in \text{int}(\dom\varphi)\) because
\begin{align*}
  \Prob \left[ f(x^{\ast}, {\zeta}) - y^{\ast} > 0  \right]
  = \Prob \left[ f(x^{\ast}, {\zeta}) - y^{\ast} \geq 0  \right] \geq 1 - \epsilon,
\end{align*}
which implies that there exists a \(\hat{\zeta} \in \Xi\) such that
\(f(x^{\ast}, \hat{\zeta}) - y^{\ast} > 0\). By continuity of \(f\), we also have
\(f(x^{\prime}, \hat{\zeta}) - y^{\prime} \geq 0\) for all \((x^{\prime}, y^{\prime})\) sufficiently close to \((x^{\ast}, y^{\ast})\). Since
\(\varphi(x^{\ast}, y^{\ast})\) is strictly positive and
\(\ln\varphi(x, t)\) is concave on \(\dom \varphi\), \(\ln\varphi(x, t)\) is locally Lipschitz at \((x^{\ast}, y^{\ast})\), \ie{}, there exist \(M > 0\) and
\(r > 0\) such that
\begin{align*}
  \abs[\Big]{\ln\varphi(x, t) - \ln\varphi(x^{\ast}, y^{\ast})}
  \leq M \norm{ (x - x^{\ast}, t - y^{\ast})}_2 \qquad \forall (x, t) \in \mathcal{B}\big((x^{\ast}, y^{\ast}), r\big),
\end{align*}
where \(\mathcal{B}((x^{\ast}, y^{\ast}), r)\) denotes a Euclidean ball centered around \((x^{\ast}, y^{\ast})\) with radius \(r\). For all \(s > 0\) sufficiently small such that
\(s \cdot \norm{(d_x, d_y)}_2 \leq r / 2\) and all scalar \(t\) such that \(\abs{t - y^{\ast}} < s \abs{d_y} \),
we have
\begin{align*}
  \abs[\big]{\ln\varphi(x^{\ast} + s d_x, t) - \ln\varphi(x^{\ast}, t)}
  & \leq \abs[\big]{\ln\varphi(x^{\ast} + s d_x, t) - \ln\varphi(x^{\ast}, y^{\ast})} + \abs[\big]{\ln\varphi(x^{\ast}, y^{\ast}) - \ln\varphi(x^{\ast}, t)} \\
  & \leq M \norm{(s d_x, t - y^{\ast})}_2 + M \norm{(0, t - y^{\ast})}_2 \\
  & \leq M \norm{(s d_x, s d_y)}_2 + M \norm{(0, s d_y)}_2 \\
  & \leq 2 s M \norm{(d_x, d_y)}_2,
\end{align*}
where the first inequality is because of the triangle inequality, the
second inequality is because \(\ln\varphi_0\) is locally Lipschitz around \((x^{\ast}, y^{\ast})\), and the third inequality is because
\begin{align*}
  \norm{(s d_x, t - y^{\ast})}_2^2
  = \norm{s d_x}^2_2 + \abs{t - y^{\ast}}^2
  < \norm{s d_x}^2_2 + \abs{s d_y}^2 = \norm{(s d_x, s d_y)}_2^2.
\end{align*}
We bound the first term
in~\eqref{eq:algo-direc-diff-1} by discussing the following two cases. First, if \(d_y > 0\), then it holds that
\begin{align*}
  & \int^{y^{\ast} + s d_y }_{y^{\ast}} \Big( \varphi(x^{\ast} + s d_x, t) - (1 - \epsilon) \Big) \dLeb{t}
  = \int^{y^{\ast} + s d_y }_{y^{\ast}} \Big( \exp \big[ \ln\varphi(x^{\ast} + s d_x, t) \big] - (1 - \epsilon) \Big) \dLeb{t} \\
  \leq \
  & \int^{y^{\ast} + s d_y }_{y^{\ast}} \Big( \exp \Big[\ln\varphi(x^{\ast}, t) + 2sM \norm{(d_x, d_y)}_2 \Big] - (1 - \epsilon) \Big) \dLeb{t} \\
  = \
  & \exp \big[2sM \norm{(d_x, d_y)}_2 \big] \int^{y^{\ast} + s d_y }_{y^{\ast}}
    \Big( \varphi(x^{\ast}, t) - (1 - \epsilon) \exp \left[-2sM \norm{(d_x, d_y)}_2 \right] \Big) \dLeb{t} \\
  = \
  & \exp \big[2sM \norm{(d_x, d_y)}_2 \big] \left(\int^{y^{\ast} + s d_y }_{y^{\ast}}
    \big[ \varphi(x^{\ast}, t) - (1 - \epsilon) \big]\dLeb{t} +
    (1 - \epsilon)\big(1 - \exp \left[-2sM \norm{(d_x, d_y)}_2 \right]\big) s d_y \right).
\end{align*}
It follows that
\begin{align*}
  & \lim_{s \to 0^+} \frac{1}{s} \int^{y^{\ast} + s d_y }_{y^{\ast}} \big( \varphi(x^{\ast} + s d_x, t) - (1 - \epsilon) \big) \dLeb{t} \\
  \leq \
  & \left(\lim_{s \to 0^+} \exp \big[ 2sM \norm{(d_x, d_y)}_2 \big]\right) \cdot \left( \phi^{\prime}(x^{\ast}, y^{\ast} ; (0, d_y))
    + \lim_{s \to 0^+} \frac{1}{s} (1 - \epsilon)\big(1 - \exp \left[-2sM \norm{(d_x, d_y)}_2 \right] \big) s d_y  \right)\\
  = \
  & \phi^{\prime}(x^{\ast}, y^{\ast} ; (0, d_y)),
\end{align*}
where the inequality is because
\begin{equation*}
\lim_{s \to 0^+} \frac{1}{s} \int^{y^{\ast} + s d_y }_{y^{\ast}} \big[ \varphi(x^{\ast}, t) - (1 - \epsilon) \big]\dLeb{t} = \lim_{s \to 0^+} \frac{1}{s} \big[\phi(x^{\ast}, y^{\ast} + s d_y) - \phi(x^{\ast}, y^{\ast})\big] = \phi^{\prime}(x^{\ast}, y^{\ast} ; (0, d_y)).
\end{equation*}
Second, if \(d_y < 0\), then it holds that
\begin{align*}
  & \int^{y^{\ast} + s d_y }_{y^{\ast}} \big( \varphi(x^{\ast} + s d_x, t) - (1 - \epsilon) \big) \dLeb{t}
  = \int^{y^{\ast}}_{y^{\ast} + s d_y } \big( - \exp \left[ \ln\varphi(x^{\ast} + s d_x, t) \right] + (1 - \epsilon) \big) \dLeb{t} \\
  \leq \
  & \int^{y^{\ast}}_{y^{\ast} + s d_y } \Big( -\exp \Big[\ln\varphi(x^{\ast}, t) - 2sM \norm{(d_x, d_y)}_2 \Big] + (1 - \epsilon) \Big) \dLeb{t} \\
  = \
  & \exp \big[-2sM \norm{(d_x, d_y)}_2 \big] \int^{y^{\ast}}_{y^{\ast} + s d_y }
    \Big( - \varphi(x^{\ast}, t) + (1 - \epsilon) \exp \left[2sM \norm{(d_x, d_y)}_2 \right] \Big) \dLeb{t} \\
  = \
  & \exp \big[-2sM \norm{(d_x, d_y)}_2 \big] \int^{y^{\ast} + s d_y }_{y^{\ast}}
    \Big( \varphi(x^{\ast}, t) - (1 - \epsilon) \exp \left[2sM \norm{(d_x, d_y)}_2 \right] \Big) \dLeb{t} \\
  = \
  & \exp \big[-2sM \norm{(d_x, d_y)}_2 \big] \left(\int^{y^{\ast} + s d_y }_{y^{\ast}}
    \left[ \varphi(x^{\ast}, t) - (1 - \epsilon) \right]\dLeb{t} +
    (1 - \epsilon)\left(1 - \exp \left[2sM \norm{(d_x, d_y)}_2 \right] \right) s d_y \right),
\end{align*}
where the inequality is because
\(\ln\varphi(x^{\ast} + s d_x, t) \geq \ln\varphi(x^{\ast}, t) -
2sM \norm{(d_x, d_y)}_2\) and that the function \(-\exp(\cdot)\) is monotonically decreasing.
It follows that
\begin{align*}
  & \lim_{s \to 0^+} \frac{1}{s} \int^{y^{\ast} + s d_y }_{y^{\ast}} \big( \varphi(x^{\ast} + s d_x, t) - (1 - \epsilon) \big) \dLeb{t} \\
  \leq \
  & \left(\lim_{s \to 0^+} \exp \big[-2sM \norm{(d_x, d_y)}_2 \big]\right) \cdot \left( \phi^{\prime}(x^{\ast}, y^{\ast} ; (0, d_y))
    + \lim_{s \to 0^+} \frac{1}{s} (1 - \epsilon)\big(1 - \exp \left[2sM \norm{(d_x, d_y)}_2 \right] \big) s d_y  \right)\\
  = \
  & \phi^{\prime}(x^{\ast}, y^{\ast} ; (0, d_y)).
\end{align*}
Finally, applying the above analysis on both terms in~\eqref{eq:algo-direc-diff-1} yields
\begin{align*}
  \phi^{\prime}(x^{\ast}, y^{\ast}; (d_x, d_y))
  & = \lim_{s \to 0^+} \frac{1}{s} \Big[ \phi(x^{\ast} + s d_x, y^{\ast} + s d_y) - \phi(x^{\ast}, y^{\ast}) \Big] \\
  & \leq \phi^{\prime}\big(x^{\ast}, y^{\ast}; (0, d_y)\big) + \phi^{\prime}\big(x^{\ast}, y^{\ast}; (d_x, 0)\big) \leq 0,
\end{align*}
which completes the proof.
\end{proof}

\begin{lemma} \label{lem:cont-dir-derivative-of-cvx}
For all \((x, y) \in \dom\phi\) with \(\phi(x, y) > 0\), the directional derivative \(\phi'(x, y; \Delta)\) at \((x, y)\) along direction \(\Delta\) is continuous and positively homogeneous in \(\Delta\).
\end{lemma}

\begin{proof}
For notation brevity, we denote \(z = (x, y)\). Then, it holds that
\begin{align*}
  \phi^{\prime}(z; \Delta)
  & = \lim_{s \to 0^+} \frac{1}{s} \big[ \phi(z + s \Delta) - \phi(z) \big] \\
  & = \lim_{s \to 0^+} \left\{ \frac{
    \exp(\ln\phi(z + s \Delta)) - \exp(\ln\phi(z))}{\ln \phi(z + s \Delta) - \ln \phi(z)} \cdot
    \frac{\ln \phi(z + s \Delta) - \ln \phi(z)}{s} \right\} \\
  & = \phi(z) \lim_{s \to 0^+} \frac{\ln \phi(z + s \Delta) - \ln \phi(z)}{s}
    = \phi(z) \cdot (\ln \phi)^{\prime}(z; \Delta),
\end{align*}
where the third equality follows from the L'H{\^{o}}pital's rule. Since \((\ln \phi)^{\prime}(z; \Delta)\) is convex and positively homogeneous in \(\Delta\) by Proposition~\(17.2\) in~\cite{bauschke-2011-convex-hilber}, so is \(\phi'(z; \Delta)\). The continuity of \(\phi'(z; \Delta)\) follows from its convexity, which completes the proof.
\end{proof}

\begin{SecondUpdate}
\subsection{Rate of Convergence of Algorithm~\ref{algo:bcm-for-rho-u}}\label{apx-sec:rate-of-converge}

We study the rate of convergence of Algorithm~\ref{algo:bcm-for-rho-u}, which belongs to the class of block alternating minimization algorithms for convex programs. As Remark~\ref{rmk:convergence-cvx} and Example~\ref{exa:counter-3} indicate, even the convergence (let alone the rate of convergence) of such algorithms may rely on non-trivial assumptions of \(\phi(x,y)\), e.g., continuous differentiability or Lipschitz gradient with respect to both \(x\) and \(y\), which are not satisfied in our case.
Nevertheless, in what follows we show that, for fixed \(x\), \(\phi\) is continuously differentiable and has a Lipschitz
gradient with respect to \(y\), which paves a way towards establishing the linear convergence of Algorithm~\ref{algo:bcm-for-rho-u}. We need the following (mild) technical assumptions.
\begin{assumption}[see~(3.8) in~\cite{beck-2015-conver-alter}] \label{asm:level-set}
The level set
\(S_{\geq \phi_1} := \Set{(x, y)\in \dom \phi \colon \phi(x, y) \geq \phi(x_1, y_1)}\)
is compact.
\end{assumption}
Similar to~\cite{beck-2015-conver-alter}, we define by \(\overline{R}\) the ``diameter'' of \(S_{\geq \phi_1}\):
\begin{align*}
\overline{R} := \max \Set{\norm{(x_1, y_1) - (x_2, y_2)}_2 \colon (x_i, y_i) \in S_{\geq \phi_1}, i \in [2]}.
\end{align*}

\begin{assumption}\label{asm:pivot-zeta-always-exists}
There exists a \(\zeta_0 \in \Xi\) such that
\begin{align}
f(x, \zeta_0) \geq y \quad \mbox{or equivalently} \quad A \, \zeta_0 \leq b(x) - y \cdot \ones{}, \quad \forall (x, y) \in S_{\geq \phi_1}. \label{eq:rate-assump-pivot}
\end{align}
\end{assumption}
For \((x, y) \in S_{\geq \phi_1}\),
\(\Set{\zeta \colon A \zeta \leq b(x) - y \cdot \ones}\) denotes a polyhedron of
\(\zeta\) whose distance to the unsafe set is \(y\). This polyhedron is the intersection of
halfspaces with (fixed) normal vectors \(a_i^{\top}, i \in [m]\) and intercepts
\(b_i(x) - y, i \in [m]\) parameterized by \((x,y)\). Assumption~\ref{asm:pivot-zeta-always-exists} demands that the intersection of all polyhedra, which are parameterzied by \((x,y) \in S_{\geq \phi_1}\), is non-empty. In other words, there exists a ``core'' \(\zeta_0\) that is distant from the unsafe
set for all choices of \((x,y) \in S_{\geq \phi_1}\). This assumption is always satisfied when \(A\) is an identity matrix, or
more generally, when the column space of \(A\) contains \(\ones\). This is because
the~\RHS{} of~\eqref{eq:rate-assump-pivot} is bounded and for any \(a > 0\) there always exists a \(\zeta_0\) such
that \(A \zeta_0 \leq - a \cdot \ones{}\). In addition, this assumption is
satisfied whenever~\eqref{eq:rate-assump-pivot} represents an intersection of
polytopes (\eg{}, Example~\ref{exa:counter-3}) because \((x,y) \in S_{\geq \phi_1}\) implies \(\mathbb{P}[f(x, \zeta) \geq y] \geq 1 - \epsilon\).
Without loss of generality, we further assume \(\zeta_0 = 0\), with the possibility of
applying a proper translation to \(\Xi\). In this case, it follows that
\(b(x) - y \cdot \ones \geq 0\) for all \((x, y) \in S_{\geq \phi_1}\).

\subsubsection{Differentiability of \(\phi(x, y)\) with respect to \(y\)}%
\label{sec:rate-diff-phi}

\begin{proposition}\label{prop:lip-grad-phi-y}
For fixed \(x\), \(\phi(x, \cdot)\) has a continuous and Lipschitz
gradient, \ie{}, \(\nabla_y \phi(x, y)\) exists and there is a (universal)
\(L > 0\) such that
\begin{align*}
\abs{\nabla_y \phi(x, y_1) - \nabla_y \phi(x, y_2)} \leq L \abs{y_1 - y_2},
\quad \forall x \in X, y_1, y_2 \in \reals_+ \colon (x, y_1), (x, y_2) \in S_{\geq \phi_1}.
\end{align*}
\end{proposition}
The proof of Proposition~\ref{prop:lip-grad-phi-y} relies on the following lemmas.

\begin{lemma}\label{lem:trunc-set-perturb}
For \(x \in \reals^n, R > 0, y \geq 0\) and \(0 \leq \Delta \leq y\), define
\(\mathcal{Z}_y := \Set{\zeta \in \Xi \colon f(x, \zeta) \geq y}\) and
\(\mathcal{B}_r := \set{\zeta \in \Xi \colon \norm{\zeta}_2 \leq r}\) for \(r > 0\).
Then, there exists an \(M_0 > 0\), whose value only depends on the matrix
\(A\), such that
\begin{align}
\left(R\mathcal{B}_1 \cap \mathcal{Z}_{y - \Delta}\right) \subseteq \left(R \mathcal{B}_1 \cap \mathcal{Z}_y \right) + M_0 \Delta \mathcal{B}_1, \label{eq:rate-set-inclusion}
\end{align}
where the \(+\) on the~\RHS{} denotes the Minkowski sum.
\end{lemma}

\begin{proof}%
By the continuity and convexity of \(f\) in \(\zeta\), \(\mathcal{Z}_y\) is
closed and convex; and its intersection with \(R \mathcal{B}_1\) is also closed
and convex. Furthermore, the Minkowski sum of two convex sets,
\(R \mathcal{B}_1 \cap \mathcal{Z}_y\) and \(M_0 \Delta \mathcal{B}_1\), is
also closed and convex, therefore we can exploit the relationship between the
support function and Minkowski sum to prove~\eqref{eq:rate-set-inclusion}. In
particular, we aim to show
\begin{align*}
h(R \mathcal{B}_1 \cap \mathcal{Z}_{y - \Delta}; u)
\leq h(R \mathcal{B}_1 \cap \mathcal{Z}_y; u) + M_0 \Delta h(\mathcal{B}_1; u),
\quad \forall u \in \Xi, \norm{u}_2 = 1.
\end{align*}
We first derive an equivalent reformulation for
\(h(R \mathcal{B}_1 \cap \mathcal{Z}_y; u)\) for \(u \in \Xi\) with unit length.
\begin{align*}
h(R \mathcal{B}_1 \cap \mathcal{Z}_y; u)~
& = - \min_{\zeta} \Set{ -u^{\top}\zeta \colon b(x) - A\zeta \geq y \cdot \ones{}, \norm{\zeta}_2 \leq R } \\
& = - \max_{\pi \geq 0, \norm{q} \leq q_0} \Set{- (b(x) - y \cdot \ones)^{\top} \pi - R q_0 \colon - \pi^{\top}A + q^{\top} = -u^{\top}},
\end{align*}
where the second equality is due to strong duality under the relaxed Slater's
condition~\cite[Theorem~3.2.2]{ben2013optimization}, as \(\zeta = 0\) is always
feasible by Assumption~\ref{asm:pivot-zeta-always-exists}. Because a perturbation on \(y\) only affects the dual objective
function, we can estimate the change in the optimal value by bounding the
optimal dual solution \(\pi\), which indeed exists because \(R \mathcal{B}_1\)
is bounded. To be more specific, observe that
\begin{align*}
h(R \mathcal{B}_1 \cap \mathcal{Z}_{y - \Delta}; u)
& = \min_{\pi \geq 0} \Set{ (b(x) - (y - \Delta) \ones)^{\top} \pi + R \norm{A^{\top} \pi - u}_2} \\
& \leq  (b(x) - (y - \Delta) \ones)^{\top} \pi^{\ast} + R \norm{A^{\top} \pi^{\ast} - u}_2\\
& = h(R \mathcal{B}_1 \cap \mathcal{Z}_y; u) + \Delta \cdot \ones^{\top} \pi^{\ast},
\end{align*}
where \(\pi^{\ast}\) is the optimizer of
\(h(R \mathcal{B}_1 \cap \mathcal{Z}_y; u)\). If we can show that there exists
a universal upper bound \(M_0\) for \(\ones^{\top} \pi^{\ast}\) that is
independent from \(x, y, u\) and \(R\), then~\eqref{eq:rate-set-inclusion}
holds.

To this end, we investigate the Karush-Kuhn-Tucker optimality conditions. Let
the optimal primal dual pair be \(\zeta^{\ast}\) and
\(\pi^{\ast}, q^{\ast}, q^{\ast}_0\), respectively, and they satisfy
\begin{align*}
\text{primal feasible:} \quad
& b(x) - A\zeta^{\ast} \geq y \cdot \ones{}, \norm{\zeta^{\ast}}_2 \leq R, \\
\text{dual feasible:} \quad
& \norm{q^{\ast}}_2 \leq q^{\ast}_0, q^{\ast} = A^{\top} \pi^{\ast} - u, \pi^{\ast} \geq 0, \\
\text{complementary slackness:} \quad
& (b(x) - y \cdot \ones{} - A\zeta^{\ast})^{\top} \pi^{\ast} = 0, R q^{\ast}_0 + (\zeta^{\ast})^{\top} q^{\ast} = 0.
\end{align*}
Simplifying the above conditions by removing \(q^{\ast}\), we obtain
\begin{align*}
& b(x) - A \zeta^{\ast} \geq y \cdot \ones{}, \norm{\zeta^{\ast}}_2 \leq R, \\
& (b(x) - y \cdot \ones{} - A\zeta^{\ast})_i \cdot \pi^{\ast}_i = 0, \quad \forall i \in [m], \\
& - (\zeta^{\ast})^{\top} (A^{\top} \pi^{\ast} - u) \geq R \norm{A^{\top} \pi^{\ast} - u}_2, \pi^{\ast} \geq 0.
\end{align*}
To bound the \(1\)-norm of \(\pi^{\ast}\), we denote by
\(I^{\ast} := \set{i \in [m] \colon \pi^{\ast}_i \neq 0}\) the support of
\(\pi^{\ast}\) and discuss two cases:
\begin{enumerate}
\item If \(A^{\top} \pi^{\ast} = u\), then the following linear program finds a
\(\pi^{\ast}\) with the smallest size, which can give a tight bound for \(\ones{}^{\top} \pi^{\ast}\):
\begin{align}
\min_{\pi \geq 0} \Set{\pi^{\top} \ones \colon A^{\top} \pi = u, \pi_i = 0, \forall i \not\in I^{\ast}}
= \min_{\pi_{I^{\ast}} \geq 0} \Set{\pi_{I^{\ast}}^{\top} \ones \colon A_{I^{\ast}}^{\top} \pi_{I^{\ast}} = u}, \label{eq:rate-set-inclu-case-i}
\end{align}
where \(\pi_{I^{\ast}} := [\pi_i]_{i \in I^{\ast}} \in \reals^{\abs{I^{\ast}}}\)
and \(A_{I^{\ast}} \in \reals^{\abs{I^{\ast}} \times n}\) are the restrictions
of \(\pi\) and rows of \(A\) to the active index set \(I^{\ast}\).
\item If \(A^{\top} \pi^{\ast} \neq u\), then \(\norm{\zeta^{\ast}} = R\), and
there exists an \(\alpha > 0\) such that
\(A^{\top} \pi^{\ast} - u = - \alpha \zeta\) because
\begin{align}
R \leq -(\zeta^{\ast})^{\top} \left(\frac{A^{\top} \pi^{\ast} - u}{\norm{A^{\top} \pi^{\ast} - u}_2}  \right) \leq \norm{\zeta^{\ast}}_2 \leq R. \label{eq:rate-set-inclu-case-ii}
\end{align}
Similarly, we can construct a linear program to find a \(\pi^{\ast}\) of a minimal size:
\begin{align*}
\inf_{\pi \geq 0, \alpha > 0} \Set{\pi^{\top} \ones \colon A^{\top} \pi + \zeta^{\ast} \alpha = u, \pi_i = 0, \forall i \not\in I^{\ast}}
= \inf_{\pi_{I^{\ast}} \geq 0, \alpha > 0} \Set{\pi_{I^{\ast}}^{\top} \ones \colon A_{I^{\ast}}^{\top} \pi + \zeta^{\ast} \alpha = u}.
\end{align*}
\end{enumerate}
Combining the two cases, we conclude that the optimal dual variable \(\pi^{\ast}\) is
supported on \(I^{\ast}\), and \(u\) lives in the positive cone of
\(\set{\zeta^{\ast}, a_i, i \in I^{\ast}}\).
Furthermore, by Caratheodory's theorem, we can assume that
\(\set{a_i, i \in I^{\ast}, \zeta^{\ast}}\) are linearly
independent, because otherwise we can extract a linearly independent subset. Let
\(\overline{\zeta} := \zeta^{\ast} / \norm{\zeta^{\ast}}_2\), and we study
\begin{align}
  \min_{\pi_{I^{\ast}} \geq 0, \alpha \geq 0} \Set{\pi_{I^{\ast}}^{\top} \ones \colon A_{I^{\ast}}^{\top} \pi_{I^{\ast}}
  + \overline{\zeta} \alpha = u}, \label{eq:rate-set-inclu-goal}
\end{align}
which generalizes the two cases above. Specifically, we aim to show that its optimal
value is universally bounded for all \(\norm{u}_2 = 1\) and unit vector
\(\overline{\zeta}\), such that constraints in \(I^{\ast}\) are active. To
start with, we denote by \(\mathcal{C}(\overline{\zeta})\) the positive cone
spanned by unit vectors \(\Set{a_i, i \in I^{\ast}, \overline{\zeta}}\), and
consider the following problem
\begin{align}
  \eqref{eq:rate-set-inclu-goal}
  \leq \max_{\pi_{I^{\ast}} \geq 0, \alpha \geq 0, \norm{u}_2 \leq 1} \Set{\pi_{I^{\ast}}^{\top} \ones \colon A_{I^{\ast}}^{\top} \pi_{I^{\ast}}
  + \overline{\zeta} \alpha = u}
  =
  \max_{\pi, \alpha} ~
  & \ones^{\top} \pi, \label{eq:rate-set-inclu-goal-bd1} \\
  \st{} ~
  & \norm{ A^{\top}_{I^{\ast}} \pi + \overline{\zeta} \alpha }_2 \leq 1, \nonumber{} \\
  & \pi \in \reals^{\abs{I^{\ast}}}_+, \alpha \geq 0,  \nonumber{}
\end{align}
~\eqref{eq:rate-set-inclu-goal-bd1} upper bounds~\eqref{eq:rate-set-inclu-goal} because it looks for a unit
vector in \(\mathcal{C}(\overline{\zeta})\) with the largest \(1\)-norm
representation. Further
relaxing~\eqref{eq:rate-set-inclu-goal-bd1} by outer-approximating the unit ball
with its supporting hyperplanes \(a_i^{\top} (\cdot) \leq 1, i \in I^{\ast}\)
and \(\overline{\zeta}^{\top}(\cdot) \leq 1\), we obtain
\begin{align}
  \max_{\pi, \alpha} ~
  & \ones^{\top} \pi , \label{eq:rate-set-inclu-goal-bd2} \\
  \st{} ~
  & A_{I^{\ast}} \left( A^{\top}_{I^{\ast}} \pi + \overline{\zeta} \alpha  \right) \leq 1, \nonumber{} \\
  & \overline{\zeta}^{\top} \left( A^{\top}_{I^{\ast}} \pi + \overline{\zeta} \alpha  \right) \leq 1, \nonumber{} \\
  & \pi \in \reals^{\abs{I^{\ast}}}_+, \alpha \geq 0.  \nonumber{}
\end{align}
Because
\(A_{I^{\ast}} \overline{\zeta} = (b_{I^{\ast}}(x) - y \cdot \ones_{I^{\ast}}) / \norm{\zeta^{\ast}}_2 \geq 0\)
(by Assumption~\ref{asm:pivot-zeta-always-exists}),
we have \(A_{I^{\ast}} \overline{\zeta} \alpha \geq 0\). Then,
\begin{align}
~\eqref{eq:rate-set-inclu-goal-bd2}
& \leq \
\max_{\pi \geq 0, \alpha \geq 0} ~ \Set{
\ones^{\top} \pi \colon
A_{I^{\ast}} A^{\top}_{I^{\ast}} \pi + A_{I^{\ast}}\overline{\zeta} \alpha \leq 1
}, \nonumber{} \\
& \leq \
\max_{\pi \geq 0} ~ \Set{
\ones^{\top} \pi \colon
A_{I^{\ast}} A^{\top}_{I^{\ast}} \pi \leq 1 \label{eq:rate-set-inclu-goal-bd3}
}.
\end{align}
As the rows of \(A_{I^{\ast}}\) are linearly independent,
\(A_{I^{\ast}} A^{\top}_{I^{\ast}}\) is positive definite and the
recession cone in~\eqref{eq:rate-set-inclu-goal-bd3} only contains zero:
\begin{align*}
  \pi \geq 0, A_{I^{\ast}} A^{\top}_{I^{\ast}} \pi \leq 0
  \implies 0 \leq \norm{A^{\top}_{I^{\ast}} \pi}_2^2 = \pi^{\top} A_{I^{\ast}} A^{\top}_{I^{\ast}} \pi \leq 0.
\end{align*}
Therefore, the optimal value of~\eqref{eq:rate-set-inclu-goal-bd3} is finite, providing an universal upper bound
on~\eqref{eq:rate-set-inclu-goal}. Because~\eqref{eq:rate-set-inclu-goal-bd3}
only depends on the active set \(I^{\ast}\) and the total number of active sets
is finite, we conclude that the desired \(M_0\) exists and finish the proof.
\end{proof}

\begin{lemma}
Suppose that \(g \colon \Xi \to \reals\) is a log-concave density function. Then, there
exist \(a > 0\) and \(b \in \reals\) such that
\begin{align*}
g(\zeta) \leq \sum_{i=1}^{M_1} \Ind{ \zeta \in R_i \mathcal{B}_1 },
\end{align*}
where the constants \(R_i := (b - \ln{}(i-1)) / a\) are decreasing in \(i\) and
\(M_1 := \ceil{1 + e^b}\).
\end{lemma}

\begin{proof}
By~\cite[Lemma~1]{cule2010theoretical}, there exist \(a > 0\) and
\(b \in \reals\) such that \(g \leq \exp{} (-a \norm{\zeta}_2 + b)\) on
\(\Xi\). For \(i \in \integers{}\) and \(i \geq 1\), \(M_1\) the largest
integer such that \(R_{M_1} \geq 0\). We bound \(g\) as follows:
\begin{align*}
  g(\zeta)
  & \leq \exp{} (-a \norm{\zeta}_2 + b)
  \leq \int_0^{+\infty} \Ind{ t \leq \ceil{\exp{} (- a \norm{\zeta}_2 + b)} } \dLeb{t} \\
  & = \sum_{i = 1}^{+\infty} \Ind{ t \leq \ceil{\exp{} (- a \norm{\zeta}_2 + b)} }
  \leq \sum_{i = 1}^{+\infty} \Ind{ t \leq \exp{} (- a \norm{\zeta}_2 + b) + 1 } \\
  & = \sum_{i = 1}^{+\infty} \Ind{ \norm{\zeta}_2 \leq R_{i}}
  = \sum_{i = 1}^{M_1} \Ind{ \norm{\zeta}_2 \leq R_{i}},
\end{align*}
where the second inequality is by the \redmodify{layer} cake representation of integrals
and the ceiling operation, and the second equality is by the definition of the \(R_i\)'s.
\end{proof}

\begin{proof}[Proof of Proposition~\ref{prop:lip-grad-phi-y}]
By the definition of \(\phi(x, y)\) and fundamental theorem of calculus
\(\nabla_y \phi(x, y) = \Prob \left[ f(x, \zeta) \geq y \right] - (1 - \epsilon)\),
which is continuous. It remains to show that \(\nabla_y \phi(x, y)\) is
Lipschitz on \(S_{\geq \phi_1}\). For \(0 \leq \Delta \leq y\), we have
\begin{align}
  ~ & \Prob \left[ f(x, \zeta) \geq y - \Delta \right] - \Prob \left[ f(x, \zeta) \geq y \right] \label{eq:rate-lip-bound}\\
 =~ & \int\limits_{\Xi} \Ind{ f(x, \zeta) \in [y, - \epsilon, y) } g(\zeta) \dLeb{\zeta}
 \leq \int\limits_{\Xi} \Ind{ f(x, \zeta) \in [y, - \epsilon, y) } \cdot \sum_{i=1}^{M_1} \Ind{\zeta \in R_i \mathcal{B}_1} \dLeb{\zeta} \nonumber{}\\
 = ~& \sum_{i=1}^{M_1} \Leb{} \left( \left\{ \zeta \in \Xi \colon f(x, \zeta) \in [y- \epsilon, y) \right\} \cap R_i \mathcal{B}_1 \right) \nonumber{}\\
 = ~&\sum_{i=1}^{M_1} \Leb{} \left( \left\{ \zeta \in \Xi \colon f(x, \zeta) \in [y- \epsilon, +\infty) \right\} \cap R_i \mathcal{B}_1 \right)
    - \Leb{} \left( \left\{ \zeta \in \Xi \colon f(x, \zeta) \in [y, +\infty) \right\} \cap R_i \mathcal{B}_1 \right). \nonumber
\end{align}
Define \(\mathcal{Z}_y := \Set{\zeta \in \Xi \colon f(x, \zeta) \geq y}\), and
by Lemma~\ref{lem:trunc-set-perturb}, we can bound the summands
of~\eqref{eq:rate-lip-bound} as follows:
\begin{align*}
  & \Leb{} \left( \mathcal{Z}_{y - \Delta} \cap R_i \mathcal{B}_1 \right)
  - \Leb{} \left( \mathcal{Z}_y \cap R_i \mathcal{B}_1 \right) \\
  \leq ~
  & \Leb{} \left( (\mathcal{Z}_{y} \cap R_i \mathcal{B}_1) + M_0 \Delta \mathcal{B}_1 \right)
  - \Leb{} \left( \mathcal{Z}_y \cap R_i \mathcal{B}_1 \right) \\
  = ~
  & \sum_{j = 0}^{q} (M_0 \Delta)^j {q \choose j} V(\underbrace{(\mathcal{Z}_{y} \cap R_i \mathcal{B}_1), \ldots, (\mathcal{Z}_{y} \cap R_i \mathcal{B}_1)}_{n-j}, \underbrace{\mathcal{B}_1, \ldots, \mathcal{B}_1}_j) - \Leb{} \left( \mathcal{Z}_y \cap R_i \mathcal{B}_1 \right) \\
  = ~
  & \sum_{j = 1}^{q} (M_0 \Delta)^j {q \choose j} V(\underbrace{(\mathcal{Z}_{y} \cap R_i \mathcal{B}_1), \ldots, (\mathcal{Z}_{y} \cap R_i \mathcal{B}_1)}_{n-j}, \underbrace{\mathcal{B}_1, \ldots, \mathcal{B}_1}_j) \\
  \leq ~
  & \sum_{j = 1}^{q} (M_0 \Delta)^j {q \choose j} V(\underbrace{\max{}\{R_i, 1\} \cdot \mathcal{B}_1, \ldots, \max{}\{R_i, 1\} \cdot \mathcal{B}_1}_n) \\
  = ~
  & \sum_{j = 1}^{q} (M_0 \Delta)^j {q \choose j} \cdot \max \set{R_i^n, 1} \cdot \Leb{} (\mathcal{B}_1),
\end{align*}
where the first inequality is by Lemma~\ref{lem:trunc-set-perturb}, the
first equality is by Proposition~\ref{prop:steiner-formula} (see Appendix~\ref{apx-preliminary}), and the second
inequality is because the mixed volume operator \(V(\cdot)\) is monotone. It follows that
\begin{align*}
\eqref{eq:rate-lip-bound}
& \leq \sum_{j = 1}^{q} (M_0 \Delta)^j {q \choose j} \cdot \Leb{} (\mathcal{B}_1) \sum_{i = 1}^{M_1} \max \set{R_i^n, 1},
\end{align*}
where the~\RHS{} is convex as it is a polynomial of \(\Delta\) with nonnegative
coefficients. Furthermore, \(\Delta\) has a compact domain (a restriction of
\(S_{\geq \phi_1}\)), on which the convex~\RHS{} is bounded by a linear term
\(M_3 \Delta\) for some \(M_3 > 0\). Thus, we conclude that
\begin{align*}
\Prob \left[ f(x, \zeta) \geq y - \Delta \right] - \Prob \left[ f(x, \zeta) \geq y \right] \leq M_3 \Delta,
\end{align*}
establishing the Lipschitz continuity of \(\nabla_y \phi(x, y)\).
\end{proof}

\begin{corollary}\label{cor:lip-grad-ln-phi-y}
For fixed \(x\), \(-\ln{} \phi(x, y)\) is continuously differentiable in \(y\)
on \(S_{\geq \phi_1}\) and has a Lipschitz gradient, that is, there exists a (universal)
\(L_1 > 0\) such that
\begin{align*}
\abs{\nabla_y (-\ln{} \phi(x, y_1)) - \nabla_y (-\ln{} \phi(x, y_2)}) \leq L_1 \abs{y_1 - y_2},
\quad \forall x \in X, y_1, y_2 \in \reals_+ \colon (x, y_1), (x, y_2) \in S_{\geq \phi_1}.
\end{align*}
\end{corollary}

\begin{proof}
By the chain rule and \(\phi_1 > 0\), we have
\[
\abs{\nabla_y (- \ln{} \phi(x, y))} = \frac{1}{\abs{\phi(x, y)}} \abs{\nabla_y \phi(x, y)} \leq \frac{1}{\phi_1} L,
\]
where the inequality is because \((x, y) \in S_{\geq \phi_1}\) and the Lipschitz
continuity of \(\nabla_y \phi(x, y)\) shown in
Proposition~\ref{prop:lip-grad-phi-y}. By setting \(L_1\) to be \(L / \phi_1\), we conclude the proof.
\end{proof}

\subsubsection{Linear Convergence Rate}%
\label{sec:rate-linear}

To establish the linear convergence of Algorithm~\ref{algo:bcm-for-rho-u}, we make the following technical assumption on the iterates \(x_{k+1}\) and \(y_{k+1}\).

\begin{assumption} \label{asm:approx-station}
In Step 2 of Algorithm~\ref{algo:bcm-for-rho-u}, the oracle \(\mathcal{O}_u(y_k, \varepsilon_k)\) returns an \(x_{k+1}\) such that it is an approximate stationary point of the function
\[
\Phi(x, y_{k}) := -\ln{} \phi(x, y_{k}) + \chi\set{x \in X}.
\]
That is, there exists an \(\bm{e}\) such that \(\norm{\bm{e}}_2 \leq \gamma_{k} := \varepsilon_{k} / \overline{R}\) and \(\bm{e} \in \partial_x \left( - \ln{}\phi(x_{k+1}, y_{k}) \right)
+ \partial_x \left(\chi \set{x_{k+1} \in X}  \right)\). In addition, in Step 3 of Algorithm~\ref{algo:bcm-for-rho-u}, we find a \(y_{k+1}\) such that \(|y_{k+1} - \VaR_{\epsilon} \big( f(x_{k+1}, {\zeta}) \big)| \leq \varepsilon_k/L_1\), where \(L_1\) is the Lipschitz constant of \(\nabla_y (-\ln{} \phi(x_{k+1}, y))\).
\end{assumption}
Assumption~\ref{asm:approx-station} is standard in the analysis of continuous optimization algorithms. For Step 2, it strengthens the original demand of \(\varepsilon_k\)-optimality on \(\mathcal{O}(y_k, \varepsilon_k)\). Indeed, Assumption~\ref{asm:approx-station} implies that \(x_{k+1}\) is \(\varepsilon_k\)-optimal. %
To see this, we exploit the convexity of \(\Phi\) and bound the optimality gap
\begin{align*}
\Phi(x_{k+1}, y_{k}) - \Phi(x^{\ast}_{k+1}, y_{k})
\leq -\bm{e}^{\top} \left( x^{\ast}_{k+1} - x_{k+1} \right)
\leq \norm{\bm{e}}_2 \cdot \norm{x^{\ast}_{k+1} - x_{k+1}}_2
= \overline{R} \cdot \norm{\bm{e}}_2 \leq \overline{R} \gamma_{k+1},
\end{align*}
where the first inequality is by the first-order characterization of the convex \(\Phi\).
Imposing the exponential function on both sides, we have
\begin{align*}
\phi(x^{\ast}_{k+1}, y_{k}) \leq \phi(x_{k+1}, y_{k})
+ \left( \exp{}(\overline{R} \gamma_{k+1}) - 1 \right) \phi(x_{k+1}, y_{k})
\leq \phi(x_{k+1}, y_{k}) + \varepsilon_{k+1},
\end{align*}
where the last inequality holds if \(\gamma_{k+1}\) is small enough as \(\phi\) is
bounded on \(S_{\geq \phi_1}\). For Step 3, Assumption~\ref{asm:approx-station} is equivalent to that of Theorem~\ref{thm:algo-cvx-converge} up to the Lipschitz constant \(L_1\).
The linear convergence of Algorithm~\ref{algo:bcm-for-rho-u} follows.
\begin{theorem} \label{thm:cvg-rate}
Under Assumptions~\ref{asm:level-set}--\ref{asm:approx-station}, let \(\set{(x_k, y_k)}_{k=1}^{\infty}\) be the sequence generated by
Algorithm~\ref{algo:bcm-for-rho-u} with
\[
\varepsilon_{k+1} \leq \min \Set{ \frac{(\Phi(x_1, y_1) - \Phi(x_2, y_2))}{(3 \cdot 2^{k+1})}, \frac{1}{2} \varepsilon_k }, \quad \forall k \geq 2.
\]
Then, there exists an \(M_4 > 0\) such that, for any iterate \(k \geq 2\), we have
\begin{align*}
\Phi(x_k, y_k) - \Phi^{\ast} \leq
\max \Set{ \frac{4}{ 2 L_1 (k-1) \overline{R}^2}, \frac{(k+1) M_4}{2^{\floor{k / 2}}}},
\end{align*}
where \(\Phi^{\ast}\) denotes the optimal value.
\end{theorem}
Before proving Theorem~\ref{thm:cvg-rate} and to facilitate the analysis, we rewrite
Algorithm~\ref{algo:bcm-for-rho-u} under Assumption~\ref{asm:approx-station} as Algorithm~\ref{algo:bcm-for-rho-u-prime}.

\renewcommand{\thealgocf}{1\everymodeprime}
\begin{algorithm}[!htbp]
\caption{Evaluation of \(\rho(u)\)\label{algo:bcm-for-rho-u-prime}}
\SetKwInOut{Inputs}{Inputs}
\SetKw{Return}{return}
\Inputs{budget \(u\), risk level \(\epsilon\), %
a diminishing sequence \(\set{\varepsilon_k}_k\), %
and an \(x_1\) such that \(y_1 := \VaR_{\epsilon}\big(f(x_1, \zeta)\big) > 0\).%
}
\For{\(k = 1, 2, \ldots\)}{
Find \(y_{k+1}\) such that \(\abs{y_{k+1} - \VaR_{\epsilon} \big( f(x_k, {\zeta}) \big)} \leq \varepsilon_{k+1} / L_1\)\;
Find \(x_{k+1}\) such that it is an approximate stationary point of \(\Phi(x, y_{k+1})\)\;
\If{stopping criterion is satisfied}{
    \Return{\(\phi(x_{k+1}, y_{k+1})\).}
  }
}
\end{algorithm}
\renewcommand{\thealgocf}{1}
\setcounter{algocf}{1}
We first borrow ideas from Lemmas~3.4,~3.5, and~3.6
in~\cite{beck-2015-conver-alter} to prove the following preparatory lemmas.

\begin{lemma}\label{lem:rate-linear-grad-bound}
Let \(\set{(x_k, y_k)}_{k \geq 1}\) be the sequence generated by
Algorithm~\ref{algo:bcm-for-rho-u-prime} and \((x^{\ast}, y^{\ast})\) be an optimal
solution. Then, in any iteration \(k\),
\begin{align*}
\Phi(x_k, y_{k+1}) - \Phi(x^{\ast}, y^{\ast}) \leq \abs{G_{L_1}^1(y_k)} \abs{y^{\ast} - y_k} + 2 \varepsilon_k,
\end{align*}
where %
\(G_{L_1}^1(y_k)\) is the gradient mapping (see
Definition~\ref{def:prox-grad-mapping}) associated with \(- \ln{}\phi(x_k, \cdot)\) and
\(\chi\set{y \geq 0}\).
\end{lemma}

\begin{proof}
We first bound the difference between \(-\ln{} \phi(x_k, y_{k+1})\) and
\(-\ln{} \phi(x^{\ast}, y^{\ast})\).
\begin{align}
& (-\ln{} \phi(x_k, y_{k+1})) - (-\ln{} \phi(x^{\ast}, y^{\ast})) \label{eq:rate-linear-bound-1}\\
= ~
&
(-\ln{} \phi(x_k, y_{k+1})) - (-\ln{} \phi(x_k, y^{\ast}_{k+1})) +
(-\ln{} \phi(x_k, y^{\ast}_{k+1})) - (-\ln{} \phi(x^{\ast}, y^{\ast})) \nonumber{}\\
\leq ~
& \varepsilon_{k+1} + (-\ln{} \phi(x_k, T^1_{L_1}(y_k))) - (-\ln{} \phi(x^{\ast}, y^{\ast})) \nonumber{}\\
\leq ~
& \varepsilon_{k+1} + (- \ln{} \phi(x_k, y_k)) + \nabla_y(-\ln{} \phi(x_k, y_k))^{\top} \left( T^1_{L_1}(y_k) - y_k \right)
+ \frac{L_1}{2} \abs{T^1_{L_1}(y_k) - y_k}^2 - (-\ln{} \phi(x^{\ast}, y^{\ast})), \nonumber{}
\end{align}
where \(T_M^1(y_k)\) is the proximal gradient mapping associated with
\(-\ln{} \phi(x_k, \cdot)\) and \(\chi \set{y \geq 0}\), the first inequality is
because of Corollary~\ref{cor:lip-grad-ln-phi-y} %
and
\((-\ln{} \phi(x_k, y^{\ast}_{k+1})) \leq (-\ln{} \phi(x_k, T^1_{L_1}(y_k)))\), and
the second inequality is by Proposition~\ref{prop:local-quad-ubd}.

Let \(\bm{p}\) be an element in \(\partial_x(-\ln{}\phi(x_k, y_k))\), then by
convexity of \(-\ln{} \phi(x, y)\) we have
\begin{align*}
- \ln{} \phi(x^{\ast}, y^{\ast}) \geq
-\ln{}\phi(x_k, y_k) + \nabla_y (-\ln{} \phi(x_k, y_k))^{\top} (y^{\ast} - y_k)
+ \bm{p}^{\top} (x^{\ast} - x_k),
\end{align*}
from which
\begin{align}
\eqref{eq:rate-linear-bound-1} ~
& \leq \varepsilon_{k+1} + \nabla_y(-\ln{} \phi(x_k, y_k))^{\top} \left( T^1_{L_1}(y_k) - y_k \right)
+ \frac{L_1}{2} \abs{T^1_M(y_k) - y_k}^2 \nonumber{} \\
& \phantom{\leq} ~
+ \nabla_y (-\ln{} \phi(x_k, y_k))^{\top} (y_k - y^{\ast})
+ \bm{p}^{\top} (x_k - x^{\ast}). \label{eq:rate-linear-bound-2}
\end{align}

Because \(x_k\) is approximately stationary, we can choose \(\bm{p}\) such that
\begin{align*}
0 \in \partial_x \left( \chi\set{x \in X} \right) + \bm{p} - \bm{e}.
\end{align*}
By Corollary~\ref{cor:approx-prox-opt-cond} and Definition~\ref{def:prox-grad-mapping},
\begin{gather*}
x_k = \prox_{\frac{1}{L_1} \chi\set{x \in X}} \left( x_k - \frac{1}{L_1} \left( \bm{p} + \bm{e} \right) \right), \\
T^1_{L_1}(y_k) = \prox_{\frac{1}{L_1} \chi\set{y \geq 0}} \left( y_k - \frac{1}{L_1} \nabla_y (-\ln{} \phi(x_k, y_k)) \right).
\end{gather*}
Invoking Proposition~\ref{prop:prox-equiv-def}, we obtain
\begin{align}
\chi\set{ y^{\ast} \geq 0 }
& \geq \chi\set{ T^1_{L_1}(y_k) \geq 0 }
+ L_1 \left( y_k - \frac{1}{L_1} \nabla_y (-\ln{} \phi(x_k, y_k)) - T^1_{L_1}(y_k) \right)^{\top} \left( y^{\ast} - T^1_{L_1}(y_k) \right), \nonumber{}\\
\implies \quad 0 & \geq
L_1 \left( y_k - \frac{1}{L_1} \nabla_y (-\ln{} \phi(x_k, y_k)) - T^1_{L_1}(y_k) \right)^{\top} \left( y^{\ast} - T^1_{L_1}(y_k) \right),
\label{eq:rate-linear-bound-3}
\end{align}
where \(y^{\ast}\) and \(T^1_{L_1}(y_k)\) play the roles of \(u\) and \(w\) in Proposition~\ref{prop:prox-equiv-def}, respectively. Likewise, for \(x_k\), we have
\begin{align}
\chi\set{x^{\ast} \in X}
& \geq \chi\set{x_k \in X} + L_1 \left( x_k - \frac{1}{L_1} \left( \bm{p} + \bm{e} \right) - x_k \right)^{\top} \left( x^{\ast} - x_k  \right), \nonumber{}\\
\implies \quad 0 & \geq L_1 \left( x_k - \frac{1}{L_1} \left( \bm{p} + \bm{e} \right) - x_k \right)^{\top} \left( x^{\ast} - x_k  \right),
\label{eq:rate-linear-bound-4}
\end{align}
where \(x^{\ast}\) plays the role of \(u\) and \(x_k\) plays the role of \(w\).
Combining~\eqref{eq:rate-linear-bound-2},~\eqref{eq:rate-linear-bound-3},
and~\eqref{eq:rate-linear-bound-4} yields
\begin{align*}
\Phi(x_k, y_{k+1}) - \Phi(x^{\ast}, y^{\ast})
& \leq \eqref{eq:rate-linear-bound-2} - L_1 \left( y_k - \frac{1}{L_1} \nabla_y (-\ln{} \phi(x_k, y_k)) - T^1_{L_1}(y_k) \right)^{\top} \left( y^{\ast} - T^1_{L_1}(y_k) \right) \\
& \phantom{\leq (   )} - L_1 \left( x_k - \frac{1}{L_1} \left( \bm{p} + \bm{e} \right) - x_k \right)^{\top} \left( x^{\ast} - x_k  \right) \\
& = \varepsilon_{k+1} + \frac{L_1}{2} \abs{T^1_M(y_k) - y_k}^2 + (x_k - x^{\ast})^{\top} \left( \bm{p} - (\bm{p} + \bm{e}) \right) \\
& \phantom{\leq (   )} + \nabla_y (-\ln{} \phi(x_k, y_k))^{\top} (T^1_{L_1}(y_k) - y^{\ast}) \\
& \phantom{\leq (   )} + L_1 \left( y_k
- \frac{1}{L_1} \nabla_y (-\ln{} \phi(x_k, y_k))
- T^1_{L_1}(y_k) \right)^{\top}(T^1_{L_1}(y_k) - y^{\ast}) \\
& = \varepsilon_{k+1} + \frac{L_1}{2} \abs{T^1_M(y_k) - y_k}^2 - (x_k - x^{\ast})^{\top} \bm{e}
+ L_1 (y_k - T^1_{L_1}(y_k))^{\top} (T^1_{L_1}(y_k) - y^{\ast}).
\end{align*}
Because \(G^1_{L_1}(y_k) := L_1 (y_k - T^1_{L_1}(y_k))\), we arrive at
\begin{align*}
\Phi(x_k, y_{k+1}) - \Phi(x^{\ast}, y^{\ast})
& \leq \varepsilon_{k+1} + \frac{1}{2L_1} G^1_{L_1}(y_k)^2
- (x_k - x^{\ast})^{\top} \bm{e}
+ G^1_{L_1}(y_k)^{\top} (T^1_{L_1}(y_k) - y_k + y_k - y^{\ast})\\
& = \varepsilon_{k+1} + \frac{1}{2L_1} G^1_{L_1}(y_k)^2
- (x_k - x^{\ast})^{\top} \bm{e}
- \frac{1}{L_1} G^1_{L_1}(y_k)^2 + G^1_{L_1}(y_k) (y_k - y^{\ast}) \\
& \leq \varepsilon_{k+1}
+ \norm{(x_k - x^{\ast})}_2 \norm{\bm{e}}_2
+ \abs{G^1_{L_1}(y_k)} \abs{y_k - y^{\ast}} \\
& \leq \varepsilon_{k+1}
+ \overline{R} \gamma_{k+1}
+ \abs{G^1_{L_1}(y_k)} \abs{y_k - y^{\ast}}
= 2\varepsilon_{k+1} + \abs{G^1_{L_1}(y_k)} \abs{y_k - y^{\ast}},
\end{align*}
where the second inequality is by the Cauchy inequality.
\end{proof}

\begin{lemma}\label{lem:rate-lbd-on-decrease}
Let \(\set{(x_k, y_k)}_{k=1}^{\infty}\) be the sequence generated by %
Algorithm~\ref{algo:bcm-for-rho-u-prime} and \((x^{\ast}, y^{\ast})\) be an optimal
solution. Then, in any iteration \(k\),
\begin{align*}
\Phi(x_k, y_k) - \Phi(x_{k+1}, y_{k+1})
\geq (\Phi(x_{k+1}, y_{k+1}) - \Phi(x^{\ast}, y^{\ast}) - 2 \varepsilon_{k+1})^2 / (2 L_1 \overline{R}^2) - \varepsilon_{k+1}.
\end{align*}
\end{lemma}

\begin{proof}
By Lemma~\ref{lem:rate-linear-grad-bound}, we have
\begin{align}
\Phi(x_{k+1}, y_{k+1}) - \Phi(x^{\ast}, y^{\ast})
\leq \Phi(x_k, y_{k+1}) - \Phi(x^{\ast}, y^{\ast})
\leq 2 \varepsilon_{k+1} + \abs{G^1_{L_1}(y_k)} \abs{y_k - y^{\ast}}. \label{eq:rate-linear-grad-bound2}
\end{align}
Then, we can bound the change in objective value in adjacent iterations as follows:
\begin{align*}
\Phi(x_k, y_k) - \Phi(x_{k+1}, y_{k+1})
& \geq
\Phi(x_k, y_k) - \Phi(x_k, y_{k+1}) \\
& =
\Phi(x_k, y_k) - \Phi(x_k, y^{\ast}_{k+1}) - (\Phi(x_k, y_{k+1}) - \Phi(x_k, y^{\ast}_{k+1})) \\
& \geq
\Phi(x_k, y_k) - \Phi(x_k, T^1_{L_1}(y_k)) - \varepsilon_{k+1}
\geq \frac{1}{2L_1} G^1_{L_1}(y_k)^2 - \varepsilon_{k+1} \\
& \geq \frac{\left(\Phi(x_{k+1}, y_{k+1}) - \Phi(x^{\ast}, y^{\ast}) - 2\varepsilon_{k+1}\right)^2}{2L_1 \overline{R}^2} - \varepsilon_{k+1},
\end{align*}
where the third inequality is by Proposition~\ref{prop:prox-suff-decrease}, and
the last inequality is by~\eqref{eq:rate-linear-grad-bound2} and
\(\abs{y_k - y^{\ast}} \leq \overline{R}\).
\end{proof}

\begin{lemma}\label{lem:rate-technical-lemma}
Let \(\set{A_{k}}_{k = 1}^{\infty} \subseteq \reals_+\) be a nonnegative and
monotonically decreasing sequence, \(\eta > 0\), and
\(\set{\beta_k}_{k=1}^{\infty} \subseteq \reals_+\) be a diminishing sequence such that
\(\beta_k \leq 2^{-k} A_1\) for all \(k\). Suppose that
\begin{align*}
A_k - A_{k+1} \geq \eta (A_{k+1} - \beta_{k+1})^2, \quad \forall k \geq 1,
\end{align*}
then
\begin{align*}
A_k \leq \max \Set{ \frac{4}{\eta (k-1)}, \frac{(k+1) A_1}{2^{\floor{k / 2}}}}, \quad \forall k \geq 2.
\end{align*}
\end{lemma}
\begin{proof}
\begin{align*}
\frac{1}{A_{k+1}} - \frac{1}{A_k}
& = \frac{A_k - A_{k+1}}{A_{k+1}A_k}
\geq \frac{\eta (A_{k+1} - \beta_{k+1})^2}{A_{k+1}A_k}
= \frac{\eta (A^2_{k+1} - 2 \beta_{k+1} A_{k+1}+ \beta^2_{k+1})}{A_{k+1}A_k}\\
& \geq \frac{\eta (A^2_{k+1} - 2 \beta_{k+1} A_{k+1})}{A_{k+1}A_k}
= \frac{\eta}{A_k} \left( A_{k+1} - 2 \beta_{k+1} \right).
\end{align*}

Note that if all iterates satisfy \(A_{k+1} - 2 \beta_{k+1} \geq A_k/2\), that
is, \(1/A_{k+1} \geq 1/A_k + \eta / 2\), then \(\set{1/A_k}_{k=1}^{\infty}\)
increases linearly. On the flip side, if all iterates satisfy \(A_{k+1} - 2 \beta_{k+1} < A_k/2\)
and \(\set{\beta_k}_{k=1}^{\infty}\) decrease to zero fast enough, then
\(\set{A_k}_{k=1}^{\infty}\) decreases approximately at an exponential rate,
faster than a linear rate. To effectively combine the two cases, we define
\begin{align*}
\mathcal{K}^n_1 := \Set{k \in \integers{}_+ \colon 1 \leq k \leq n, A_{k+1} - 2 \beta_{k+1} \geq A_k / 2},
\quad \mathcal{K}^n_2 := \Set{k \in \integers{}_+ \colon 1 \leq k \leq n, k \not\in \mathcal{K}_1},
\end{align*}
where \(n\) is an iteration index. If \(n\) is even, we discuss the following two cases:
\begin{enumerate}
\item If \(\abs{\mathcal{K}^n_1} \geq n / 2\), then for any \(k \leq n - 1\),
\begin{align*}
1/A_{k+1} \geq
\begin{cases}
1 / A_k & \text{ if \(k \in \mathcal{K}^n_2\),} \\
1 / A_k + \eta/2 & \text{o.w.}
\end{cases}
\end{align*}
Hence, \(1 / A_n \geq (n / 2) \cdot (\eta / 2) = \eta \cdot n / 4\), implying
that \(A_n \leq 4 / (\eta \cdot n)\).
\item If \(\abs{\mathcal{K}^n_1} < n / 2\), then
\(K := \abs{\mathcal{K}_2^n} \geq n/2\). Without loss of generality, we denote \(\mathcal{K}^n_2 = \set{ j_1, j_2, \ldots, j_K}\). Then,
\begin{align}
A_{k+1} \leq
\begin{cases}
A_k / 2 + 2 \beta_{k+1} & \text{\(k \in \mathcal{K}^n_2\),} \\
A_k & \text{\(k \in \mathcal{K}^n_2\).}
\end{cases} \label{eq:rate-ineq-lemma-1}
\end{align}
So
\(A_n \leq \cdots \leq A_{j_K}/2 + 2 \beta_{j_K + 1} \leq \cdots \leq A_{j_2}/2 + 2 \beta_{j_2 + 1} \leq \cdots A_{j_1}/2 + 2 \beta_{j_1 + 1} \leq \cdots \leq A_1\),
from which we can obtain an upper bound for \(A_n\):
\begin{align*}
A_n
& \leq \cdots \leq \frac{1}{2} A_{j_K} + 2 \beta_{j_K + 1}
\leq \frac{1}{2} \left( \frac{1}{2} A_{j_{K-1}} + 2 \beta_{j_{K-1} + 1}  \right) + 2 \beta_{j_K + 1} \\
& = \frac{1}{2^2} A_{j_{(K-1)}} + \beta_{j_{(K-1)} + 1} + 2 \beta_{j_K + 1}
\leq \frac{1}{2^3} A_{j_{(K-2)}} + \frac{1}{2} \beta_{j_{(K-2)} + 1} + \beta_{j_{(K-1)} + 1} + 2 \beta_{j_K + 1} \leq \cdots \\
& \leq \frac{1}{2^K} A_{j_1} + \sum_{i=1}^K 2^{(1-K+i)} \beta_{j_i + 1}
\leq \frac{1}{2^{n/2}} A_1 + \sum_{i=1}^K 2^{(1-K+i)} \beta_{i+1}, \\
& \leq \frac{1}{2^{n/2}} A_1 + \sum_{i=1}^n 2^{(1-K+i)} \beta_{i+1},
\end{align*}
where the second, third, and fourth inequality are
by~\eqref{eq:rate-ineq-lemma-1}, the second to last inequality is due to
\(K \geq n/2\), \(i \leq j_i\), and the monotonicity of \(\beta_k\)'s. Furthermore,
by the choices of \(\set{\beta_k}_{k=1}^{\infty}\), we have
\begin{align*}
A_n
& \leq \frac{1}{2^{n/2}} A_1 + \sum_{i=1}^n 2^{(1-K+i)} 2^{-(i+1)} \beta_0
\leq \frac{1}{2^{n/2}} A_1 + n2^{-K} A_1
\leq \left( 2^{-n/2} + n2^{-n/2} \right) A_1
= (n+1) 2^{-n/2} A_1.
\end{align*}
\end{enumerate}
Combining the two cases, we have
\(A_n \leq \max\set{ \frac{4}{(\eta n)}, \frac{(n+1)A_1}{2^{n/2}} }\) when \(n\)
is even. When \(n\) is odd, then
\(A_n \leq A_{n-1} \leq \max\set{ \frac{4}{\eta(n-1)}, \frac{n A_1}{2^{(n-1)/2}} }\).
To sum up, for \(n \geq 2\), we have
\begin{align*}
A_n \leq \max \Set{ \frac{4}{\eta (n-1)}, \frac{(n+1) A_1}{2^{\floor{n / 2}}}}.
\end{align*}
\end{proof}
We are now ready to present a proof for Theorem~\ref{thm:cvg-rate}.

\begin{proof}[Proof of Theorem~\ref{thm:cvg-rate}]
We denote by \(A_k := \Phi(x_k, y_k) - \Phi^{\ast} + \varepsilon_k\), then by \(\Phi(x_2, y_2) \geq \Phi^{\ast}\) we obtain
\begin{align*}
\varepsilon_{k+1} \leq \frac{(\Phi(x_1, y_1) - \Phi(x_2, y_2))}{(3 \cdot 2^{k+1})} \leq \frac{(\Phi(x_1, y_1) - \Phi^{\ast})}{(3 \cdot 2^{k+1})}, \quad \forall k \geq 2.
\end{align*}
In addition, by Lemma~\ref{lem:rate-lbd-on-decrease} and
\(2 \varepsilon_{k+1} \leq \varepsilon_k\) we have
\begin{align*}
A_k - A_{k+1}
& = \Phi(x_k, y_k) - \Phi(x_{k+1}, y_{k+1}) + \varepsilon_k - \varepsilon_{k+1} \\
& \geq (A_{k+1} - 3 \varepsilon_{k+1})^2 / (2 L_1 \overline{R}^2) - \varepsilon_{k+1} + \varepsilon_k - \varepsilon_{k+1} \\
& \geq (A_{k+1} - 3 \varepsilon_{k+1})^2 / (2 L_1 \overline{R}^2).
\end{align*}
Finally, note that \(3 \cdot \varepsilon_{k} \leq \left( \Phi(x_1, y_1) - \Phi^{\ast} \right) / 2^k \leq A_1 / 2^k\). For any \(M_4 > A_1\), we have
\begin{align*}
\Phi(x_k, y_k) - \Phi^{\ast}
\leq A_k
\leq \max \Set{ \frac{4}{ 2 L_1 (k-1) \overline{R}^2}, \frac{(k+1) M_4}{2^{\floor{k / 2}}}}, \quad k \geq 2.
\end{align*}
where the second inequality is by Lemma~\ref{lem:rate-technical-lemma}.
\end{proof}
\end{SecondUpdate}

\subsection{Proof of Theorem~\ref{cor:approx-guarantee-for-hierarchy}} \label{apx-cor:approx-guarantee-for-hierarchy}

We first present preparatory propositions~\ref{prop:approx-cvx-g-limit-1},~\ref{prop:approx-bd-for-K1-K2}, and~\ref{prop:approx-guarantee-two-levelsets}. Then, we put them together to prove Theorem~\ref{cor:approx-guarantee-for-hierarchy}.

\begin{definition}
For \(\hat{\mathcal{C}}_N \subseteq \mathcal{C}_{\delta}\), define
\[
  \underline{\delta} :=
  \sup_{(\ell, u) \in \bd{}(\hat{\mathcal{C}}_N)} g_{\epsilon}(\ell, u),
\]
and we say \(\hat{\mathcal{C}}_N\) is \emph{supported} by
\(\mathcal{C}_{\underline{\delta}}\). In addition, with respect to the new origin \((\ell_0, u_0)\), define \(\underline{\delta}^+ := g_{\epsilon}(\ell_0, u_0) > \underline{\delta}\) and
\begin{gather}
g^s_{\epsilon} \colon \reals^2_+ \ni (\Delta \ell, \Delta u)
\mapsto g_{\epsilon}(\ell_0 + \Delta \ell, u_0 - \Delta u) \in \reals_+, \label{eq:approx-def-gsepsi}\\
\overline{g}^s_{\epsilon} \colon \reals^2 \ni (\Delta \ell, \Delta u)
\mapsto g^s_{\epsilon}(\abs{\Delta \ell}, \abs{\Delta u}) \in \reals_+, \label{eq:approx-def-gsepsibar}
\end{gather}
where \(g^s_{\epsilon}\) is the \(g_{\epsilon}\) function restricted to
\([\ell_0, +\infty) \times (-\infty, u_0]\), and \(\overline{g}^s_{\epsilon}\)
extends \(g^s_{\epsilon}\) by reflecting it over the two axes.
\end{definition}
An immediate implication of the above definition is that \(\mathcal{C}_{\underline{\delta}} \subseteq \hat{\mathcal{C}}_N\).

\begin{remark}
Under Assumptions~\ref{asm:ellip-ref-and-norm}
and~\ref{asm:unimodal-ref-w-diff-density}, we see that \(g^s_{\epsilon}\)
inherits the log-concavity of \(g_{\epsilon}\), and so is
\(\overline{g}^s_{\epsilon}\). This is because \(\overline{g}^s_{\epsilon}\) is
continuous; and for any
\((\Delta \ell_1, \Delta u_1), (\Delta \ell_2, \Delta u_2) \in \reals^2\) and
their midpoint \((\Delta \ell_{1/2}, \Delta u_{1/2})\), we have
\begin{align*}
  \overline{g}^s_{\epsilon}(\Delta \ell_{1/2}, \Delta u_{1/2})
  & = g^s_{\epsilon}(\abs{\Delta \ell_{1/2}}, \abs{\Delta u_{1/2}})
    \geq g^s_{\epsilon}(m_1(\abs{\Delta \ell_1}, \abs{\Delta \ell_2}; 1/2), \abs{\Delta u_{1/2}}) \\
  & \geq g^s_{\epsilon}(m_1(\abs{\Delta \ell_1}, \abs{\Delta \ell_2}; 1/2), m_1(\abs{\Delta u_1}, \abs{\Delta u_2}; 1/2)) \\
  & = g^s_{\epsilon}(m_1((\abs{\Delta \ell_1}, \abs{\Delta u_1}), (\abs{\Delta \ell_2}, \abs{\Delta u_2}); 1/2)) \\
  & \geq m_0(g^s_{\epsilon}(\abs{\Delta \ell_1}, \abs{\Delta u_1}), g^s_{\epsilon}(\abs{\Delta \ell_2}, \abs{\Delta u_2}); 1/2)
    = m_0(\overline{g}^s_{\epsilon}(\Delta \ell_1, \Delta u_1), \overline{g}^s_{\epsilon}(\Delta \ell_2, \Delta u_2); 1/2),
\end{align*}
where \(m_0,m_1\) are defined in Definition~\ref{def:alpha-ccv}, and the first two inequalities are due to the definition of
\(g_{\epsilon}\) and the convexity of \(\abs{\cdot}\).
\end{remark}

\begin{remark}\label{rmk:gs-epsi-strictly-decreas}
\(\overline{g}^s_{\epsilon}\) is radially and strictly
decreasing on \(\overline{g}^s_{\epsilon} > 0\). To see this, it suffices to examine \(g^s_{\epsilon}\). Pick any \(t_0 > 1\) and any nonzero
\((\Delta \ell_0, \Delta u_0) \in \reals^2_+\) at which
\(g^s_{\epsilon}(\Delta \ell_0, \Delta u_0) > 0\), we compare
\(g^s_{\epsilon}(\Delta \ell_0, \Delta u_0)\) and
\(g^s_{\epsilon}(t_0 \Delta \ell_0, t_0 \Delta u_0)\) by discussing the
following two cases:
\begin{enumerate}
\item \(\Delta \ell_0 \neq 0\) and \(\Delta u_0 \neq 0\): Define
      \[
      \underline{t} := \min \set{(t_0 - 1) \min \set{\Delta \ell_0, \Delta u_0} , \sup_{t > 0} \set{\Phi(u_0 - \Delta u_0 - t) - \Phi(\ell_0 + \Delta \ell_0 + t) > (1 - \epsilon)} },
      \]
      which is strictly positive. Then,
      \begin{align*}
        & g_{\epsilon}(\ell_0 + \Delta \ell_0, u_0 - \Delta u_0) \\
        = \
        & \int\limits_0^{+\infty} \left[ \Phi(u_0 - \Delta u_0 - t) - \Phi(\ell_0 + \Delta \ell_0 + t) - (1 - \epsilon)\right]^+ \dLeb{t} \\
        = \
        & \int\limits_0^{\underline{t}} \left[ \Phi(u_0 - \Delta u_0 - t) - \Phi(\ell_0 + \Delta \ell_0 + t) - (1 - \epsilon)\right]^+ \dLeb{t} \\
        & + \int\limits_0^{+\infty} \left[ \Phi(u_0 - (\Delta u_0 + \underline{t}) - t) - \Phi(\ell_0 + (\Delta \ell_0 + \underline{t}) + t) - (1 - \epsilon)\right]^+ \dLeb{t} \\
        \geq \
        & \int\limits_0^{\underline{t}} \left[ \Phi(u_0 - \Delta u_0 - t) - \Phi(\ell_0 + \Delta \ell_0 + t) - (1 - \epsilon)\right]^+ \dLeb{t} \\
        & + \int\limits_0^{+\infty} \left[ \Phi(u_0 - t_0 \Delta u_0 - t) - \Phi(\ell_0 + t_0 \Delta \ell_0 + t) - (1 - \epsilon)\right]^+ \dLeb{t} > \ g_{\epsilon}(\ell_0 + t_0 \Delta \ell_0, u_0 - t_0 \Delta u_0),
      \end{align*}
      where the second equality is due to variable substitution, the first
      inequality is by the definition of \(\underline{t}\), and the last
      inequality is because \(\underline{t} > 0\) and the integrand of the first
      term in the summation is strictly positive on \((0, t)\).
\item One of \(\Delta \ell_0\) and \(\Delta u_0\) is \(0\): Since \(g_{\epsilon}\) is symmetric, we can assume \(\Delta \ell_0 = 0\) and
      \(\Delta u_0 \neq 0\) without loss of generality. Then, by the continuity of
      \(g_{\epsilon}\) and the argument in the previous case, we have
      \begin{align*}
        g_{\epsilon}(\ell_0, u_0 - \Delta u_0)
        = \lim_{n \to \infty} g_{\epsilon}(\ell_0 + \frac{1}{n}, u_0 - \Delta u_0)
        > \lim_{n \to \infty} g_{\epsilon}(\ell_0 + \frac{t_0}{n}, u_0 - t_0 \Delta u_0)
        = g_{\epsilon}(\ell_0, u_0 - t_0 \Delta u_0).
      \end{align*}
\end{enumerate}
\end{remark}

The next proposition relates \(\overline{g}_{\epsilon}\) with
\(\bd{}(\mathcal{C}_{\delta})\).

\begin{proposition}\label{prop:approx-cvx-g-limit-1}
Suppose that \(\epsilon \in (0, 1/2)\) and \(\delta > 0\). Then, for a sequence
of points
\(\set{(\ell_n, u_n), n \in \naturals} \subseteq \bd{}(\mathcal{C}_{\delta})\),
if \(\ell_n \searrow -\infty\) as \(n \to \infty\), then \(u_n \to u^{\ast}\) as
\(n \to \infty\), where \(u^{\ast}\) is the solution of the equation
\(\overline{g}_{\epsilon}(u) = \delta\).
\end{proposition}

\begin{proof}
Since \(\ell_n \searrow -\infty\) and \((\ell_n, u_n) \in \bd{}(\mathcal{C}_{\delta})\),
\(u_n\) is decreasing in \(n\). Consider the sequence of functions
\(\set{g_{\epsilon}^n, n \in \naturals}\), where
\begin{align*}
  g_{\epsilon}^n(u) := \int\limits_0^{+\infty} \big( \Phi(u - t) - \Phi(\ell_n + t) - (1 - \epsilon) \big)^+ \dLeb{t}.
\end{align*}
Evidently, the sequence \(\set{g_{\epsilon}^n}_{n = 1}^{\infty}\) is increasing, bounded
from above by \(\overline{g}_{\epsilon}\), and continuous for all \(n\) by the dominated
convergence theorem. Take a \(\underline{u} > 0\) such that
\(\overline{g}_{\epsilon}(\underline{u}) < \delta\) and define a restricted domain
\(\dom_g := [\underline{u}, u_1]\) for all \((g_{\epsilon}^n)\)'s and \(\overline{g}_{\epsilon}\). Since
\(g_{\epsilon}^n(\underline{u}) \leq \overline{g}_{\epsilon}(\underline{u}) < \delta\) for all \(n\)
and \(g_{\epsilon}^n(u_1) \geq g_{\epsilon}^1(u_1) = \delta\), the solution to the equations
\(\set{u \colon g_{\epsilon}^n(u) = \delta} \subseteq \dom_g\) by the intermediate value
theorem. First, we show that \(g_{\epsilon}^n \to \overline{g}_{\epsilon}\) uniformly as
\(n \to \infty\) on \(\dom_g\). Notice that
\begin{align*}
  \abs{g_{\epsilon}^n(u) - \overline{g}_{\epsilon}(u)}
  & \leq \int\limits_0^{+\infty} \Big| \big( \Phi(u - t) - \Phi(\ell_n + t) - (1 - \epsilon) \big)^+
    - \big( \Phi(u - t) - (1 - \epsilon) \big)^+ \Big| \dLeb{t} \\
  & \leq \int\limits_0^{+\infty} \Phi(\ell_n + t) \cdot \Ind{ \Phi(u - t) \geq (1 - \epsilon)} \dLeb{t}
  = \int\limits_0^{+\infty} \Phi(\ell_n + t) \cdot \Ind{ t \leq u_1 - \Phi^{-1} (1 - \epsilon)} \dLeb{t}.
\end{align*}
For any \(u \in \dom_g\), the dominated convergence theorem implies that
\begin{align*}
  \lim_{n \to \infty} \abs[\Big]{g_{\epsilon}^n(u) - \overline{g}_{\epsilon}(u)}
  \leq
  \int\limits_0^{+\infty} \lim_{n \to \infty} \Phi(\ell_n + t) \cdot \Ind{ t \leq u_1 - \Phi^{-1} (1 - \epsilon)} \dLeb{t} = 0.
\end{align*}
Due to the strict monotonicity of \(\overline{g}_{\epsilon}\) in \(u\), its inverse
function \((\overline{g}_{\epsilon})^{-1}\) is well-defined. Furthermore, it is continuous
because \(\dom_g\) is compact. Second, we bound the distance between
\(u_n\) and \(u_{\ast}\). For any \(\varepsilon > 0\), there exists an
\(N_{\varepsilon} \in \naturals\) such that
\begin{align*}
  n > N_{\varepsilon} \implies \sup_{u \in \dom_g} \abs{g^n_{\epsilon}(u) - \overline{g}_{\epsilon}(u)} < \varepsilon.
\end{align*}
Let \(u^{\ast}_n\) be the solution of \(g^n_{\epsilon}(u) = \delta\), then for all
\(n > N_{\varepsilon}\),
\begin{align*}
  u_{\ast} \leq u^{\ast}_n \leq (\overline{g}_\epsilon)^{-1}(\delta + \varepsilon),
\end{align*}
where the first inequality is because \(g^n_{\epsilon}\) is monotone and
\(g^n_{\epsilon}(u^{\ast}_n) = \delta = \overline{g}_{\epsilon}(u^{\ast}) \geq g^n_{\epsilon}(u^{\ast})\), and the
second inequality is because \((\overline{g}_\epsilon)^{-1}\) is monotone and
\(\overline{g}_\epsilon(u^{\ast}_n) \leq g^n_\epsilon(u^{\ast}_n) + \varepsilon\). We complete the proof by noting that
\begin{align*}
  \inf_{\varepsilon > 0} \sup_{n \geq N_{\varepsilon}} \abs{ u^{\ast}_n - u^{\ast}}
  \leq \inf_{\varepsilon > 0} \left( (\overline{g}_\epsilon)^{-1}(\delta + \varepsilon) - u^{\ast} \right) = 0,
\end{align*}
where the last equality is because \((\overline{g}_\epsilon)^{-1}\) is continuous.
\end{proof}

The next proposition characterizes the level sets of \(\overline{g}^s_{\epsilon}\)
through polar coordinates.

\begin{proposition}\label{prop:approx-bd-for-K1-K2}
Suppose that \(\epsilon \in (0, 1/2)\), \(\delta > 0\), \(\Prob_0\) is unimodal
with CDF \(\Phi\) and density function \(\Phi^{\prime}\), and
\(\overline{g}^s_{\epsilon}\) is defined in~\eqref{eq:approx-def-gsepsibar}. Let
\(\mathcal{K}_{\delta}\) be its \(\delta\)-superlevel set, then it holds that
\begin{align}
  \bd{}(\mathcal{K}_{\delta})
  = \Set{(\rho \cos \theta, \rho \sin \theta) \in \reals^2 \colon
  \rho \norm{(\cos \theta, \sin \theta)}_{\mathcal{C}_{\delta}} = 1}, \label{eq:approx-Kd-polar}
\end{align}
where
\(\norm{x}_{\mathcal{K}_{\delta}} = \inf \set{r > 0 \colon x \in r \cdot \mathcal{K}_{\delta}}\)
is the gauge induced by \(\mathcal{K}_{\delta}\). Furthermore, for
\(0 < \delta_1 < \delta_2 < \underline{\delta}^+\), it holds that
\begin{align}
  \mathcal{K}_{\delta_2}
  \subseteq \mathcal{K}_{\delta_1}
  \subseteq \left(1 + \frac{\sqrt{2} \ln{}(\delta_2 / \delta_1)}{\underline{D} \cdot \underline{\rho}}\right)\mathcal{K}_{\delta_2},
  \label{eq:approx-Kd1-Kd2-CKd1}
\end{align}
where \(%
\displaystyle \underline{D} = \frac{\Phi^{\prime}(u_0)}{(\Phi(u_0) - \Phi(\ell_0) - (1 - \epsilon))}\)
and
\(\displaystyle \underline{\rho} = \norm{(1, 0)}_{\mathcal{C}_{\delta_2}}^{-1}\).
\end{proposition}

\begin{proof}
Remark~\ref{rmk:gs-epsi-strictly-decreas} shows that
\(\overline{g}^s_{\epsilon}\) is strictly decreasing along each radial
direction. Then, %
\begin{align*}
  \bd{}(\mathcal{K}_{\delta})
  & = \Set{(\rho \cos \theta, \rho \sin \theta) \in \reals^2 \colon
    \rho = \sup_{u > 0} \set{u \colon (u \cos \theta, u \sin \theta) \in \mathcal{K}_{\delta}}
    }\\
  & = \Set{(\rho \cos \theta, \rho \sin \theta) \in \reals^2 \colon
    \rho = \sup_{u' > 0} \set{1/u' \colon (1/u' \cos \theta, 1/u' \sin \theta) \in \mathcal{K}_{\delta}}
    } \\
  & = \Set{(\rho \cos \theta, \rho \sin \theta) \in \reals^2 \colon
    \rho = \left( \norm{(\cos \theta, \sin \theta)}_{\mathcal{K}_{\delta}}\right)^{-1}
    },
\end{align*}
where the second equality is obtained by the change of variable
\(u^{\prime} \gets 1/u\).

The first inclusion in~\eqref{eq:approx-Kd1-Kd2-CKd1} follows from
\(\delta_1 < \delta_2\). To prove the second inclusion, we pick an arbitrary
\(\theta \in [0, \pi / 4)\) and focus on the restriction of
\(\ln \overline{g}^s_{\epsilon}(\cdot, \cdot)\)'s hypograph to direction \(\theta\):
\begin{align*}
  \mathcal{H}_{\theta} := \Set{(\rho, \delta_{\text{ln}}) \in \reals_+ \times \reals \colon
  \ln{} \overline{g}^s_{\epsilon}(\rho \cos \theta, \rho \sin \theta) \geq \delta_{\text{ln}}}.
\end{align*}
Because \(\overline{g}^s_{\epsilon}(\cdot, \cdot)\) is log-concave,
\(\ln \overline{g}^s_{\epsilon}(\cdot, \cdot)\) is a concave function, and
\(\mathcal{H}_{\theta}\) is a convex set which can be approximated from above by
hyperplanes. In particular,
\begin{align*}
  (\rho_i, \ln{}(\delta_i)) \in \bd{}(\mathcal{H}_{\theta}),
  \text{ where } \rho_i := \left( \norm{(\cos \theta, \sin \theta)}_{\mathcal{K}_{\delta_i}} \right)^{-1}, \forall i \in [2],
\end{align*}
because clearly \((\rho_i, \ln{}(\delta_i))\) and for any \(\varepsilon > 0\),
\((\rho_i + \varepsilon/2, \ln{}(\delta_i)) \not\in \mathcal{H}_{\theta}\) for \(i \in [2]\).
Therefore, the supporting hyperplane at \((\rho_2, \ln{}(\delta_2))\) is an
upper bound of
\(\ln{} \overline{g}^s_{\epsilon}(\rho \cos \theta, \rho \sin \theta)\):
\begin{align*}
  \ln{} \overline{g}^s_{\epsilon}(\rho \cos \theta, \rho \sin \theta)
  \leq
  \left.\frac{\dif{}}{\dif{} \rho}(\ln{} \overline{g}^s_{\epsilon}(\rho \cos \theta, \rho \sin \theta)) \right\vert_{\rho = \rho_2}
  (\rho - \rho_2) + \ln{}(\delta_2) =: \hat{g}^s_{\epsilon, \theta, \delta_2}(\rho),
\end{align*}
and the superlevel sets of \(\hat{g}^s_{\epsilon, \theta, \delta_2}\) are
supersets of those of
\(\ln \overline{g}^s_{\epsilon}(\rho \cos \theta, \rho \sin \theta)\). In particular,
\begin{align*}
  [0, \rho_1] \subseteq [0, \overline{\rho}_1],
  \text{ where } \overline{\rho}_1 \in \reals_+ \text{ is such that }
  \hat{g}^s_{\epsilon, \theta, \delta_2}(\overline{\rho}_1) = \ln{}(\delta_1).
\end{align*}
Solving for \(\overline{\rho}_1\), we obtain that for any \(\theta \in [0, 2\pi)\),
\begin{align*}
  \overline{\rho}_1
      = \left( 1 + \frac{\ln{} (\delta_2 / \delta_1)}{
      -\left. \frac{\dif{}}{\dif{} \rho} (\ln{} \overline{g}^s_{\epsilon}(
      \rho \cos \theta, \rho \sin \theta)) \right\vert_{\rho = \rho_2} \cdot \rho_2} \right) \rho_2.
\end{align*}
Next, we seek an upper bound of \(\overline{\rho}_1\) that is independent
from \(\theta\). To this end, we analyze the derivative in the denominator:
\begin{gather*}
  \frac{\dif{}}{\dif{} \rho} \ln{} \overline{g}^s_{\epsilon}(\rho \cos \theta, \rho \sin \theta)
  = \frac{1}{\overline{g}^s_{\epsilon}(\rho \cos \theta, \rho \sin \theta)} \cdot \frac{\dif{}}{\dif{} \rho} \overline{g}^s_{\epsilon}(\rho \cos \theta, \rho \sin \theta)
\end{gather*}
where
\begin{align*}
  \frac{\dif{}}{\dif{} \rho} \overline{g}^s_{\epsilon}(\rho \cos \theta, \rho \sin \theta)
  & = \frac{\dif{}}{\dif{} \rho} \int\limits_0^{+\infty} \left[
    \Phi(u_0 - \rho \cos \theta - t) - \Phi(\ell_0 + \rho \sin \theta + t) - (1 - \epsilon)\right]^+ \dLeb{t} \\
  & = \int\limits_0^{\overline{t}_{\rho,\theta}} \frac{\dif{}}{\dif{} \rho} \left( \Phi(u_0 - \rho \cos \theta - t) - \Phi(\ell_0 + \rho \sin \theta + t) \right) \dLeb{t} \\
  & = \int\limits_0^{\overline{t}_{\rho,\theta}} \left(
    \Phi^{\prime}(u_0 - \rho \cos \theta - t) (- \cos \theta) - \Phi^{\prime}(\ell_0 + \rho \sin \theta + t) \sin \theta \right)\dLeb{t} \\
  & \leq \int\limits_0^{\overline{t}_{\rho,\theta}} \left( -\Phi^{\prime}(u_0) (\cos \theta + \sin \theta) \right) \dLeb{t}
    = \int\limits_0^{\overline{t}_{\rho,\theta}} - \Phi^{\prime}(u_0) \sqrt{2} \sin (\theta + \frac{\pi}{4}) \dLeb{t} \\
  & \leq
    - \Phi^{\prime}(u_0) \overline{t}_{\rho, \theta},
\end{align*}
where the second equality is by Leibniz integration rule, and
\(\overline{t}_{\rho,\theta}\) is defined through
\[
  \overline{t}_{\rho, \theta} := \max \Set{t \geq 0 \colon \Phi(u_0 - \rho \cos \theta - t) - \Phi(\ell_0 + \rho \sin \theta + t) \geq 1 - \epsilon}.
\]
We notice that \(\overline{t}_{\rho, \theta}\) satisfies
\[
  \ell_0 + \rho \cos \theta + \overline{t}_{\rho, \theta} < 0 < u_0 - \rho \cos \theta - \overline{t}_{\rho, \theta}.
\]
The first inequality is because
\(\theta \in [0, \pi/4)\) and \(\Phi^{\prime}\) is symmetric,
increasing on \((-\infty, 0]\), and decreasing on \([0, \infty)\). The
second inequality is because \(\sin(\theta + \frac{\pi}{4})\) achieves its
minimum at \(0\) on \([0, \frac{\pi}{4})\). In order to remove the dependence of
\(\overline{t}_{\rho, \theta}\) on \(\theta\), we seek a lower bound for \(\overline{t}_{\rho, \theta}\). To this
end, by the mean value theorem there exists a
\(t_0 \in [0, \overline{t}_{\rho, \theta}]\) such that
\begin{align*}
  \delta = \;
  & \int\limits_0^{\overline{t}_{\rho, \theta}} \left[ \Phi(u_0 - \rho \cos \theta - t) - \Phi(\ell_0 + \rho \sin \theta + t) - (1 - \epsilon) \right] \dLeb{t} \\
  = \;
  & \overline{t}_{\rho, \theta} \left[ \Phi(u_0 - \rho \cos \theta - t_0) -
    \Phi(\ell_0 + \rho \sin \theta + t_0) - (1 - \epsilon) \right],
\end{align*}
and by
\begin{align*}
  \Phi(u_0) - \Phi(\ell_0)
  \geq \Phi(u_0 - \rho \cos \theta - t_0) - \Phi(\ell_0 + \rho \sin \theta + t_0),
\end{align*}
we obtain a lower bound for \(\overline{t}_{\rho, \theta}\),
\begin{align*}
\underline{t}_{\delta} := \frac{\delta}{(\Phi(u_0) - \Phi(\ell_0) - (1 - \epsilon))} \leq \overline{t}_{\rho, \theta},
\end{align*}
from which
\begin{align*}
  - \left. \frac{\dif{}}{\dif{} \rho} \left( \ln{} \overline{g}^s_{\epsilon}(\rho \cos \theta, \rho \sin \theta) \right)\right\vert_{\rho = \rho_2}
  & = \frac{-1}{\overline{g}^s_{\epsilon}(\rho_2 \cos \theta, \rho_2 \sin \theta)} \cdot \frac{\dif{}}{\dif{} \rho} \overline{g}^s_{\epsilon}(\rho \cos \theta, \rho \sin \theta) \\
  & \geq \frac{\Phi^{\prime}(u_0) \overline{t}_{\rho, \theta}}{\overline{g}^s_{\epsilon}(\rho_2 \cos \theta, \rho_2 \sin \theta)}
  \geq \frac{\Phi^{\prime}(u_0) \underline{t}_{\delta_2}}{\overline{g}^s_{\epsilon}(\rho_2 \cos \theta, \rho_2 \sin \theta)} \\
  & = \frac{\Phi^{\prime}(u_0)}{(\Phi(u_0) - \Phi(\ell_0) - (1 - \epsilon))} = \underline{D}.
\end{align*}

Thus, for any \((\rho \cos \theta, \rho \sin \theta) \in \mathcal{K}_{\delta_1}\), we have
\begin{align*}
  0 \leq \rho \leq \left( \norm{(\cos \theta, \sin \theta)}_{\mathcal{K}_{\delta_1}} \right)^{-1}
  = \rho_1 \leq \overline{\rho}_1
  & =
  \left( 1 +
  \frac{\ln{} (\delta_2 / \delta_1)}{
    -\left. \frac{\dif{}}{\dif{} \rho} (\ln{} \overline{g}^s_{\rho}(
    \rho \cos \theta, \rho \sin \theta)) \right\vert_{\rho = \rho_2} \cdot \rho_2} \right) \rho_2 \\
  & \leq \left( 1 + \frac{\ln{} (\delta_2 / \delta_1)}{
    \underline{D} \cdot \rho_2} \right) \rho_2
    \leq \left( 1 + \frac{\sqrt{2} \ln{} (\delta_2 / \delta_1)}{
    \underline{D} \cdot \underline{\rho}} \right) \rho_2,
\end{align*}
where the second inequality is by definition of \(\underline{D}\). We
justify the last inequality as follows. Define
\(\displaystyle k := \frac{\underline{\rho}}{\rho_2 (\cos \theta + \sin \theta)} \geq 0\). Since \(\underline{\rho}(1, 0) \in \bd{} (\mathcal{K}_{\delta_2})\), we have
\begin{gather*}
  \delta_2 = \overline{g}^s_{\epsilon}((\underline{\rho}, 0)) = g_{\epsilon}(\ell_0 + \underline{\rho}, u_0 - 0)
  = g_{\epsilon}(-u_0, -\ell_0 - \underline{\rho}) = g_{\epsilon}(\ell_0, u_0 - \underline{\rho})
  = \overline{g}^s_{\epsilon}(0, \underline{\rho}) \implies \underline{\rho}(0, 1) \in \bd{}(\mathcal{K}_{\delta_2})\\
k \cdot \rho_2 (\cos \theta, \sin \theta)
= \frac{\underline{\rho} (\cos \theta, \sin \theta)}{(\cos \theta + \sin \theta)}
= \frac{\cos \theta}{\cos \theta + \sin \theta} \cdot \underline{\rho} (1, 0)
+ \frac{\sin \theta}{\cos \theta + \sin \theta} \cdot \underline{\rho} (0, 1) \implies k \leq 1,
\end{gather*}
where the second implication is because
\(k \cdot \rho_2(\cos \theta, \sin \theta)\) is the convex combination of
\(\underline{\rho}(1, 0)\) and \(\underline{\rho}(0, 1)\) and is parallel to
\(\rho_2 (\cos \theta, \sin \theta)\). Therefore, we have
\[
\rho_2 \geq k \cdot \rho_2 = \norm[\Bigg]{%
  \frac{\underline{\rho}(\cos \theta, \sin \theta)}{\sqrt{2} \sin (\theta + \frac{\pi}{4})}}_2
  \geq \frac{\underline{\rho}}{\sqrt{2}}.
\]
Since the upper bound above is independent from \(\theta\), it holds that
\begin{align*}
  \frac{\rho}{\left( 1 + \frac{\ln{} \delta_2 / \delta_1}{ \underline{D} \cdot \underline{\rho}} \right)} \leq \rho_2
  = \sup_{u > 0} \set{u \colon (u \cos \theta, u \sin \theta) \in \mathcal{K}_{\delta_2}}
  \iff \frac{1}{\left(  1 + \frac{\ln{} (\delta_2 / \delta_1)}{\underline{D} \cdot \underline{\rho}} \right)}(
  \rho \cos \theta, \rho \sin \theta) \in \mathcal{K}_{\delta_2},
\end{align*}
and~\eqref{eq:approx-Kd1-Kd2-CKd1} follows.
\end{proof}

\begin{proposition}\label{prop:approx-guarantee-two-levelsets}
Suppose that \(\epsilon \in (0, 1/2)\) and \(\Prob_0\) is unimodal
with CDF \(\Phi\) and density function \(\Phi^{\prime}\).
For any \(\delta_1, \delta_2\) such that
\(0 < \delta_1 < \delta_2 < \underline{\delta}^+\), denote by
\(g_{\epsilon}|_{\ell_0}(u) \) the restriction of the function \(g_{\epsilon}\) to the
vertical line \(\ell = \ell_0\), and by \(g^{-1}_{\epsilon}|_{\ell_0}(\cdot)\)
its inverse (whose existence is guaranteed by
Remark~\ref{rmk:gs-epsi-strictly-decreas}):
\begin{gather*}
g_{\epsilon}|_{\ell_0}(u) := g_{\epsilon}(\ell_0, u),
\quad
g^{-1}_{\epsilon}|_{\ell_0}(\delta) := \inf \set{u > 0 \colon g_{\epsilon}(\ell_0, u) \geq \delta}.
\end{gather*}
In addition, define
\(\gamma_{13}, \gamma_4\) as
\begin{gather*}
\gamma_{13} := \frac{u_0 - (\overline{g}_{\epsilon})^{-1}(\delta_1)}{%
  u_0 - g^{-1}_{\epsilon}|_{\ell_0}(\delta_2)},
\quad
\gamma_4 := 1 + \frac{\sqrt{2} \ln{}(\delta_2 / \delta_1)}{\underline{D} \cdot \underline{\rho}},
\end{gather*}
where \(\underline{D}\) and \(\underline{\rho}\) are defined in Proposition~\ref{prop:approx-bd-for-K1-K2}. Then, it holds that
\begin{enumerate}[label=(\arabic*)]
\item\label{item:approx-orthant-2} \(\mathcal{C}_{\delta_1} \cap \mathcal{O}_2 = \mathcal{C}_{\delta_2} \cap \mathcal{O}_2\).
\item\label{item:approx-orthant-13} \(%
\mathcal{C}_{\delta_1} \cap \mathcal{O}_i %
\subseteq \gamma_{13} \cdot (\mathcal{C}_{\delta_2} \cap \mathcal{O}_i) \text{ for any } i \in \set{1, 3}\).
\item\label{item:approx-orthant-4} \(\mathcal{C}_{\delta_1} \cap \mathcal{O}_4 %
\subseteq \gamma_4 \cdot \mathcal{C}_{\delta_2} \cap \mathcal{O}_4\).
\end{enumerate}
\end{proposition}

\begin{proof}
Because
\(g_{\epsilon}(\ell_0, u_0) = \underline{\delta}^+ > \delta_2 > \delta_1 > 0\),
we have \((\ell_0, u_0) \in \mathcal{C}_{\delta_2}\), from which
\(\mathcal{O}_2 \cap \mathcal{C}_{\delta_2} = \mathcal{O}_2 = \mathcal{O}_2 \cap \mathcal{C}_{\delta_1}\),
\ie{}, item~\ref{item:approx-orthant-2} holds.

Because \(\mathcal{C}_{\delta}\) is symmetric with respect to \(u + \ell = 0\),
it is sufficient to prove the case \(i = 3\) in
item~\ref{item:approx-orthant-13}. To this end, we notice that
\((\ell_0, g^{-1}_{\epsilon}|_{\ell_0}(\delta_2)) \in \mathcal{C}_{\delta_2}\),
and so
\((-\infty, \ell_0] \times [g^{-1}_{\epsilon}|_{\ell_0}(\delta_2), u_0] \subseteq \mathcal{C}_{\delta_2} \cap \mathcal{O}_3\)
because \((-1, 0)\) and \((0, 1)\) are extreme rays of \(\mathcal{C}_{\delta}\)
for all \(\delta > 0\). In
addition, %
\begin{align*}
  \inf_{\ell, u} \Set{u \colon (\ell, u) \in \mathcal{C}_{\delta_1} \cap \mathcal{O}_3}
  & = \inf_{\ell, u} \Set{u \colon g_{\epsilon}(\ell, u) \geq \delta_1, (\ell, u) \in (-\infty, \ell_0] \times (0, u_0]} \\
  & = \inf_{\ell} \Set{ \inf_{0 < u \leq u_0} \Set{u \colon g_{\epsilon}(\ell, u) \geq \delta_1} \colon \ell \in (-\infty, \ell_0] } \\
  & = \lim_{n \to \infty} \inf_{0 < u \leq u_0} \Set{u \colon g_{\epsilon}(\ell_0 - n, u) \geq \delta_1} \\
  & = \inf_u \Set{u \colon \overline{g}_{\epsilon}(u) \geq \delta_1} = (\overline{g}_{\epsilon})^{-1}(\delta_1).
\end{align*}
where the third equality is because the sequence of functions
\(\Set{g_{\epsilon}(\ell_0 - n, \cdot )}_{n = 1}^{+\infty}\) is increasing in \(n\), \ie{} %
\begin{align*}
  \cdots
  \leq g_{\epsilon}(\ell_0 - n, \cdot)
  \leq g_{\epsilon}(\ell_0 - (n - 1), \cdot)
  \leq \cdots
  \leq g_{\epsilon}(\ell_0 - 1, \cdot)
  \leq g_{\epsilon}(\ell_0, \cdot)
\end{align*}
and it follows that
\(\Set{\inf_{0 < u \leq u_0} \Set{u \colon g_{\epsilon}(\ell_0 - n, u) \geq \delta}}_{n=1}^{+\infty}\)
is a decreasing sequence. The fourth equality is because of
Proposition~\ref{prop:approx-cvx-g-limit-1}. Now, item~\ref{item:approx-orthant-13} follows from
\begin{align*}
  \gamma_{13} \cdot (\mathcal{C}_{\delta_2} \cap \mathcal{O}_3)
  & \supseteq \gamma_{13} \cdot (-\infty, \ell_0] \times [g^{-1}_{\epsilon}|_{\ell_0}(\delta_2), u_0] \\
  & = (-\infty, \ell_0] \times [u_0 - \gamma_{13} \cdot (u_0 - g^{-1}_{\epsilon}|_{\ell_0}(\delta_2)), u_0] \\
  & = (-\infty, \ell_0] \times [(\overline{g}_{\epsilon})^{-1}(\delta_1), u_0] \supseteq \mathcal{C}_{\delta_1} \cap \mathcal{O}_3.
\end{align*}

Finally, since
\(\mathcal{C}_{\delta_i} \cap \mathcal{O}_4 = \mathcal{K}_{\delta_i} \cap \mathcal{O}_4\)
for \(i \in \set{1,2}\), item~\ref{item:approx-orthant-4} follows from
Proposition~\ref{prop:approx-bd-for-K1-K2}.
\end{proof}

Now we are ready to prove Theorem~\ref{cor:approx-guarantee-for-hierarchy}.

\begin{proof}[Proof of Theorem~\ref{cor:approx-guarantee-for-hierarchy}]
Similar to the proof of Proposition~\ref{prop:approx-guarantee-two-levelsets},
we partition \(\hat{\mathcal{C}}_N\) and \(\mathcal{C}_{\delta}\) into four subsets to obtain:
\begin{enumerate}[label=(\arabic*)]
\item\label{item:approx-orthant-2} \(%
\mathcal{C}_{\delta} \cap \mathcal{O}_2 = \hat{\mathcal{C}}_N \cap \mathcal{O}_2\).
\item\label{item:approx-orthant-13} \(%
\mathcal{C}_{\delta} \cap \mathcal{O}_i %
\subseteq \frac{u_0 - (\overline{g}_{\epsilon})^{-1}(\delta)}{u_0 - (\overline{g}_{\epsilon})^{-1}(\delta + \tau)} %
\cdot (\hat{\mathcal{C}}_N \cap \mathcal{O}_i), \forall i \in \set{1, 3}\).
\item\label{item:approx-orthant-4} \(\mathcal{C}_{\delta} \cap \mathcal{O}_4 %
\subseteq \left( 1 + \frac{\sqrt{2}\ln{}((\delta + \tau) / \delta)}{\underline{D} \cdot \underline{\rho}} \right) \cdot (\hat{\mathcal{C}}_N \cap \mathcal{O}_4)\).
\end{enumerate}
Then, the claim follows by taking the maximum between the two approximation error coefficients for \(i = 3, 4\). The asymptotic exactness holds by observing that both coefficients reduce to \(1\) as \(\tau\) tends to zero.
\end{proof}

\subsection{Proof of Theorem~\ref{thm:opt-det-reform}} \label{apx-thm:opt-det-reform}
\begin{proof}
By Theorem~\(1\) in~\cite{gao-2016-distr-robus}, (\ODRCCL{}) is equivalent to inequality
\begin{align*}
  \sup_{\QProb \in \mathcal{P}} \QProb [A(x)\xi \leq b(x)] \equiv \min_{\lambda \geq 0} \Set{
  \lambda \delta - \Expect_{\Prob} \left[
  \inf_{\xi \in \Xi} \Set{ \lambda \norm{\zeta - {\xi}} - \Ind{ \xi \in \Safe(x)}} \right]}
  \geq 1 - \epsilon.
\end{align*}
Noting that for any fixed \(x \in \reals^n\) and \(\zeta \in \Xi\)
\begin{align*}
\inf_{\xi \in \Xi} \Set{ \lambda \norm{\zeta - \xi} - \Ind{ \xi \in \Safe(x)}}
& = \left\{\begin{array}{ll}
-1     & \text{if \(\zeta \in \Safe(x)\)} \\[0.5em]
\min \Set{ \lambda \cdot \Dist{\zeta, \Safe(x)} - 1, 0}     & \text{if \(\zeta \notin \Safe(x)\)}
\end{array}\right. \\[0.5em]
& = \min \Set{ \lambda \cdot \Dist{\zeta, \Safe(x)} - 1, 0},
\end{align*}
we recast \(\mathcal{X}^o\) as
\begin{align*}
  \lambda \delta + \Expect_{\Prob} \Big[
  \max \Set{ 1 - \lambda \cdot \Dist{{\zeta}, \Safe(x)} , 0} \Big]
  \geq 1 - \epsilon \qquad \forall \lambda \geq 0.
\end{align*}
We notice that the above inequality automatically holds when \(\lambda = 0\) because, in this case, the LHS equals one. Hence, we can drop this case and assume that \(\lambda > 0\). Then, we divide both sides by \(\lambda\) and denote \(\gamma = 1/\lambda\) to obtain
\begin{align*}
  \delta + \Expect_{\Prob} \left[ \left( \gamma - \Dist{{\zeta}, \Safe(x)}, 0 \right)^+  \right]
  \geq (1 - \epsilon) \gamma \qquad \forall \gamma \geq 0.
\end{align*}
We notice that the above inequality holds for all  \(\gamma < 0\) because, in that case, the LHS is positive and the RHS is negative. Hence, we expand the domain of \(\gamma\) to be the whole real line and finish the proof as follows:
\begin{align*}
  & (- \gamma) + \frac{1}{1 - \epsilon} \Expect_{\Prob} \left[
  \left( - \Dist{{\zeta}, \Safe(x)} - (- \gamma), 0 \right)^+  \right]
  \geq - \frac{\delta}{1 - \epsilon} \qquad \forall \gamma \in \reals \\
  \iff ~
  & \inf_{-\gamma \in \reals} \Set{
    (-\gamma) + \frac{1}{1 - \epsilon} \Expect_{\Prob} \left[
    \left( - \Dist{{\zeta}, \Safe(x)} - (- \gamma), 0 \right)^+  \right]}
    \geq \frac{-\delta}{1 - \epsilon} \\
  \iff ~
  & \CVaR_{\epsilon} \Big( - \Dist{{\zeta}, \Safe(x)} \Big) + \frac{\delta}{1 - \epsilon} \geq 0.
\end{align*}
\end{proof}

\subsection{Proof of Lemma~\ref{lem:dist-convex}} \label{apx-lem:dist-convex}
\begin{proof}
Since \(\Dist{\zeta, \Safe(x)}\) is defined through a convex program, in which the Slater's condition holds, we take the dual to obtain
\begin{align*}
  \Dist{\zeta, \Safe(x)}
  = \max_{\lambda \leq 0}
  & \Set{\lambda^{\top} [b(x) - A \zeta] \colon \norm{A^{\top} \lambda}_{\ast} \leq 1}.
\end{align*}
This completes the proof.

\end{proof}

\subsection{Proof of Theorem~\ref{thm:opt-rhs-cvx}} \label{apx-thm:opt-rhs-cvx}
\begin{proof}
By Theorem~\ref{thm:opt-det-reform}, (\ODRCCL{}) admits the following reformulations:
\begin{align*}
  & \CVaR_{\epsilon} \Set{ -\Dist{{\zeta}, \Safe(x)}} \geq - \frac{\delta}{1 - \epsilon} \\
  \iff
  & \inf_{\gamma \in \reals} \Set{
    \gamma + \frac{1}{1 - \epsilon} \Expect_{\Prob} \Set{ \Big[ - \Dist{{\zeta}, \Safe(x)} - \gamma \Big]^+}}
    \geq - \frac{\delta}{1 - \epsilon} \\
  \iff
  & \gamma + \frac{1}{1 - \epsilon} \Expect_{\Prob} \Set{ \Big[ - \Dist{{\zeta}, \Safe(x)} - \gamma \Big]^+}
    \geq - \frac{\delta}{1 - \epsilon} \qquad \forall \gamma \in \reals.
\end{align*}
In what follows, we prove that the LHS of the last reformulation is log-concave in \(x\) for any fixed \(\gamma\). Since log-concave functions are quasi-concave and continuous (see Lemma 2.4 in~\cite{norkin-1993-analy-optim}), the convexity and closedness of \(\mathcal{X}^o_{\text{R}}\) follows from their preservation under intersection. To this end, we notice that
\[
\Expect_{\Prob} \Big[ \phi({\zeta}, x) \Big]
= \int_{\Xi} \phi(x, \zeta) \cdot f_{\zeta} (\zeta) \dLeb{\zeta},
\]
where \(\phi(x, \zeta) := \big[ - \Dist{\zeta, \Safe(x)} - \gamma \big]^+\) and \(f_{\zeta}\) represents the probability density function of \(\zeta\). It suffices to show that \(\phi(x, \zeta) \cdot f_{\zeta} (\zeta)\) is jointly log-concave in \((x, \zeta)\) because log-concavity preserves under marginalization (see Theorem 3.3 in~\cite{saumard-2014-log-concav}). In view that log-concavity also preserves under multiplication, we complete the proof by showing that \(f_{\zeta} (\zeta)\) is log-concave in \(\zeta\) and \(\phi(x, \zeta)\) is jointly log-concave in \((x, \zeta)\).
\begin{enumerate}
\item Since \(\Prob\) is \(\alpha\)-concave, its density function \(f_{\zeta}\) is \(\alpha'\)-concave by Proposition~\ref{prop:alpha-prob-meas-and-pdf}, where
\begin{align*}
  \alpha^{\prime} =
  \begin{cases}
  \frac{\alpha}{1 - m \alpha} & \text{if } \alpha \in [0, 1/m) \\
  +\infty & \text{if } \alpha = 1 / m
  \end{cases}
\end{align*}
and \(\alpha' \geq 0\). Hence, \(f_{\zeta}\) is log-concave by Lemma~\ref{lem:monotone-m-alpha}.
\item For any pair of \((x_1, \zeta_1), (x_2, \zeta_2) \in \reals^n \times \Xi\) and any \(\theta \in [0, 1]\), define \((x_{\theta}, \zeta_{\theta}) := \theta (x_1, \zeta_1) + (1 -
\theta)(x_2, \zeta_2)\). Then, it holds that
\begin{align*}
  \phi(x_{\theta}, \zeta_{\theta})
  = \Big( - \Dist{\zeta_{\theta}, \Safe(x_{\theta})} - \gamma \Big)^+
  & \geq \Big( m_1 \big(- \Dist{\zeta_1, \Safe(x_1)} - \gamma, - \Dist{\zeta_2, \Safe(x_2)} - \gamma; \theta \big) \Big)^+ \\
  & \geq m_0 \Big(\phi(x_1, \zeta_1), \phi(x_2, \zeta_2); \theta \Big),
\end{align*}
where the first inequality is because \(\Dist{\zeta, \Safe(x)}\) is jointly convex in \((x, \zeta)\). To see the second inequality, we discuss the following two cases.
\begin{enumerate}[label=(\roman*)]
\item If either \(\phi(x_1, \zeta_1)\) or \(\phi(x_2, \zeta_2)\) equals zero, then \(m_0 \big(\phi(x_1, \zeta_1), \phi(x_2, \zeta_2); \theta \big)\) equals zero by definition.
\item If both \(\phi(x_1, \zeta_1)\) and \(\phi(x_2, \zeta_2)\) are strictly positive, then
\begin{align*}
  \Big( m_1 \big(- \Dist{\zeta_1, \Safe(x_1)} - \gamma, - \Dist{\zeta_2, \Safe(x_2)} - \gamma; \theta \big) \Big)^+
= \ & m_1 \Big(\phi(x_1, \zeta_1), \phi(x_2, \zeta_2); \theta \Big) \\
\geq \ & m_0 \Big(\phi(x_1, \zeta_1), \phi(x_2, \zeta_2); \theta \Big),
\end{align*}
where the inequality follows from Lemma~\ref{lem:monotone-m-alpha}.
\end{enumerate}
\end{enumerate}
\end{proof}

\subsection{A Generalized Theorem~\ref{thm:opt-rhs-cvx} For Quasi-Concave Inequalities} \label{thm:opt-rhs-cvx-general}

\begin{FirstUpdate}
\begin{theorem}
Suppose that the reference distribution \(\Prob\) of \(\mathcal{P}\) is \(\alpha\)-concave with \(0 \leq \alpha \leq 1/m\). Then, the set
\begin{align*}
  \mathcal{X}^o_{\text{R}} := \Set{ x \in \reals^n \colon
  \sup_{\QProb \in \mathcal{P}} \QProb \Big[ h(x, \xi) \geq 0 \Big] \geq 1 - \epsilon}
\end{align*}
is convex and closed for \(\delta > 0\), where \(h \colon \reals^n \times \reals^m \to \reals\) is quasi-concave.
\end{theorem}
\end{FirstUpdate}

\begin{FirstUpdate}
\begin{proof}
In this proof, we show that the distance \(\Dist{\zeta, \Safe(x)}\) from \(\zeta \in \reals^m\) to the safe set \(\Safe(x)\) is jointly convex in \((\zeta, x)\) on \(\Xi \times \reals^n\). Then, the conclusion follows from the proof of Theorem~\ref{thm:opt-rhs-cvx}.

To show the convexity of \(\Dist{\zeta, \Safe(x)}\), we recall that \(h \colon \reals^n \times \Xi \to \reals\) is quasi-concave. Then, the superlevel set
\(\mathcal{H}_{\geq 0} := \Set{(x, \xi) \colon h(x, \xi) \geq 0}\) is convex. In addition,
\begin{align*}
  \Dist{\zeta, \Safe(x)}
  & = \inf_{\xi \in \Xi} \Set{\norm{\xi - \zeta} \colon (x, \xi) \in \mathcal{H}_{\geq 0}} \\
  & = \min_{\xi \in \Xi} \Set{\norm{\xi - \zeta} \colon (x, \xi) \in \cl{}\left(\mathcal{H}_{\geq 0}\right)},
\end{align*}
where the second equality is because \(\norm{\cdot}\) is continuous. Take
\((x_1, \zeta_1), (x_2, \zeta_2) \in \reals^n \times \Xi\), then there exist
two minimizers \(\xi_1, \xi_2 \in \Xi\) such that they are the closest points in
\(\cl{}\left( \mathcal{H}_{\geq 0} \right)\) to \((x_1, \zeta_1)\) and
\((x_2, \zeta_2)\), respectively. It follows that, for \(\lambda \in (0, 1)\) and
\((x_{\lambda}, \zeta_{\lambda}) := \lambda (x_1, \zeta_1) + (1 - \lambda) (x_2, \zeta_2)\),
\begin{align*}
  \Dist{\zeta_{\lambda}, \Safe(x_{\lambda})}
  & = \min_{\xi \in \Xi} \Set{\norm{\xi - \zeta_{\lambda}} \colon (x_{\lambda}, \xi) \in \cl{}\left(\mathcal{H}_{\geq 0}\right)} \\
  & \leq \norm{\lambda \xi_1 + (1 - \lambda) \xi_2 - \zeta_{\lambda}}
    = \norm{\lambda (\xi_1 - \zeta_1) + (1 - \lambda) (\xi_2 - \zeta_2)} \\
  & \leq \lambda \Dist{\zeta_1, \Safe(x_1)} + (1 - \lambda) \Dist{\zeta_2, \Safe(x_2)},
\end{align*}
where the first inequality is because
\((x_{\lambda}, \lambda \xi_1 + (1 - \lambda) \xi_2) \in \cl{} \left( \mathcal{H}_{\geq 0} \right)\).
\end{proof}
\end{FirstUpdate}

\subsection{Proof of Lemma~\ref{lem:optimistic-ts-rel-T-T0}} \label{apx-lem:optimistic-ts-rel-T-T0}
\begin{proof}
Theorem~\ref{thm:opt-det-reform} yields
\begin{align*}
  \mathcal{X}^o_{\text{T}}
  = \Set{x \in \reals^n \colon
  \CVaR_{\epsilon}\left( - \Dist{\xi, \mathcal{S}(x)} \right) + \frac{\delta}{1 - \epsilon} \geq 0},
\end{align*}
where the distance to the safe set \(\mathcal{S}(x)\) is
\begin{gather*}
  \Dist{\xi, \mathcal{S}(x)}
  = \inf_{\eta} \Set{\norm{\eta - \xi} \colon \ell \leq x^{\top} \eta \leq u}
  = \frac{1}{\norm{x}_{\ast}} \left[ f(\ell, u, \xi) \right]^+
\end{gather*}
and \(f(\ell, u, \xi) := (x^{\top} \xi - u) \vee (\ell - x^{\top}\xi)\). Similarly, we recast \(\mathcal{X}^o_{\text{T}_0}\) as
\begin{align*}
  \mathcal{X}^o_{\text{T}_0}
  = \Set{(\ell ,u) \in \reals_- \times \reals_+ \colon
  \CVaR_{\epsilon}(-(f_0(\ell, u, \zeta))^+) + \frac{\delta}{1 - \epsilon} \geq 0
  },
\end{align*}
where \(f_0(\ell, u, \zeta) := (\zeta - u) \vee (\ell - \zeta)\) and the \(\CVaR{}\) is with respect to \(\Prob_0\). But for \(\zeta \sim \Prob_0\) and \(\xi \sim \Prob\), we have \(\zeta \disteq x^{\top}\xi/\|x\|_*\) and so
\[
\frac{1}{\norm{x}_{\ast}} \left[ f(\ell, u, \xi) \right]^+ = \max\left\{\frac{x^{\top}\xi - u}{\|x\|_*}, \frac{\ell- x^{\top}\xi}{\|x\|_*}, 0\right\} \disteq \max\left\{\zeta - \frac{u}{\|x\|_*}, \frac{\ell}{\|x\|_*} - \zeta, 0\right\} = (f_0(\ell, u, \zeta))^+.
\]
The conclusion follows.
\end{proof}

\subsection{Proof of Theorem~\ref{thm:ts-occ-convexity}} \label{apx-thm:ts-occ-convexity}
\begin{proof}
We first show the reformulation of \(\mathcal{X}^o_{\text{T}_0}\). By
Theorem~\ref{thm:opt-det-reform} and Lemma~\ref{lem:reform-neg-cvar}, we recast
\(\mathcal{X}^o_{\text{T}_0}\) as
\begin{align*}
  \frac{\delta}{1 - \epsilon} + \Ind{
  0 \geq \VaR_{\epsilon}(-f_0(\ell, u, \zeta))} \cdot
  \left( \CVaR_{\epsilon}(-f_0(\ell, u, \zeta)) - \frac{1}{1 - \epsilon} \Expect \left[ (-f_0(\ell, u, \zeta) )^+ \right]\right) \geq 0,
\end{align*}
where \(f_0(\ell, u, \zeta) \equiv (\zeta - u) \vee (\ell - \zeta)\) and \(\zeta \sim \Prob_0\). Then, we break down the indicator function to obtain
\begin{gather*}
\mathcal{X}^o_{\text{T}_0} = \mathcal{X}^o_{\text{T}_1} \cup \big(
(\reals_- \times \reals_+ \setminus \mathcal{X}^o_{\text{T}_1}) \cap \mathcal{X}^o_{\text{T}_2}\big), \\
  \text{ where } \quad
  \mathcal{X}^o_{\text{T}_1} = \Set{(\ell, u) \in \reals_- \times \reals_+ \colon
    0 < \VaR_{\epsilon}(- f_0(\ell, u, \zeta))}, \\
  \text{ and } \quad
  \mathcal{X}^o_{\text{T}_2} = \Set{(\ell, u) \in \reals_- \times \reals_+ \colon
      \frac{\delta}{1 - \epsilon} + \CVaR_{\epsilon}(-f_0(\ell, u, \zeta))
      \geq \frac{1}{1 - \epsilon} \Expect \left[ (-f_0(\ell, u, \zeta))^+ \right]}.
\end{gather*}
For \(\mathcal{X}^o_{\text{T}_1}\), we have
\begin{align*}
  0 < \VaR_{\epsilon}(-f_0(\ell, u, \zeta))
  & \iff  \Prob_0 \left[ - f_0(\ell, u, \zeta) \leq 0 \right] < \epsilon \\
  & \iff  \Prob_0 \left[ \zeta \geq u \text{ or } \ell \geq \zeta \right] < \epsilon \\
  & \iff  \Prob_0 \left[ \ell \leq \zeta \leq u \right] > 1 - \epsilon, \\
  \text{or equivalently:} \quad 0 \geq \VaR_{\epsilon}(-f_0(\ell, u, \zeta))
  & \iff \Prob_0 \left[ \ell \leq \zeta \leq u \right] \leq (1 - \epsilon).
\end{align*}
For \(((\reals_- \times \reals_+) \setminus \mathcal{X}^o_{T_1}) \cap \mathcal{X}^o_{\text{T}_2} \), we have
\begin{align*}
  & \delta + (1 - \epsilon) \CVaR{}_{\epsilon}(-f_0(\ell,u,\zeta)) \geq \Expect \left[ (- f_0(\ell,u,\zeta))^+ \right] \\
  \iff
  & \delta + \Expect \left[ -f_0(\ell,u,\zeta) \cdot \Ind{-f_0(\ell,u,\zeta) \geq \VaR_{\epsilon}(-f_0(\ell,u,\zeta))} \right]
    \geq \Expect \left[ - f_0(\ell,u,\zeta)  \cdot \Ind{-f_0(\ell,u,\zeta) \geq 0}\right] \\
  \iff
  & \delta + \Expect \left[ -f_0(\ell,u,\zeta) \cdot \Ind{\VaR_{\epsilon}(-f_0(\ell,u,\zeta)) \leq -f_0(\ell,u,\zeta) \leq 0} \right] \geq 0 \\
  \iff
  & \delta \geq \Expect \left[ f_0(\ell,u,\zeta) \cdot \Ind{0 \leq f_0(\ell,u,\zeta) \leq \VaR_{1 - \epsilon}(f_0(\ell,u,\zeta))} \right].
\end{align*}
Plugging the definition of \(f_0(\ell,u,\zeta)\) into the~\RHS{} yields
\begin{align*}
  & \Expect{} \left[ f_0(\ell,u,\zeta) \cdot \Ind{0 \leq f_0(\ell,u,\zeta) \leq \VaR_{1 - \epsilon}(f_0(\ell,u,\zeta))}\right] \\
  = \;
  & \int\limits_{\Xi} (\zeta - u) \vee (\ell - \zeta) \cdot \Ind{%
    0 \leq (\zeta - u) \vee (\ell - \zeta) \leq \VaR_{1 - \epsilon}(f_0(\ell,u,\zeta))} \dProb{\zeta} \\
  = \;
  & \int\limits_{\Xi} \int\limits_0^{+\infty} \Ind{%
    t \leq (\zeta - u) \vee (\ell - \zeta) \leq \VaR_{1 - \epsilon}(f_0(\ell,u,\zeta))} \dLeb{t} \dProb{\zeta} \\
  = \;
  & \int\limits_0^{+\infty} \Prob \left[ t \leq (\zeta - u) \vee (\ell - \zeta) \leq \VaR_{1 - \epsilon}((\zeta - u) \vee (\ell - \zeta)) \right] \dLeb{t} \\
  = \;
  & \int\limits_0^{+\infty} \left( (1 - \epsilon) - \Prob \left[ (\zeta - u) \vee (\ell - \zeta) \leq t \right] \right)^+ \dLeb{t}
  = \int\limits_0^{+\infty} \left( (1 - \epsilon) - \Prob \left[ \ell - t \leq \zeta \leq u + t \right] \right)^+ \dLeb{t}.
\end{align*}
Therefore,
\begin{align*}
  \big((\reals_- \times \reals_+) \setminus \mathcal{X}^o_{T_1}\big) \cap \mathcal{X}^o_{\text{T}_2}
  & = \Set{(\ell, u) \in \reals_- \times \reals_+ \colon
    \begin{aligned}
      & 0 \geq \VaR_{\epsilon} \left( -f_0(\ell, u, \zeta) \right) \\
      & \delta \geq \int\limits_0^{+\infty} \left(
        (1 - \epsilon) - \Prob \left[ \ell - t \leq \zeta \leq u + t \right] \right)^+ \dLeb{t}.
    \end{aligned}}
\end{align*}
It follows that
\begin{align*}
  \mathcal{X}^o_{\text{T}_0}
  & = \mathcal{X}^o_{\text{T}_1} \cup \big(
  (\reals_- \times \reals_+ \setminus \mathcal{X}^o_{\text{T}_1}) \cap \mathcal{X}^o_{\text{T}_2}\big) \\
  & = \Set{(\ell, u) \in \reals_- \times \reals_+ \colon 0 < \VaR_{\epsilon}(-f_0(\ell, u, \zeta))} \\
  & \phantom{=} \cup \Set{(\ell, u) \in \reals_- \times \reals_+ \colon
    \begin{aligned}
      & 0 \geq \VaR_{\epsilon} \left( -f_0(\ell, u, \zeta) \right) \\
      & \delta \geq \int\limits_0^{+\infty} \left(
        (1 - \epsilon) - \Prob_0 \left[ \ell - t \leq \zeta \leq u + t \right] \right)^+ \dLeb{t}
    \end{aligned}} \\
  & = (\reals_- \times \reals_+) \cap \Set{
    (\ell, u) \in \reals_- \times \reals_+ \colon
    \begin{aligned}
    & \Prob_0 \left[ \ell \leq \zeta \leq u \right] > (1 - \epsilon), \text{ \textbf{or} } \\
    & \delta \geq \int\limits_0^{+\infty} \left(
      (1 - \epsilon) - \Prob_0 \left[ \ell - t \leq \zeta \leq u + t \right] \right)^+ \dLeb{t}
    \end{aligned}} \\
  & = \Set{
    (\ell, u) \in \reals_- \times \reals_+ \colon
    \delta \geq \int\limits_0^{+\infty} \left(
    (1 - \epsilon) - \Prob_0 \left[ \ell - t \leq \zeta \leq u + t \right] \right)^+ \dLeb{t}} \\
  & = \Set{
    (\ell, u) \in \reals_- \times \reals_+ \colon
    \delta \geq h_{\epsilon}(\ell, u)
    }
\end{align*}
where the second to the last equality is because for all
\((\ell_1, u_1) \in \reals_- \times \reals_+\) such that
\(\Prob_0 \left[ \ell_1 \leq \zeta \leq u_1 \right] \geq (1 - \epsilon)\), we
have
\begin{align*}
  (1 - \epsilon) - \Prob_0 \left[ \ell_1 - t \leq \zeta \leq u + t \right]
  \leq  (1 - \epsilon) - \Prob_0 \left[ \ell_1 \leq \zeta \leq u \right] \leq 0, \quad \forall t \geq 0,
\end{align*}
implying
\begin{align*}
  \int\limits_0^{+\infty} \left( (1 - \epsilon) - \Prob_0 \left[ %
  \ell_1 - t \leq \zeta \leq u_1 + t \right] \right)^+ \dLeb{t} = 0 \leq \delta.
\end{align*}

Second, we show that \(\mathcal{X}^o_{\text{T}_0}\) is convex. By assumption,
\(\Prob_0 \disteq{} R \cdot e^{\top}_1 U_n\) is unimodal on \(\reals\) and its
distribution function \(\Phi\) is concave on \((0, +\infty)\) and convex on
\((-\infty, 0)\). Hence, \(\Phi(u + t) - \Phi(\ell - t)\) is jointly concave in \((\ell, u, t)\) on \(\reals_- \times \reals^2_+\). It follows that
the integrand of \(h_{\epsilon}\) is jointly convex in \((\ell, u, t)\), and so \(h_{\epsilon}\) is convex in \((\ell, u) \in \reals_- \times \reals_+\)
because partial integration of a convex function preserves its convexity.

Finally, to prove that \(\mathcal{X}^o_{\text{T}}\) is convex, it remains to
show that \((x, \ell, u) \in \mathcal{X}^o_{\text{T}}\) if and only if there
exists an \(s \geq \norm{x}_{\ast}\) such that
\((\ell, u, s) \in \text{co}\left(\mathcal{X}^o_{\text{T}_0}\right)\). To this end, we discuss
the following two cases:
\begin{enumerate}
\item \(x = 0\): Suppose that \((0, \ell, u) \in \mathcal{X}^o_{\text{T}_0}\),
      then \(\ell \leq 0 \leq u\) and for \(s_n := 1/n\) we have
      \begin{align*}
        h_{\epsilon}(\ell/s_n, u/s_n) =
        \int\limits_0^{+\infty} \Big[ (1 - \epsilon) - (%
        \Phi(n \cdot u + t) - \Phi(n \cdot \ell - t)) \Big]^+ \dLeb{t} \to 0 \text{ as \(n \to \infty\)}.
      \end{align*}
      Therefore, there exists an \(n\) such that
      \((\ell, u, 1/n) \in \mathcal{X}^o_{\text{T}_0}\). On the contrary, if
      there exists an \(s > 0\) such that
      \((\ell, u, s) \in \text{co}(\mathcal{X}^o_{\text{T}_0})\), then it is the limit
      point of a sequence
      \(\set{(\ell_n, u_n, s_n)}_{n=1}^{+\infty} \subseteq \mathcal{X}^o_{\text{T}_0}\)
      satisfying \(\ell_n \leq u_n\) for all \(n\). Then, \(\ell \leq u\) as well,
      implying that \((0, \ell, u) \in \mathcal{X}^o_{\text{T}}\).
\item \(x \neq 0\): Suppose that
      \((x, \ell, u) \in \mathcal{X}^o_{\text{T}}\). Then,
      Lemma~\ref{lem:optimistic-ts-rel-T-T0} implies that
      \(\left(\frac{\ell}{\norm{x}_{\ast}}, \frac{u}{\norm{x}_{\ast}}\right) \in \mathcal{X}^o_{\text{T}_0}\),
      \ie{}, \((\ell, u, \norm{x}_{\ast}) \in \text{co}\left(\mathcal{X}^o_{\text{T}_0}\right)\).
      On the contrary, suppose that there exists an
      \(s \geq \norm{x}_{\ast} > 0\) such that
      \((\ell, u, s) \in \text{co}\left(\mathcal{X}^o_{\text{T}_0}\right)\), then it is the
      limit point of a sequence
      \(\set{(\ell_n, u_n, s_n)}_{n=1}^{+\infty} \subseteq \mathcal{X}^o_{\text{T}_0}\).
      Observe that
      \begin{align*}
        h_{\epsilon}\left( \frac{\ell}{\norm{x}_{\ast}}, \frac{u}{\norm{x}_{\ast}} \right)
        \leq h_{\epsilon}\left( \frac{\ell}{s}, \frac{u}{s} \right)
        = \lim_{n \to \infty} h_{\epsilon} \left( \frac{\ell_n}{s_n}, \frac{u_n}{s_n} \right) \leq \delta,
      \end{align*}
      where the first inequality is because \(h_{\epsilon}(\ell, u)\) is
      increasing in \(\ell\) and decreasing in \(u\), and the
      equality is due to the continuity of \(h_{\epsilon}\). Therefore, \(\left( \frac{\ell}{\norm{x}_{\ast}}, \frac{u}{\norm{x}_{\ast}} \right) \in \mathcal{X}^o_{\text{T}_0}\)
      and so \((x, \ell, u) \in \mathcal{X}^o_{\text{T}}\) by
      Lemma~\ref{lem:optimistic-ts-rel-T-T0}.
\end{enumerate}
\end{proof}

\end{appendix}

\end{document}